\documentclass[11pt,letterpaper,reqno]{amsart}

\usepackage{bgkm-paper}

\title{KPP traveling waves in the half-space}

\author{Julien Berestycki}
\address{JB: Department of Statistics and Magdalen College, University of Oxford, UK}
\email{\tt julien.berestycki@stats.ox.ac.uk}

\author{Cole Graham}
\address{CG: Division of Applied Mathematics, Brown University, 182 George St, Providence, RI 02906, USA}
\email{\tt cole\_graham@brown.edu}

\author{Yujin H. Kim}
\address{YK: Courant Institute, New York University, 251 Mercer Street, New York, NY 10012, USA}
\email{\tt yujin.kim@courant.nyu.edu}

\author{Bastien Mallein}
\address{BM: LAGA UMR 7539, Université Sorbonne Paris Nord, 99 avenue Jean-Baptiste Clément, F-93430, Villetaneuse, France}
\email{\tt mallein@math.univ-paris13.fr}

\begin{document}
\begin{abstract}
  We study traveling waves of the KPP equation in the half-space with Dirichlet boundary conditions.
  We show that minimal-speed waves are unique up to translation and rotation but faster waves are not.
  
  We represent our waves as Laplace transforms of martingales associated to branching Brownian motion in the half-plane with killing on the boundary.
  We thereby identify the waves' asymptotic behavior and uncover a novel feature of the minimal-speed wave $\Phi$.
  Far from the boundary, $\Phi$ converges to a \emph{logarithmic shift} of the 1D wave $w$ of the same speed:
  $\displaystyle \lim_{y \to \infty} \Phi\big(x + \tfrac{1}{\sqrt{2}}\log y, y\big) = w(x)$.
\end{abstract}
\maketitle

\section{Introduction}

We study the KPP equation in the Dirichlet half-space:
\begin{equation}
  \label{eq:KPP}
  \begin{cases}
    \partial_t u = \frac{1}{2} \Delta u + u - u^2 & \text{in } \H^d,\\
    u = 0 & \text{on } \partial \H^d.
  \end{cases}
\end{equation}
Here $\H^d \coloneqq \R^{d-1} \times \R_+$ and $d \geq 2$.
This reaction--diffusion equation exhibits a wealth of propagation phenomena including \emph{traveling waves}---solutions that move at constant speed parallel to the boundary.
In this paper, we exploit the close relationship between \eqref{eq:KPP} and branching Brownian motion to construct a host of traveling waves and characterize those of minimal speed.
We focus on the quadratic nonlinearity in \eqref{eq:KPP} for simplicity, but our results extend to more general equations; see Remark~\ref{rem:reaction} for details.

\subsection*{Motivation}

Reaction--diffusion equations model phenomena in fields ranging from chemistry to sociology.
They can describe the progression of a chemical reaction through a medium or a species invading new territory.
Fundamentally, reaction--diffusion equations combine growth and dispersal; together, these features generate spatial propagation.
At long times, such propagation commonly settles into a constant-speed pattern known as a traveling wave.
Rigorously, on the line, solutions of reaction--diffusion equations with localized initial data often converge to traveling waves in suitable moving frames~\cite{KPP, AW}.
In multiple dimensions, the same holds in the whole space~\cite{Ducrot} and in cylinders with compact cross-section~\cite{BN, MJM}.

The half-space is a complex intermediate---both anisotropic and transversally noncompact.
In~\cite{BeG}, H. Berestycki and the second author construct traveling waves of any speed $c \geq \sqrt{2}$ in the half-space and show that localized disturbances roughly propagate at speed $\sqrt{2}$.
Two major questions remain: are the traveling waves unique up to translation, and do parabolic solutions converge to such waves in a suitable frame?
Here, we address the first question; we consider the second in a forthcoming work.

In one dimension, traveling waves of a given speed are unique up to translation.
In multiple dimensions, however, traveling waves in the whole space with supercritical speeds $c > \sqrt{2}$ are \emph{not} unique.
This multiplicity is due to waves with level sets oblique to the direction of propagation~\cite{HMR}.
In contrast, the minimal speed $\sqrt{2}$ does not support oblique level sets.
Minimal-speed waves are planar and unique up to translation and rotation~\cite{HN}.
We show that the half-space exhibits similar behavior.
Up to isometry, the half-space supports a single minimal-speed wave but many supercritical waves.
To our knowledge, this is the first proof of uniqueness for traveling waves with nontrivial and noncompact transverse structure.

We are further motivated by the remarkable relationship between the PDE \eqref{eq:KPP} and the stochastic branching particle system known as branching Brownian motion (BBM).
Precisely, solutions of \eqref{eq:KPP} constitute the Laplace transform of BBM.
First observed by McKean \cite{McKean75}, this relationship has long been used to study both \eqref{eq:KPP} and BBM.
For example, one-dimensional traveling waves can be expressed as Laplace transforms of martingales associated to BBM~\cite{LS87,HHK06,Kyprianou04,Harris99}.

Here, we develop this theory in the half-space.
We express our traveling waves on $\H^d$ as Laplace transforms of certain martingales associated to BBM in $\H^d$.
Using this representation, we determine the large-scale structure of said waves and uncover unexpected asymptotic phenomena in the minimal-speed setting.
Our approach interweaves analytic and probabilistic arguments in novel fashion.
This is not a mere convenience---we are presently unable to prove the full complement of our results using either discipline alone.
Our reasoning and results thus shed light on the deep relationship between \eqref{eq:KPP} and BBM in the half-space.

\subsection*{Results}

We denote coordinates on $\H^d$ by $\tbf{x} = (x, \tbf{x}', y) \in \R \times \R^{d-2} \times \R_+$.
We study traveling-wave solutions of \eqref{eq:KPP} that move parallel to the boundary.
Due to the rotational symmetry of $\H^d$ orthogonal to $\partial\H^d$, we are free to assume that our waves move in the $+x$ direction.
Then a traveling wave solution of \eqref{eq:KPP} of speed $c \geq 0$ takes the form $\Psi(x - c t, \tbf{x}', y)$ for some $\Psi \in \m{C}^2(\H^d) \cap \m{C}(\closure)$.
It follows that $\Psi$ satisfies the elliptic reaction--diffusion equation
\begin{equation}
  \label{eq:main}
  \begin{cases}
    \frac{1}{2} \Delta \Psi + c \partial_x \Psi + \Psi - \Psi^2 = 0 & \text{in } \H^d,\\
    \Psi = 0 & \text{on } \partial \H^d.
  \end{cases}
\end{equation}
We restrict our attention to bounded solutions of \eqref{eq:main}.
By the maximum principle, all such solutions lie between $0$ and $1$.

Some solutions of \eqref{eq:main} depend solely on the distance $y$ to the boundary $\partial\H^d$.
In this case the drift term $c \partial_x \Psi$ vanishes, so such solutions are steady states of the parabolic problem \eqref{eq:KPP}.
H.~Berestycki and the second author have shown that the half-space supports precisely two nonnegative bounded steady states \mbox{\cite[Theorem~1.1(A)]{BeG}}.
These are $0$ and $\varphi(y)$, the unique positive bounded solution of the following ODE on $\R_+$:
\begin{equation}
  \label{eq:steady}
  \frac{1}{2} \varphi'' + \varphi - \varphi^2 = 0, \quad \varphi(0) = 0.
\end{equation}
To ensure our traveling waves vary in $x$, we forbid these two ``trivial'' solutions.
\begin{definition}
  \label{def:TW}
  A \emph{traveling wave} of speed $c \geq 0$ is a nonnegative bounded solution of \eqref{eq:main} that is neither $0$ nor $\varphi$.
\end{definition}
In~Theorem~1.4(A) of \cite{BeG}, H.~Berestycki and the second author also considered the existence of traveling waves: $\H^d$ supports a traveling wave of speed $c$ if and only if $c \geq c_* \coloneqq \sqrt{2}$.
In this paper, we consider the uniqueness and structure of such waves.
We first discuss uniqueness.
\begin{theorem}
  \label{thm:unique}
  For each $d \geq 2$, there is exactly one traveling wave on $\H^d$ of speed $c_* = \sqrt{2}$, up to translation.
  In contrast, for every $c > c_*$, there exist infinitely many traveling waves of speed $c$ that are distinct modulo translation.
\end{theorem}
We note that a traveling wave on $\H^2$ extends to a wave on $\H^d$. 
Hence the unique minimal-speed wave on $\H^d$ depends on $x$ and $y$ alone.
In fact, the reduction to two dimensions (Proposition~\ref{prop:monotone} below) is an important step in the proof of Theorem~\ref{thm:unique}.
Thanks to this reduction, most of our analysis takes place in the half-\emph{plane} $\H \coloneqq \H^2 = \R \times \R_+$.

To prove Theorem~\ref{thm:unique}, we exploit the connection between the KPP equation \eqref{eq:main} and branching Brownian motion.
In our probabilistic analysis, we fix $d = 2$ and thus work on $\H = \R \times \R_+$.
Let $\big(X_t(u), Y_t(u) \midsemi u \in \mathcal{N}_t\big)$ denote a BBM in $\R^2$ without killing; $\mathcal{N}_t$ is the set of particles alive at time $t$ and $\big(X_t(u),Y_t(u)\big)$ is the position of particle $u$ at time $t$.
Each particle moves in $\R^2$ according to an independent two-dimensional Brownian motion and splits at unit rate into two child particles.
Given $u \in \mathcal{N}_t$ and $s \leq t$, we write $\big(X_s(u), Y_s(u)\big)$ for the position at time $s$ of the unique ancestor of $u$ alive at time $s$.
Let $\P_y$ denote the law of the BBM started from a single particle at position $(0, y)$.

Our half-plane BBM is the process $\big(X_t(u), Y_t(u) \midsemi u \in \mathcal{N}_t^+\big)$, where
\begin{equation*}
  \mathcal{N}_t^+ \coloneqq \big\{u \in \mathcal{N}_t : \inf_{s \leq t} Y_s(u) > 0\big\}.
\end{equation*}
In words, it is a branching Brownian motion whose particles are killed when they hit the boundary $\partial \H$.
In analogy with BBM in $\R$, we define an associated derivative martingale
\begin{align}
  \label{eqn:defdermartin}
  Z_t &\coloneqq \sum_{u \in \mathcal{N}_t^+} \big[\sqrt{2} t - X_t(u)\big] Y_t(u) \e^{\sqrt{2}X_t(u) - 2t}.
\end{align}
We also define a two-parameter family of additive martingales
\begin{equation}
  \label{eqn:supercritical-mg}
  W_t^{\lambda,\mu} \coloneqq \sum_{u \in \mathcal{N}^+_t} \e^{\lambda X_t(u)} \sinh[\mu Y_t(u)] \e^{-(\lambda^2/2 + \mu^2/2 + 1)t} \quad \text{for } \lambda, \mu > 0.
\end{equation}
The long-time limits of these martingales play a central role in our analysis.
\newpage
\begin{proposition}
  \label{prop:mg-convergence} 
  The following hold for any $y>0$\textup{:}
  \begin{enumerate}[label = \textup{(\roman*)}, itemsep = 0.7ex]
  \item
    $Z$ is a $\P_y$-martingale with a.s. limit $Z_\infty \gneqq 0$.
    
  \item
    $W^{\lambda,\mu}$ is a nonnegative $\P_y$-martingale with a.s. limit $W_\infty^{\lambda, \mu}$.
    If $\lambda^2 +\mu^2 < 2$, then $W_\infty^{\lambda, \mu} \gneqq 0$.
    Otherwise, $W_\infty^{\lambda, \mu} = 0$ $\P_y$-a.s.
  \end{enumerate}
\end{proposition}
\noindent
The notation $M \gneqq 0$ indicates a nonnegative random variable $M$ that is not almost surely zero.
For analogous results in one dimension, see, for example, \cite{LS87, Kyprianou04}.

We now construct KPP traveling waves from the Laplace transforms of these martingale limits.
In the following, $\E_y$ denotes expectation with respect to $\P_y$.
\begin{theorem}
  \label{thm:construction}
  The function
  \begin{equation}
    \label{eq:defTW}
    \Phi(x ,y) \coloneqq 1 - \E_y \exp\left(-\e^{-\sqrt{2}x} Z_\infty\right)
  \end{equation}
  is a shift of the unique minimal-speed traveling wave on $\H$.
  Moreover, for all $\lambda,\mu > 0$ such that $\lambda^2 + \mu^2 < 2$,
  \begin{equation}
    \label{eq:def-super}
    \Phi_{\lambda,\mu} (x,y) \coloneqq 1 - \E_y \exp\left(-\e^{-\lambda x} W_\infty^{\lambda,\mu}\right)
  \end{equation}
  is a traveling wave of speed $(\lambda^2 + \mu^2 + 2)/(2 \lambda) > \sqrt{2}$.
\end{theorem}
Theorems~\ref{thm:unique} and \ref{thm:construction} are closely related.
To prove minimal-speed uniqueness in Theorem~\ref{thm:unique}, we relate an arbitrary minimal-speed wave to the particular wave $\Phi$ defined in \eqref{eq:defTW}.
Drawing on the comparison principle and potential theory, we show that all minimal-speed traveling waves satisfy a certain tail bound.
In probability, this is known as tameness---traveling waves cannot be too exotic.
Following~\cite{AlM22}, we then use a probabilistic ``disintegration'' argument to show that every tame wave is necessarily a shift of $\Phi$.

This strategy differs from purely analytic approaches to traveling-wave uniqueness.
It has been standard practice in the analytic literature to prove sharp asymptotic behavior as a precursor to uniqueness; see, e.g., \cite{BN}.
Here, we only need an \emph{upper} bound in the form of tameness; the probabilistic disintegration handles the rest.
It seems likely that this hybrid approach could bear fruit in other problems.

The multiplicity of supercritical traveling waves in Theorem~\ref{thm:unique} follows from the fact that we construct a \emph{two}-parameter family of waves in Theorem~\ref{thm:construction}.
As a result, there are generally many waves with the same speed.
Let $\m{Q} \subset \R_+^2$ denote the open quarter-disk of radius $\sqrt{2}$ centered at the origin.
Given $c > c_*$, we define
\begin{equation}
  \label{eq:speed-set}
  \m{P}_c \coloneqq \big\{(\lambda,\mu) \in \m{Q} : (\lambda -c)^2 + \mu^2 = c^2 - 2\big\}.
\end{equation}
Then $\m{P}_c$ is the set of parameters $(\lambda, \mu)$ such that $\Phi_{\lambda,\mu}$ is a traveling wave with speed $c$.
The arcs $\m{P}_c$ foliate the quarter-disk $\m{Q}$ by speed; see Figure~\ref{fig:speeds}.
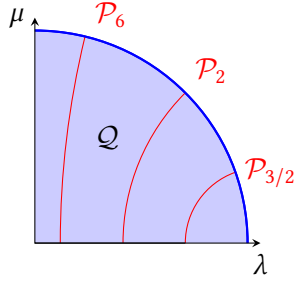
\begin{figure}[ht]
  \centering
  \begin{tikzpicture}[scale = 2]
    \fill[color=blue!20] (0,0) -- (0,1.414) arc[start angle =90,end angle = 0, radius = 1.414] -- cycle;
    \draw[thick, color=blue] (0,1.414) arc[start angle =90,end angle = 0, radius = 1.414];
    
    \draw[color = red]  (0.586,0) arc[start angle = 180, end angle = 135, radius = 1.414] node[above right] {$\mathcal{P}_2$};
    \draw[color = red]  (1,0) arc[start angle = 180, end angle = 110, radius = 0.5] node[right] {$\mathcal{P}_{3/2}$};
    \draw[color = red]  (0.169,0) arc[start angle = 180, end angle = 166.4, radius = 5.831] node[above right] {$\mathcal{P}_6$};
    \draw[thick, color=blue] (0,1.414) arc[start angle =90,end angle = 0, radius = 1.414];

    \node[] at (0.48,0.7) {$\m{Q}$};
    
    \draw[-stealth] (1,0) -- (0,0) -- (0,1.5) node[left] {$\mu$};
    \draw[-stealth] (0,0) -- (1.5,0) node[below] {$\lambda$};
  \end{tikzpicture}
  \caption{
    Parameter space for the supercritical waves $\Phi_{\lambda,\mu}$.
    Waves corresponding to $(\lambda, \mu) \in \mathcal{P}_c$ move with speed $c > c_*$.
  }
  \label{fig:speeds}
\end{figure}

\noindent
As $c \to c_*$, the set $\mathcal{P}_c$ converges to the single point $(\sqrt{2},0)$, which is formally associated with the derivative martingale $Z$.
This collapse hints at the uniqueness of the minimal-speed wave.
For a more complete discussion of the phenomenology along the boundary of $\m{Q}$, see Section~\ref{sec:supercritical}.

We now turn to the asymptotic behavior of our traveling waves.
The limits as $x \to \pm \infty$ are fairly simple: our waves are heteroclinic orbits connecting the steady state $\varphi$ on the left with $0$ on the right.
The waves exhibit more subtle behavior when we take $y \to \infty$.
In this regime, the boundary recedes and our waves become asymptotically one-dimensional.
Given $c \geq c_*$, let $w_c$ denote the unique (up to translation) one-dimensional traveling wave of speed~$c$, which satisfies the ODE
\begin{equation*}
  \frac{1}{2} w_c'' + c w_c' + w_c - w_c^2 = 0, \quad w_c(-\infty) = 1, \quad w_c(+\infty) = 0.
\end{equation*}
To fully determine $w_c$, we work with the translate given by the Laplace transform of a $c$-dependent martingale related to one-dimensional BBM; see \eqref{eqn:defw} and \eqref{eq:supercritical-1D-wave} for details.

At the minimal speed, we show that $\Phi$ converges to $w_{c_*}$ as $y \to \infty$ after a horizontal shift that is \emph{logarithmic in $y$}.
This novel phenomenon reflects the delicate structure of the derivative martingale $Z$, as we discuss below.
Supercritical waves exhibit a different complication: they are asymptotically one-dimensional but tilted with respect to the coordinate axes.
Given $(\lambda, \mu) \in \m{Q}$, let $R_{\lambda, \mu}$ denote clockwise rotation by the angle $\arctan(\mu/\lambda)$.
We show that the rotated wave $\Phi_{\lambda, \mu} \circ R_{\lambda, \mu}$ converges to a one-dimensional wave of speed
\begin{equation}
  \label{eq:speed}
  c(\lambda, \mu) \coloneqq \frac{\lambda^2 + \mu^2 + 2}{2\sqrt{\lambda^2 + \mu^2}}.
\end{equation}
To simplify the resulting statement, we extend $\Phi_{\lambda, \mu}$ by $0$ to the entire plane $\R^2$.
\begin{theorem}
  \label{thm:asymptotics}
  Every wave $\Phi^*$ in the collection $\{\Phi, \Phi_{\lambda, \mu}\}_{(\lambda, \mu) \in \m{Q}}$ satisfies $0 < \Phi^* < 1$, $\partial_x \Phi^* < 0$, and $\partial_y \Phi^* > 0$.
  The limits
  \begin{equation*}
    \Phi^*(-\infty, \anon) = \varphi \And \Phi^*(+\infty, \anon) = 0
  \end{equation*}
  hold uniformly and locally uniformly, respectively.
  Moreover, as $y \to \infty$,
  \begin{equation}
    \label{eq:y-limits}
    \Phi\big(x + \tfrac{1}{\sqrt{2}}\log y, y\big) \to w_{c_*}(x) \And \Phi_{\lambda, \mu} \circ R_{\lambda, \mu}(x, y) \to w_{c(\lambda, \mu)}(x)
  \end{equation}
  uniformly in $x$ for all $(\lambda, \mu) \in \m{Q}$, with $c(\lambda, \mu)$ given in \eqref{eq:speed}.
\end{theorem}
The speed of $\Phi_{\lambda, \mu}$ in Theorem~\ref{thm:construction} is $(\lambda^2 + \mu^2 + 2)/(2 \lambda)$, which differs from $c(\lambda, \mu)$ in \eqref{eq:speed} and Theorem~\ref{thm:asymptotics}.
The former is the speed of $\Phi_{\lambda, \mu}$ in the $x$-direction.
The latter is the apparent motion of $\Phi_{\lambda, \mu}$ perpendicular to its level sets in the $y \to \infty$ limit.
The discrepancy reflects the fact that the asymptotic level sets of $\Phi_{\lambda, \mu}$ are tilted at angle $\theta(\lambda, \mu) \coloneqq \arctan(\mu/\lambda)$ relative to vertical.
Thus the speeds differ by the geometric factor $\cos \theta(\lambda, \mu)$.

Traveling waves with asymptotically oblique level sets have been previously studied in the whole space for a variety of reactions~\cite{BH,HMR}.
In the KPP setting, Hamel and Nadirashvili~\cite{HN} have constructed an infinite-dimensional manifold of multidimensional traveling waves.
We believe that their construction is closely related to the parameter space $\m{Q}$ depicted in Figure~\ref{fig:speeds}.
In the half-space, one can construct an analogous manifold of waves of a given speed $c > c_*$ by taking arbitrary convex combinations of the additive martingales $\{W^{\lambda, \mu}\}_{(\lambda, \mu) \in \m{P}_c}$.
We speculate that the link between traveling waves and martingales can be used to rigorously classify \emph{all} traveling waves on $\H^d$, and indeed on $\R^d$.

As alluded to above, the most surprising feature of Theorem~\ref{thm:asymptotics} is the logarithmic shift $\tfrac{1}{\sqrt{2}} \log y$ in $\Phi$ as $y \to \infty$.
From a probabilistic standpoint, this novel phenomenon can be explained as follows.
Recall that \eqref{eq:defTW} expresses $\Phi$ in terms of the derivative martingale $Z$ defined in \eqref{eqn:defdermartin}.
As we move away from the boundary, the role of killing lessens, and we might expect $Z$ to resemble a one-dimensional derivative martingale.
In this spirit, define
\begin{equation*}
  D_t \coloneqq \sum_{u \in \m{N}_t} \big[\sqrt{2}t - X_t(u)\big] \e^{\sqrt{2} X_t(u) - 2t}.
\end{equation*}
Note that this sum ranges over the entire population $\m{N}_t$ of the BBM in $\R^2$.
Thus $D$ neglects killing, and is in fact the classical one-dimensional derivative martingale.
It has an a.s. positive limit $D_\infty$ whose Laplace transform is the minimal-speed one-dimensional traveling wave:
\begin{equation}
  \label{eqn:defw}
  w_{c_*}(x) \coloneqq 1 - \E \exp\left(-\e^{-\sqrt{2}x} D_\infty\right).
\end{equation}
In Proposition~\ref{prop:cvZy}, we use a first and second moment method conditioned on horizontal motion to show that
\begin{equation}
  \label{eq:1D-limit}
  Z_\infty(y)/y \to D_\infty \quad \text{in probability as } y \to \infty
\end{equation}
for a family of random variables $Z_\infty(y)$ with the law of $Z_\infty$ under $\P_y$.
Interpreting \eqref{eq:1D-limit} through the definitions \eqref{eq:defTW} and \eqref{eqn:defw}, we find
\begin{equation*}
  \Phi\left(x + \tfrac{1}{\sqrt{2}}\log y, y\right) = 1 - \E \exp\left(- \e^{-\sqrt{2}x} Z_\infty(y)/y\right) \to w_{c_*}(x) \quad \text{as } y \to \infty.
\end{equation*}
Thus, the limiting relation \eqref{eq:1D-limit} between the martingales $Z$ and $D$ implies the asymptotic behavior of $\Phi$ in Theorem~\ref{thm:asymptotics}.
We take a similar (simpler) approach to the asymptotics of the supercritical waves $\Phi_{\lambda, \mu}$; see Proposition~\ref{prop:wave} for details.

The asymptotic tail behavior of the minimal-speed wave has historically played an important role in the study of KPP propagation~\cite{LS87,HNRR}.
We expect the same will be true on the half-space.
We therefore develop a more precise understanding of the asymptotics of $\Phi$ as $x \to \infty$.
These are related to the well-known tail behavior of the one-dimensional wave: there exist $K_* > 0, a \in \R$, and $\delta > 0$ such that
\begin{equation}
  \label{eq:wave-constant}
  w_{c_*}(x) = K_* \big[x + a + \m{O}\big(\e^{-\delta x}\big)\big] \e^{-\sqrt{2}x} \quad \text{as } x \to \infty.
\end{equation}
In the following, we let $\log_+s \coloneqq \max\{\log s, 0\}$ and recall that $\tbf{x} \coloneqq (x, y)$ on $\H$.
\begin{theorem}
  \label{thm:tail}
  There exists $E \in L^\infty(\H)$ such that if $x > \tfrac{1}{\sqrt{2}} \log_+ y$,
  \begin{equation}
    \label{eq:tail-asymp-intro}
    \Phi(x, y) = K_*\big[x - \tfrac{1}{\sqrt{2}}\log_+ \norm{\tbf{x}} + E(x, y)\big] y \e^{-\sqrt{2} x}.
  \end{equation}
\end{theorem}
\noindent
Our proof is rooted in potential theory and uses Theorem~\ref{thm:asymptotics} as input.

Note that if we evaluate $\Phi$ at $\big(x + \tfrac{1}{\sqrt{2}}\log y, y\big)$ and take $y \to \infty$, the asymptotic behavior in \eqref{eq:tail-asymp-intro} comports with \eqref{eq:y-limits} and \eqref{eq:wave-constant}.
We highlight one final curiosity: if we instead hold $y$ fixed and take $x \to \infty$, we find
\begin{equation*}
  \Phi(x, y) = K_* \big[x - \tfrac{1}{\sqrt{2}} \log x + \m{O}_y(1)\big] y \e^{-\sqrt{2}x},
\end{equation*}
where the implied constant in $\m{O}_y(1)$ depends on $y$.
This hearkens to \eqref{eq:wave-constant} but includes a $\log x$ correction in the algebraic prefactor.
We are unaware of an analogue of this behavior in any other context.
\begin{remark}
  \label{rem:reaction}
  For simplicity, we focus on rate-1 binary branching Brownian motion.
  More broadly, one can consider BBMs with branching rate $r>0$ in which a branching particle has a random number $\m{Z}$ of offspring.
  If $g(s) \coloneqq \E s^{\m{Z}}$ denotes the probability generating function of $\m{Z}$ and $f(s) \coloneqq r[1 - s - g(1 - s)]$, then this generalized BBM is linked to the reaction--diffusion equation
  \begin{equation}
    \label{eq:KPP-general}
    \partial_t v = \frac{1}{2} \Delta v + f(v)
  \end{equation}
  via the renewal argument of McKean~\cite{McKean75}.
  In the rate-1 binary case, $\m{Z} \equiv 2$ and we recover the reaction $f(s) = s - s^2$ appearing in \eqref{eq:KPP}.
  
  Provided $\E \m{Z}^{1 + \gamma} < \infty$ for some $\gamma > 0$, this generalized BBM behaves much like binary BBM, and our methods and results apply with minor modifications.
  In particular, \eqref{eq:KPP-general} has a unique traveling wave (modulo isometry) in the half-space of minimal speed $\sqrt{2f'(0)} = \sqrt{2r(\E \m{Z} - 1)}.$
  We can thus treat a wide variety of reactions $f$ corresponding to rates $r$ and random variables $\m{Z}$ as described above.
  However, this ``probabilistic'' class does not exhaust the broader ``KPP'' class of reactions satisfying $f(s) \leq f'(0)s$.
  We anticipate that variations on our main results hold for all KPP reactions, but a proof seems to require new analytic ideas.
\end{remark}

\subsection*{Organization}

The remainder of the paper is organized as follows.
In Section~\ref{sec:BBM}, we develop the theory of BBM in the half-plane and prove Proposition~\ref{prop:mg-convergence}.
We construct our traveling waves in Section~\ref{sec:construction} and thus prove Theorem~\ref{thm:construction}.
In Section~\ref{sec:PDE}, we use purely analytic methods to prove a sharp upper bound---tameness---for all minimal-speed traveling waves.
We employ this tameness in a disintegration argument to prove the uniqueness of the minimal-speed wave in Section~\ref{sec:uniqueness}.
Finally, Section~\ref{sec:asymptotics} concerns the asymptotic behavior of our traveling waves and concludes the proofs of Theorems~\ref{thm:unique}, \ref{thm:asymptotics}, and \ref{thm:tail}.

\section*{Acknowledgements}
This research project was initiated during the thematic semester in probability and PDEs at the Centre de recherches mathématiques of the Université de Montréal led by Louigi Addario-Berry.
All authors warmly thank Louigi and the wider organizational team for their hospitality.
We are also grateful to Fran\c{c}ois Hamel for directing us to the reference~\cite{HN}, which plays an important role in our analytic arguments.

JB and BM acknowledge support from the Simons Foundation and the Centre de recherches mathématiques, JB as a Simons-CRM professor and BM as a Simons-CRM scholar in residence.
BM also acknowledges partial support from the CNRS UMI--CRM and from the GdR \emph{Branchement}.
CG is supported by the NSF Mathematical Sciences Postdoctoral Research Fellowship program through the grant DMS-2103383.
YHK was partially supported by the NSF Graduate Research Fellowship 1839302.

\section{Branching Brownian motion in the half-plane}
\label{sec:BBM}

In this section, we develop the theory of branching Brownian motion in the half-plane.
Motivated by the traveling wave construction in Theorem~\ref{thm:construction}, we investigate the convergence of various martingales associated to the BBM.
In particular, in this section we prove Proposition~\ref{prop:mg-convergence}.
Throughout, we take $d = 2$ and thus work on the half-plane $\H \coloneqq \R \times \R_+$.

Recall that we are interested in the derivative martingale $Z$ and the additive martingales $W^{\lambda,\mu}$ defined in \eqref{eqn:defdermartin} and \eqref{eqn:supercritical-mg}, respectively.
The convergence of the latter is relatively straightforward, as $W^{\lambda,\mu}$ is uniformly integrable if and only if $\lambda^2 +\mu^2 < 2$; see Section~\ref{subsec:wlambdamu} below.
The derivative martingale $Z$ poses more of a challenge---it has indefinite sign and is not uniformly integrable.
We therefore treat its convergence in several steps.
We begin in Section~\ref{sec:cvAdd} by showing that the so-called \emph{critical additive martingale}
\begin{equation*}
  W_t \coloneqq \sum_{u \in \mathcal{N}_t^+} Y_t(u) \e^{\sqrt{2} X_t(u) - 2t}
\end{equation*}
converges to $0$ almost surely as $t \to \infty$.
Next, we define ``shaved'' approximations of the derivative martingale $Z$ and prove their uniform integrability in Section~\ref{sec:cvShaved}.
In combination, these results suffice to show that the convergence of the derivative martingale $Z$.

\subsection{Branching Brownian motion in one dimension}
\label{sec:preliminaries}

Our analysis relies on the standard theory of BBM in one dimension.
Here we recall various elements of this theory, including the so-called spine decomposition.
We will subsequently adapt these ideas to the half-plane setting.

Recall that $(X_t(u), Y_t(u) \midsemi u \in \mathcal{N}_t)$ is a branching Brownian motion in the plane $\R^2$ without killing.
It follows that $(X_t(u) \midsemi u \in \mathcal{N}_t)$ is a BBM on the line $\R$ started from $0$.
Given $t \geq 0$, we let
\begin{equation*}
  \mathcal{H}_t \coloneqq \sigma\big(X_s(u) \midsemi u \in \mathcal{N}_s, s \leq t\big)
\end{equation*}
denote the filtration associated to this horizontal part of the BBM.
Note that the genealogical tree $\mathcal{T}$ is measurable with respect to $\mathcal{H}$.
Moreover, conditionally on $\mathcal{H}$, the process $(Y_t(u) \midsemi u \in \mathcal{N}_t)$ can be seen as an independent Brownian motion indexed by $\mathcal{T}$.

The critical additive and derivative martingales of the BBM $(X_t(u)\midsemi u \in \mathcal{N}_t)$ are the processes
\begin{align}
  \label{eq:add1D}
  A_t &\coloneqq \sum_{u \in \mathcal{N}_t} \e^{\sqrt{2}X_t(u) - 2t},\\
  D_t &\coloneqq \sum_{u \in \mathcal{N}_t} \big[\sqrt{2}t-X_t(u)\big] \e^{\sqrt{2} X_t(u) - 2t}.\nonumber
\end{align}
The asymptotic behavior of these martingales is closely related to the asymptotic behavior of extremal particles in BBM.
In \cite{LS87}, Lalley and Sellke proved that
\begin{equation}
  \label{eqn:defDinfty}
  \lim_{t \to \infty} A_t = 0 \And \lim_{t \to \infty} D_t = D_\infty > 0 \quad \text{a.s.}
\end{equation}
In turn, the Laplace transform of $D_\infty$ yields a minimal-speed traveling wave of the KPP equation in one dimension; see \eqref{eqn:defw}.

We now describe the main steps in the proof of \eqref{eqn:defDinfty} and thereby introduce several important tools that we subsequently deploy in the half-plane.
Following \cite{ChR,LPP,Lyo}, we study the law of BBM biased by an associated nonnegative martingale.
This allows us to make use of the following characterization of absolutely continuous random measures.
\begin{proposition}
  \label{prop:Durrett}
  Let $(M_t)_{t \geq 0}$ be a mean one nonnegative $(\mathcal{F}_t)_{t \geq 0}$-martingale under law $\P$. 
  Define a new probability measure $\Q$ on $\mathcal{F}_\infty$ via $\Q(E) \coloneqq \E_{\P}\left( M_t \indset{E} \right)$ for all $t \geq 0$ and $E \in \mathcal{F}_t$, written $\Q=M \P$ for short.
  The martingale $M_t$ converges almost surely under both $\P$ and $\Q$; we write $M_{\infty}$ for the a.s. limit.
  Then
  \begin{equation*}
    \Q(E) = \E_{\P}\left( M_\infty \indset{E} \right) + \Q(E \cap \{M_\infty = \infty\}) \ForAll E \in \mathcal{F}_\infty.
  \end{equation*}
  In particular, the following are equivalent:
  \begin{enumerate}[label = \textup{(\roman*)}, itemsep = 1ex]
  \item $M$  is uniformly integrable with respect to $\P$;

  \item $\E_{\P} \, M_\infty = 1$;

  \item $\Q(M_\infty = \infty) = 0$.
  \end{enumerate}
\end{proposition}
\noindent
Suppose we wish to show that  $M$ converges  $\P$-almost surely to a nondegenerate limit.
By Proposition~\ref{prop:Durrett}, it suffices to show that $\liminf_{t \to \infty} M_t < \infty$ $\Q$-a.s.
Likewise, the inverse result follows if $\limsup_{t \to \infty} M_t = \infty$ $\Q$-a.s.
Thus to prove that $M_t \to 0$ under $\P$, one need only show that $M$ diverges under $\Q$; this is typically much easier (see the proof of Proposition~\ref{prop:additive}).

In the remainder of the subsection, let $\P$ denote the law of the one-dimensional BBM $(X_t(u) \midsemi u \in \mathcal{N}_t)$.
Recalling the critical additive martingale $A$ from \eqref{eq:add1D}, we define the tilted probability measure $\tilde\P \coloneqq A \P$.
The seminal \emph{spine decomposition} of \cite{Lyo, ChR} states that under $\tilde{\P}$, the process $(X_t(u) \midsemi u \in \mathcal{N}_t)$ is a branching Brownian motion with a \emph{spine}: a distinguished particle that moves and reproduce differently.

More precisely, let $\check{\P}$ denote the law of the following process $(X_t(u) \midsemi u \in \mathcal{N}_t)$ augmented with a distinguished spine particle $\xi_t \in  \m{N}_t$.
The process begins with a single spine particle at the origin, which performs a Brownian motion with drift $\sqrt{2}$.
After an independent exponential time $t$ of parameter $2$, this particle splits into two children, one of which is designated the new spine particle $\xi_t$.
The new spine performs a copy of the above process from its birth location, while the non-spine child starts an independent BBM (without spine).
Spine decomposition theorems for branching Brownian motion are originally due to Chauvin and Rouault~\cite{ChR}.
A trajectorial construction was first developed by Lyons, Pemantle, and Peres for the Galton--Watson process~\cite{LPP}; Lyons later generalized this to spatial branching processes~\cite{Lyo}.
\begin{proposition}
  \label{prop:spine1d}
  For all $t \geq 0$ and $E \in \mathcal{H}_t$, we have $\check{\P}(E) = \tilde{\P}(E)$. Moreover,
  \begin{equation*}
    \check{\P}\big(\xi_t = u \mid \mathcal{H}_t\big) = A_t^{-1}\e^{\sqrt{2}X_t(u) - 2t} \ForAll u \in \m{N}_t.
  \end{equation*}
\end{proposition}
\noindent
In other words, the law of $(X_t(u) \midsemi u \in \mathcal{N}_t)$ is identical under $\check{\P}$ and $\tilde{\P}$.
Moreover, conditionally on the position of the particles, the spine particle is $u \in \mathcal{N}_t$ with probability proportional to $\e^{\sqrt{2}X_t(u)}$.

This proposition can be combined with Proposition~\ref{prop:Durrett} to show that $A_t \to 0$ almost surely.
Indeed, note that
\begin{equation*}
  A_t = \sum_{u \in \mathcal{N}_t} \e^{\sqrt{2}X_t(u) - 2t} \geq \e^{\sqrt{2} X_t(\xi_t) - 2t}.
\end{equation*}
Under $\check{\P}$, the process $t \mapsto \sqrt{2} X_t(\xi_t) - 2t$ is a driftless Brownian motion, so $\limsup_{t \to \infty} A_t = \infty$ $\check{\P}$-a.s.
Applying Proposition~\ref{prop:Durrett} to the event $E=\{A_\infty <\infty\}$, we obtain
\begin{equation*}
  \E_{\P} A_\infty = \check{\P}(A_\infty < \infty) = 0.
\end{equation*}
Since $A_\infty$ is nonnegative, $A_\infty = 0$ $\P$-a.s.

To treat the \emph{derivative} martingale $D_t$, we ``shave'' it, using the same technique as in \cite{BiK04}.
Given $\al > 0$ and $t \geq 0$, let
\begin{equation*}
  \mathcal{N}_t^\alpha \coloneqq \big\{ u \in \mathcal{N}_t : X_s(u) \leq \sqrt{2} s + \alpha \text{ for all } s\le t\big\}
\end{equation*}
denote the population of particles that remain below the curve $\sqrt{2}t + \al$.
Then we define the shaved derivative martingale
\begin{equation}
  \label{eqn:shavedDerivativeMartingaleDimension1}
  D^\alpha_t \coloneqq \sum_{u \in \mathcal{N}_t^\alpha} \big[\sqrt{2} t + \alpha - X_t(u)\big] \e^{\sqrt{2} X_t(u) - 2t}.
\end{equation}
For all $\alpha > 0$, $D^\alpha$ is a nonnegative martingale converging almost surely to some $D^\alpha_\infty$ as $t \to \infty$.
Using Proposition~\ref{prop:spine1d}, one can associate a spine decomposition to the law $\Q^\alpha = D^\alpha \P$ under which $\sqrt{2} t + \alpha - X_t(\xi_t)$ is a Bessel process of dimension~$3$.
Applying Proposition~\ref{prop:Durrett}, one can then show that $D^\alpha$ is uniformly integrable, and therefore $D^\alpha_\infty$ is nondegenerate.

To relate $D^\alpha$ to $D$, we use the following fact.
With probability $1$, there exists (random) $\alpha_0 > 0$ such that $\mathcal{N}^\alpha_t =\mathcal{N}_t$ for all $t \geq 0$ and $\alpha > \alpha_0$.
More precisely, if $M_t = \max_{\m{N}_t} X_t(u)$ denotes the maximal displacement at time $t$, then it is well known (see, e.g., \cite{LS87}) that
\begin{equation}
  \label{eqn:formulaForMaxDep}
  \liminf_{t \to \infty} \sqrt{2} t - M_t = +\infty\quad \text{a.s.}
\end{equation}
In particular, it follows that $\bar{M} \coloneqq \sup_{t \geq 0}( M_t - \sqrt{2} t) < \infty$ a.s.
Hence for all $\alpha > \bar{M}$ and $t \geq 0$, we have
\begin{equation*}
  D^\alpha_t = \sum_{u \in \mathcal{N}_t} \big[\sqrt{2} t + \alpha - X_t(u)\big] \e^{\sqrt{2} X_t(u) - 2t} = D_t + \alpha A_t.
\end{equation*}
We showed above that $A_t \to 0$ as $t \to \infty$.
It follows that $D_t$ has limit $D_\infty = D^\alpha_\infty$ a.s. when $\alpha > \bar{M}$.
This also shows that the map $\alpha \mapsto D^\alpha_\infty$ is constant above the random threshold $\bar{M}$.
Hence $D_t$ converges almost surely as $t \to \infty$ to the limit
\begin{equation*}
  D_\infty = \lim_{\alpha \to \infty} D^\alpha_\infty.
\end{equation*}

In the remainder of the section, we adapt this approach to show that the derivative martingale $Z$ in the half-plane has a nondegenerate long-time limit.

\subsection{The critical additive martingale in the half-plane}
\label{sec:cvAdd}

We now consider our BBM in the half-plane $\H$.
Given $t \geq 0$, let
\begin{equation*}
  \mathcal{F}_t \coloneqq \sigma\big(X_s(u),Y_s(u) \midsemi u \in \mathcal{N}_s,\, s \leq t\big) \And \m{F}_\infty = \sigma(\cup_{t \geq 0} \m{F}_t)
\end{equation*}
denote the filtration associated with the BBM in the entire plane $\R^2$.
We implicitly use the term martingale with respect to this filtration.
Recall the critical additive martingale defined by
\begin{equation*}
  W_t = \sum_{u \in \mathcal{N}_t^+} Y_t(u) \e^{\sqrt{2}X_t(u)-2t}.
\end{equation*}
In this subsection, we show that $W$ vanishes in the long-time limit.
\begin{proposition}
  \label{prop:additive}
  The process $W$ is a $\P_y$-martingale and $W_t \to 0$ $\P_y$-a.s. as $t \to \infty$.
\end{proposition}
It is easy to check that $W$ is a nonnegative $\P_y$-martingale.
\begin{proof}[Proof that $W$ is a martingale.]
  We employ the so-called many-to-one lemma: by linearity of expectation,
  \begin{align*}
    \E_y W_t &= \E_y\left(\sum_{u \in \mathcal{N}_t} Y_t(u) \e^{\sqrt{2} X_t(u) - 2t} \ind{Y_s(u) \geq 0 \midsemi s \leq t} \right)\\
             &= \e^{t} \E_y\left(Y_t \e^{\sqrt{2}X_t - 2t} \ind{Y_s \geq 0 \midsemi s \leq t} \right),
  \end{align*}
  where $(X,Y)$ is a 2-dimensional standard Brownian motion started from $(0,y)$.
  We have used the fact that the expected population size at time $t \geq 0$ is $\e^t$, due to the rate-$1$ branching.
  As a result, by independence and the martingale properties of $\e^{\sqrt{2}X_t - t}$ and $Y_t$, we have
  \begin{equation*}
    \E_y W_t = \E\left(\e^{\sqrt{2} X_t - t} \right) \E_y\big(Y_t \ind{Y_s \geq 0 \midsemi s \leq t}\big) = y.
  \end{equation*}
  Then the branching property of the BBM yields
  \begin{equation*}
    \E_y(W_{t+s} \mid \mathcal{F}_t) = \sum_{u \in \mathcal{N}_t^+} \e^{\sqrt{2} X_t(u) - 2t} \E_{Y_t(u)} W_s = W_t
  \end{equation*}
  for all $t,s\geq 0$.
  That is, $W$ is a $\P_y$ martingale.
\end{proof}
\noindent
It follows that $y^{-1}W$ is a mean-$1$ martingale, so we can define a tilted probability measure $\bar{\P}_y \coloneqq y^{-1}W \P_y$.
We relate this to an associated spine decomposition.

Let $\hat{\P}_y$ denote the law of the branching Brownian motion with spine constructed as follows.
Let $B$ be a Brownian motion started from $0$ and $S$ be an independent Bessel process of dimension $3$ started from $y$.
Our BBM with spine starts with a single spine particle that moves according to the process $t \mapsto (B_t + \sqrt{2} t,S_t)$.
After an independent exponential time of parameter $2$, this particle splits into two children, one of which is designated the new spine particle.
The new spine performs a copy of the above process from its birth location, while the non-spine child starts an independent BBM (without spine) in $\H$ with killing on $\partial \H$.
We let $\big(X_t(u), Y_t(u) \midsemi u \in \mathcal{N}^+_t\big)$ denote the positions of the particles at time $t$ and let $\xi_t \in \mathcal{N}_t^+$ denote the label of the spine.
Note that the identity $\xi$ of the spine is not measurable with respect to $\mathcal{F}$.

The spine decomposition theorem states that $\bar{\P}_y=\hat{\P}_y$ on $\mathcal F_\infty$.
\begin{proposition}
  \label{prop:addMart}
  For all  $t \geq 0$ and $E \in \mathcal{F}_t$, we have $\hat{\P}_y(E) = \bar{\P}_y(E)$.
  Moreover,
  \begin{equation}
    \label{eqn:whoIsTheSpine}
    \hat{\P}_y(\xi_t = u \mid \mathcal{F}_t) = W_t^{-1} Y_t(u) \e^{\sqrt{2} X_t(u) - 2 t} \ForAll u \in \m{N}_t^+.
  \end{equation}
\end{proposition}
\begin{proof}
  To begin, we augment the one-dimensional spine decomposition introduced in the previous subsection with vertical motion.
  We let the spine particle move as an independent standard Brownian motion in the $y$-direction while performing its previously-described horizontal motion.
  Likewise, we let the non-spine children perform standard BBMs in $\R^2$.
  Let $\check \P_y$ denote the law of this spine process in $\R^2$, which of course differs from the law $\h\P$ introduced above.
  Let $\ti \P_y \coloneqq A \P_y$ denote the tilt of $\P_y$ with respect to the horizontal additive martingale $A_t$ defined in \eqref{eq:add1D}.
  Then Proposition~\ref{prop:spine1d} can be easily extended to show that $\check\P_y = \ti\P_y$ on $\m{F}_\infty$ and
  \begin{equation}
    \label{eq:spine-identity}
    \check\P_y(\xi_t = u \mid \m{F}_t) = A_t^{-1} \e^{\sqrt{2}X_t(u) - 2t} \ForAll u \in \m{N}_t.
  \end{equation}

  Now fix $t \geq 0$ and $E \in \m{F}_t$.
  By definition,
  \begin{equation*}
    \bar\P_y(E) = y^{-1} \E_y\left(\indset{E} \sum_{u \in \m{N}_t^+} Y_t(u) \e^{\sqrt{2} X_t(u) - 2t}\right).
  \end{equation*}
  Using \eqref{eq:spine-identity} and the definition of the tilted measure $\ti\P_y$, this becomes
  \begin{equation*}
    \bar\P_y(E) = y^{-1} \ti\E_y\left(\indset{E} \sum_{u \in \m{N}_t^+} Y_t(u) \check\P_y(\xi_t = u \mid \m{F}_t)\right).
  \end{equation*}
  The integrand is measurable with respect to $\m{F}_t$, on which $\check\P_y = \ti\P_y$, so
  \begin{equation*}
    \bar\P_y(E) = y^{-1} \check\E_y\left(\indset{E} \sum_{u \in \m{N}_t^+} Y_t(u) \check\P_y(\xi_t = u \mid \m{F}_t)\right).
  \end{equation*}
  Now linearity of expectation and the tower property yield
  \begin{align}
    \bar\P_y(E) &= y^{-1} \check\E_y \left(\indset{E} \sum_{u \in \m{N}_t^+} Y_t(u) \check\E_y(\ind{\xi_t = u} \mid \m{F}_t)\right)\nonumber\\
                &= y^{-1} \check\E_y\left(\indset{E}\, \check\E_y\left[\sum_{u \in \m{N}_t^+} Y_t(u) \ind{\xi_t = u} \; \Big|\; \m{F}_t \right]\right)\nonumber\\
                &= y^{-1} \check\E_y\left(\indset{E}\, \check\E_y\left[Y_t(\xi_t) \ind{\xi_t \in \m{N}_t^+} \mid \m{F}_t\right]\right)\nonumber\\
                &= y^{-1} \check\E_y\left(\indset{E} Y_t(\xi_t) \ind{Y_s(\xi_s) \geq 0 \midsemi s \leq t} \right).\label{eq:RN}
  \end{align}
  Under the law $\check\P_y$, the process $t \mapsto Y_t(\xi_t)$ is a standard Brownian motion on the real line independent of the relative displacement of all non-spine particles.
  In contrast, under $\h \P_y$, $Y_t(\xi_t)$ is a Bessel process of dimension $3$.
  Now $y^{-1}Y_t(\xi_t) \ind{Y_s(\xi_s) \geq 0 \midsemi s \leq t}$ is the Radon--Nikodym derivative of the law of the Bessel process with respect to the Wiener measure, so \eqref{eq:RN} yields $\bar\P_y(E) = \h\P_y(E)$.

  We now use the connection between the laws $\hat{\P}_y$ and $\check{\P}_y$ to study the distribution of $\xi_t$ conditioned on $\mathcal{F}_t$.
  Fix $t \geq 0$, $E \in \mathcal{F}_t$ and $u \in \mathcal{N}_t^+$.
  Then
  \begin{equation*}
    \h\P_y(\xi_t = u, E) = y^{-1} \check\E_y \left(\indset{E} \ind{\xi_t = u} Y_t(\xi_t) \ind{Y_s(\xi_s) \geq 0 \midsemi s \leq t}\right) = y^{-1} \check\E_y\left(\indset{E} Y_t(u) \ind{\xi_t = u}\right).
  \end{equation*}
  By the tower property and \eqref{eq:spine-identity},
  \begin{align*}
    \h\P_y(\xi_t = u, E) &= y^{-1} \check\E_y\left(\indset{E} Y_t(u) \check\E_y\left[\ind{\xi_t = u} \mid \m{F}_t\right]\right)\\
                         &= y^{-1} \check\E_y\left(\indset{E} Y_t(u) A_t^{-1} \e^{\sqrt{2}X_t(u) - 2t} \right).
  \end{align*}
  Again, the integrand is measurable with respect to $\m{F}_t$, on which $\check\P_y = \ti\P_y$, so by the definition of $\ti\P_y$,
  \begin{equation*}
    \h\P_y(\xi_t = u, E) = y^{-1} \ti\E_y\left(\indset{E} Y_t(u) A_t^{-1} \e^{\sqrt{2}X_t(u) - 2t} \right) = y^{-1} \E_y\left(\indset{E} Y_t(u) \e^{\sqrt{2}X_t(u) - 2t} \right).
  \end{equation*}
  Finally, we use the definition of $\bar\P_y$ and the equality between $\bar\P_y$ and $\h\P_y$ shown above to conclude that
  \begin{equation*}
    \h\P_y(\xi_t = u, E) = \h\E_y\left(\indset{E} W_t^{-1} Y_t(u) \e^{\sqrt{2}X_t(u) - 2t} \right).
  \end{equation*}
  This completes the proof of \eqref{eqn:whoIsTheSpine}.
\end{proof}
Using the spine decomposition and Proposition~\ref{prop:Durrett}, we show that the additive martingale vanishes in the long-time limit.
\begin{proof}[Proof of Proposition~\textup{\ref{prop:additive}}]
  We have already shown that $W$ is a nonnegative martingale.
  Hence by Doob's theorem, $W$ converges almost surely.
  Combining Propositions~\ref{prop:Durrett} and \ref{prop:addMart}, we have
  \begin{equation}
    \label{eq:add-infinity}
    y^{-1}\E_y(W_\infty \ind{W_\infty < \infty}) = \bar{\P}_y\big( \limsup_{t \to \infty} W_t < \infty\big) = \hat{\P}_y\big(\limsup_{t \to \infty} W_t < \infty\big).
  \end{equation}
  Note that under the law $\hat{\P}_y$, $W_t \geq Y_t(\xi_t)\e^{\sqrt{2} X_t(\xi_t) - 2t}$.
  Now $X_t(\xi_t) - \sqrt{2} t$ is a standard Brownian motion, which makes arbitrarily large excursions almost surely.
  Moreover, almost surely, the Bessel process $Y_t(\xi_t)$ does not vanish in the limit.
  It follows that $\limsup_{t \to \infty} W_t  = \infty$ $\hat{\P}_y$-a.s.
  In light of \eqref{eq:add-infinity}, we conclude that
  \begin{equation*}
    \E_y(W_\infty \ind{W_\infty < \infty}) = y \hat{\P}_y(W_\infty < \infty) = 0.
  \end{equation*}
  Since $W_t$ is a martingale, $W_\infty = 0$ $\P_y$-a.s.
\end{proof}

\subsection{Convergence of the derivative martingale and shaving}
\label{sec:cvShaved}

We now turn to the derivative martingale in the half-space.
The main result of this subsection is the following.
\begin{proposition}
  \label{prop:derivativeMartingale}
  The process $Z$ is a $\P_y$-martingale and
  \begin{equation*}
    Z_\infty \coloneqq \lim_{t \to \infty} Z_t \gneqq 0\quad \text{$\P_y$-a.s.}
  \end{equation*}
\end{proposition}
Again, it is easy to check the martingale property.
\begin{proof}[Proof that $Z$ is a martingale]
  By \eqref{eqn:whoIsTheSpine} Proposition~\ref{prop:addMart},
  \begin{align*}
    \E_y Z_t &= \E_y\left( \sum_{u \in \mathcal{N}_t^+} \big[\sqrt{2}t - X_t(u)\big] Y_t(u) \e^{\sqrt{2} X_t(u) - 2t} \right)\\
             &= y\hat{\E}_y\big[\sqrt{2} t - X_t(\xi_t)\big] = 0.
  \end{align*}
  Here we have have used the fact that $\sqrt{2} t - X_t(\xi_t)$ is a standard Brownian motion under $\h\P_y$.
  Then the branching property and the martingale property of $W$ yield
  \begin{align*}
    \E(Z_{t+s} \mid \mathcal{F}_t) &= \sum_{u \in \mathcal{N}_t^+} \big[\sqrt{2}t - X_t(u)\big] \e^{\sqrt{2} X_t(u) - 2 t} \E_{Y_t(u)} W_s \\
                                   &\hspace{4cm}+ \sqrt{2}s \sum_{u \in \mathcal{N}_t^+} \e^{\sqrt{2} X_t(u) - 2 t} \E_{Y_t(u)} Z_s = Z_t
  \end{align*}
  for all $t,s \geq 0$.
\end{proof}
Because $Z$ has indefinite sign, we cannot immediately deploy the methods of the previous via a probability measure biased by $Z$.
As with $D_t$, we circumvent this issue through a family of shaved derivative martingales whose asymptotic behavior resembles that of $Z_t$.

Given $\alpha > 0$, we define
\begin{equation*}
  Z^\alpha_t \coloneqq \sum_{u \in \mathcal{N}_t^{+,\alpha}} \big[\sqrt{2}t+\alpha-X_t(u)\big] Y_t(u)\e^{\sqrt{2} X_t(u) - 2 t},
\end{equation*}
where $\mathcal{N}_t^{+,\alpha} \coloneqq \{u \in \mathcal{N}_t^+ : X_s(u) \leq \sqrt{2}s + \alpha, s \leq t\}$ denotes the collection of particles in $\m{N}_t^+$ whose trajectories remain below the line $\sqrt{2}s + \al$.
Adapting the calculation above, one can easily check that $Z^\al$ is a martingale.
Moreover, the definition of $\m{N}^{+,\al}$ implies that $Z^\alpha$ is a \emph{nonnegative} martingale and thus converges almost surely to a limit $Z_\infty^\al$.

We now use familiar tools to show that $Z^\alpha$ is uniformly integrable and thus has a nondegenerate limit.
Define the tilted measure
\begin{equation*}
  \shavetilt = \frac{Z^\alpha}{\alpha y} \P_y.
\end{equation*}
Let $\h\Q_y^\al$ denote the law of a BBM in $\H$ with spine $\xi$ such that the spine particle branches at accelerated rate $2$ and moves according to the process $(\sqrt{2} t + \alpha - S_t, S'_t)$, where $S$ and $S'$ are two independent Bessel processes of dimension $3$ started from $\alpha$ and $y$, respectively.
\begin{proposition}
  For all $t \geq 0$ and $E \in \mathcal{F}_t$, we have $\hat{\Q}^\alpha_y(E) = \shavetilt(E)$.
  Moreover,
  \begin{equation*}
    \hat{\Q}^\alpha_y(\xi_t = u \mid \mathcal{F}_t) = (Z_t^\al)^{-1} \big[\sqrt{2} t + \alpha - X_t(u)\big]Y_t(u) \e^{\sqrt{2} X_t(u) - 2 t} \ForAll u \in \m{N}_t^{+,\al}.
  \end{equation*}
\end{proposition}
\begin{proof}
  We mimic the proof of Proposition~\ref{prop:addMart}, using the proposition itself as input.
  Fix $t \geq 0$ and $E \in \m{F}_t$.
  Using the definition of $\bar{\P}_y$, \eqref{eqn:whoIsTheSpine}, and $\bar\P_y = \h \P_y$ on $\m{F}_\infty$, we have
  \begin{equation*}
    \shavetilt(E) = \al^{-1} \h{\E}_y\left(\indset{E} \sum_{u \in \m{N}_t^{+,\al}} \big[\sqrt{2} t + \alpha - X_t(u)\big] \h\P_y(\xi_t = u \mid \m{F}_t) \right).
  \end{equation*}
  Writing $\h\P_y(\xi_t = u \mid \m{F}_t)$ as an expectation and manipulating the tower property, we arrive at
  \begin{equation*}
    \shavetilt(E) = \al^{-1}\hat{\E}_y\Big(\indset{E} \big[\sqrt{2} t + \alpha - X_t(\xi_t)\big]\ind{\xi_t \in \mathcal{N}_t^{+,\alpha}}\Big).
  \end{equation*}
  Finally, we recognize the Radon--Nikodym derivative of a Bessel-3 process with respect to the Wiener measure and obtain $\shavetilt(E) = \h\Q_y^\al(E)$.

  Similarly, given $u \in \m{N}_t^{+,\al}$, we find
  \begin{equation*}
    \h\Q_y^\al(\xi_t = u, E) = \al^{-1}\hat{\E}_y\Big(\indset{E} \ind{\xi_t = u} \big[\sqrt{2} t + \alpha - X_t(u)\big]\Big).
  \end{equation*}
  Using the tower property, \eqref{eqn:whoIsTheSpine}, and $\bar\P_y = \h \P_y$ on $\m{F}_\infty$, we have
  \begin{equation*}
    \h\Q_y^\al(\xi_t = u, E) = (\al y)^{-1} \E_y\left(\indset{E} Y_t(u) \big[\sqrt{2} t + \alpha - X_t(u)\big] \e^{\sqrt{2} X_t(u) - 2t}\right).
  \end{equation*}
  Recalling the definition of $\shavetilt$ and the equality between $\shavetilt$ and $\h \Q_y^\al$ shown above, we conclude that
  \begin{equation*}
    \h\Q_y^\al(\xi_t = u, E) = \E_{\h\Q_y^\al}\left(\indset{E} (Z_t^\al)^{-1} Y_t(u) \big[\sqrt{2} t + \alpha - X_t(u)\big] \e^{\sqrt{2} X_t(u) - 2t}\right).
  \end{equation*}
  Allowing $E$ to vary over $\m{F}_t$, we obtain the desired result.
\end{proof}
We now use the spine decomposition to prove uniform integrability.
\begin{lemma}
  \label{lem:cvShavedMartingale}
  For all $\alpha > 0$, $Z^\alpha$ is a uniformly integrable martingale that converges $\P_y$-almost surely as $t \to \infty$ to a nonnegative, nondegenerate random variable $Z^\alpha_\infty$.
\end{lemma}
\begin{proof}
  By Proposition \ref{prop:Durrett}, $Z^\al$ is uniformly integrable if and only if
  \begin{equation}
    \label{eq:UI-goal}
    \shavetilt(Z_\infty^\alpha < \infty) = 1.
  \end{equation}
  Let $(\tau_n)_{n \in \N}$ denote the increasing sequence of times at which the spine gives birth.
  Under $\hat{\Q}^\alpha_y$, the $(\tau_n)_{n \in \N}$ are the atoms of a Poisson process with intensity $2$.
  Let
  \begin{equation*}
    \mathcal{Y} \coloneqq \sigma\left(\big(X_s(\xi_s), Y_s(\xi_s), \tau_n\big) \midsemi s \geq 0, \, n \in \N\right)
  \end{equation*}
  denote the filtration associated to the trajectory of the spine and the birth times.
  The spine has position $(\sqrt{2} t + \alpha - S_t, S'_t)$, where $S$ and $S'$ are two independent Bessel processes of dimension $3$.
  Hence the martingale property of $Z$ for standard BBM yields
  \begin{equation*}
    \E_{\hat{\Q}^\alpha_y}(Z_t^\alpha \mid \mathcal{Y}) = {S_tS_t'} \e^{-\sqrt{2} (S_t - \alpha)} + \sum_{n =1}^\infty S_{\tau_n} S'_{\tau_n} \e^{-\sqrt{2}(S_{\tau_n} - \alpha)} \ind{\tau_n < t} \quad \hat{\Q}^\alpha_y\text{-a.s.}
  \end{equation*}
  Using Fatou's lemma and the transience of Bessel-3 processes, we find
  \begin{equation*}
    \E_{\hat{\Q}^\alpha_y}(Z_\infty^\alpha \mid \mathcal{Y}) \leq \sum_{n = 1}^\infty S_{\tau_n} S'_{\tau_n} \e^{-\sqrt{2}(S_{\tau_n} - \alpha)} \quad \hat{\Q}^\alpha_y\text{-a.s.}
  \end{equation*}
  Now, the law of the iterated logarithm for Bessel processes implies that for any $\eps > 0$, almost surely $t^{1/2- \eps} \leq S_t \leq t^{1/2+\eps}$ for sufficiently large $t$.
  It follows that
  \begin{equation*}
    \sum_{n = 1}^\infty S_{\tau_n} S'_{\tau_n} \e^{-\sqrt{2}(S_{\tau_n} - \alpha)} < \infty \quad \hat{\Q}^\alpha_y\text{-a.s.}
  \end{equation*}
  This proves \eqref{eq:UI-goal}.
  As a result, $Z^\alpha$ is a closed martingale that converges $\P_y$-almost surely and in $L^1$ to $Z^\alpha_\infty$.
  In particular, we have
  \begin{equation*}
    \E_y Z^\alpha_\infty = \E_y Z^\alpha_0 = \alpha y>0,
  \end{equation*}
  which shows that $Z^\alpha_\infty$ is positive with positive probability.
\end{proof}

We now complete the proof of Proposition~\ref{prop:derivativeMartingale} by showing the convergence of the derivative martingale $Z$ to the nonnegative limit $\lim_{\alpha \to \infty} Z_\infty^\alpha$.
\begin{proof}[Proof of Proposition~\textup{\ref{prop:derivativeMartingale}}]
  We first note that $\alpha \mapsto Z^\alpha_t$ is increasing for all $t > 0$, so $\alpha \mapsto Z^\alpha_\infty$ is also $\P^y$-a.s. increasing.
  Hence $\lim_{\alpha \to \infty} Z_\infty^\alpha$ is well-defined $\P_y\text{-a.s.}$.
  
  Next, using \eqref{eqn:formulaForMaxDep}, we observe that
  \begin{equation}
    \label{eqn:unifUb}
    \bar{M}^+ \coloneqq \sup_{t \geq 0} \sup_{u \in \mathcal{N}^+_t} X_t(u) - \sqrt{2}t < \infty \quad \text{$\P^y$-a.s.}
  \end{equation}
  For all $\alpha > \bar{M}^+$ and $t \geq 0$, we have
  \begin{align*}
    Z^\alpha_t &= \sum_{u \in \mathcal{N}^+_t} \big[\sqrt{2}t+\alpha-X_t(u) \big] Y_t(u)\e^{\sqrt{2} X_t(u) - 2 t}\\
               &= \sum_{u \in \mathcal{N}^+_t}\big[\sqrt{2}t-X_t(u) \big] Y_t(u)\e^{\sqrt{2} X_t(u) - 2 t} + \alpha \sum_{u \in \mathcal{N}^+_t} Y_t(u)\e^{\sqrt{2} X_t(u) - 2 t} = Z_t + \alpha W_t.
  \end{align*}
  Taking $t \to \infty$ and using Proposition~\ref{prop:additive}, we conclude that $\P^y$-a.s., we have
  \begin{equation*}
    \lim_{t \to \infty} Z_t = Z^\alpha_\infty \ForAll \alpha > \bar{M}^+.
  \end{equation*}
  In particular, the family $(Z_\infty^\al)_{\al > 0}$ stabilizes once $\al$ exceeds the (random) threshold $\bar{M}^+$.
  Since $\bar{M}^+$ is almost surely finite,
  \begin{equation*}
    Z_\infty \coloneqq \lim_{t \to \infty} Z_t = \lim_{\al \to \infty} Z^\al_\infty.
    \qedhere
  \end{equation*}
\end{proof}

\subsection{Supercritical additive martingales}
\label{subsec:wlambdamu}

We now turn to the supercritical additive martingales
\begin{equation*}
  W_t^{\lambda,\mu} = \sum_{u \in \mathcal{N}^+_t} \e^{\lambda X_t(u)} \sinh[\mu Y_t(u)] \e^{-(\lambda^2/2 + \mu^2/2 + 1)t}
\end{equation*}
parameterized by $\lambda, \mu > 0$.
Using the many-to-one lemma and the branching property, one can readily check that $W^{\lambda, \mu}$ is a martingale; this is sufficiently similar to our previous arguments that we omit the proof.
Because $W^{\lambda,\mu}$ is nonnegative, it has an almost sure limit $W_\infty^{\lambda, \mu} \geq 0$.
In this subsection we prove the following dichotomy for $W^{\lambda,\mu}_\infty$.
\begin{lemma}
  \label{lem:additiveMartingale}
  If $\lambda^2 +\mu^2 < 2$, then the martingale $W^{\lambda, \mu}$ is uniformly integrable and $W_\infty^{\lambda,\mu} \gneqq 0$.
  Otherwise, $W_\infty^{\lambda, \mu} = 0$ $\P_y$-a.s.
\end{lemma}
\begin{proof}
  We again employ a spine decomposition.
  Given $\lambda, \mu, y > 0$, let $\superspine$ denote the law of the following BBM with spine $\xi$.
  The spine particle moves according to a process $(X_t,Y_t)$, where $X$ is a Brownian motion with drift $\lambda$ started from $0$ and $Y$ is an independent Brownian motion with drift $\mu$ started from $y$ and conditioned to stay positive.
  The spine particle branches at the accelerated rate 2 and non-spine particles behave as independent BBMs in $\H.$

  Adapting the proof of Proposition~\ref{prop:addMart}, one can show that $\superspine$ coincides with the tilted measure $\supertilt = \frac{W^{\lambda,\mu}}{\sinh(\mu y)} \P_y$ on $\m F_\infty$.
  We omit most of the repetitive details, but highlight one new feature.
  In analogy with the reasoning following \eqref{eq:RN}, we must show that
  \begin{equation}
    \label{eq:RN-drift}
    \frac{\sinh[\mu Y_t(\xi_t)]}{\sinh(\mu y)} \e^{-\frac{\mu^2 t}{2}} \ind{Y_s(\xi_s) \geq 0 \midsemi s \leq t}
  \end{equation}
  is the Radon--Nikodym derivative of a Brownian motion with drift $\mu$ conditioned to stay positive relative to the Wiener measure.
  To see this, let $B$ be a standard Brownian motion and $V$ be a Brownian motion with drift $\mu$, both started from $y$.
  Given $t \geq 0$, let $B_{[0, t]}$ denote the process $B$ on the time interval $[0, t]$.
  For any $f \in \m{C}_c\big(\m{C}([0, t])\big)$, we use the Girsanov transformation to write
  \begin{equation}
    \label{eq:Girsanov}
    \hspace*{-5pt}\E \left[\frac{\sinh(\mu B_t)}{\sinh(\mu y)} \e^{-\frac{\mu^2 t}{2}} \ind{B_{[0,t]}\geq 0} f(B_{[0,t]})\right]\! = \E\left[\frac{1 - \e^{-2\mu V_t}}{1 - \e^{-2\mu y}} \ind{V_{[0,t]}\geq 0} f(V_{[0,t]})\right]\!.
  \end{equation}
  Recall that under $\superspine$, $Y(\xi)$ is a Brownian motion with drift $\mu$ conditioned to stay positive.
  Since
  \begin{equation*}
    \P(V_t \geq 0 \text{ for all } t \geq 0) = 1 - \e^{-2\mu y},
  \end{equation*}
  Doob's $h$-transform theory implies that the right side of \eqref{eq:Girsanov} is $\h{\E}_y^{\lambda, \mu}f\big(Y_{[0,t]}(\xi_t)\big)$.
  This confirms that the expression in \eqref{eq:RN-drift} is the claimed Radon--Nikodym derivative and justifies this form of the spine decomposition theorem.

  Observe that $\superspine$-a.s., we have
  \begin{equation*}
    W_t^{\lambda,\mu} \geq \e^{\lambda X_t(\xi_t) - \lambda^2t/2} \sinh[\mu Y_t(\xi_t)] \e^{-\mu^2t/2 } \e^{-t}.
  \end{equation*}
  Recalling that $X(\xi)$ and $Y(\xi)$ have drift $\lambda$ and $\mu$, respectively, we have
  \begin{equation*}
    \liminf_{t \to \infty} \frac{1}{t} \log W_t^{\lambda,\mu} \geq \frac{\lambda^2 + \mu^2}{2} - 1 \quad \superspine\text{-a.s.}    
  \end{equation*}
  Hence, if $\lambda^2 + \mu^2 > 2$ (or if $\lambda^2 + \mu^2= 2$ by the law of iterated logarithm), we have $\supertilt(W_\infty^{\lambda,\mu} < \infty) = 0$.
  Then Proposition~\ref{prop:Durrett} implies that $W_\infty^{\lambda,\mu} = 0$ $\P_y$-a.s.

  If $\lambda^2 + \mu^2 < 2$, we condition with respect to the spine's position and branching times $(\tau_n)_{n \in \N}$.
  The martingale property for $W^{\lambda, \mu}$ yields the  $\superspine$-a.s. bound
  \begin{equation*}
    \hat{\E}^{\lambda,\mu}_y(W_\infty^{\lambda,\mu} \mid \mathcal{Y}) \leq \frac{1}{2}\sum_{n = 1}^\infty \e^{\lambda X_{\tau_n} + \mu Y_{\tau_n} -(\lambda^2/2 + \mu^2/2+1) \tau_n.}
  \end{equation*}
  This is $\hat{\P}^{\lambda,\mu}$-almost surely finite, so Proposition~\ref{prop:Durrett} implies that $W^{\lambda,\mu}$ is uniformly integrable under $\P_y$.
  It follows that $\E_y W_\infty^{\lambda,\mu} = \sinh(\mu y) > 0$, so $W_\infty^{\lambda,\mu} \gneqq 0$.
\end{proof}
\begin{proof}[Proof of Proposition~\textup{\ref{prop:mg-convergence}}]
  The proposition unites Proposition~\ref{prop:derivativeMartingale} and Lemma~\ref{lem:additiveMartingale}.
\end{proof}

\section{Constructions of KPP traveling waves}
\label{sec:construction}

We now use the nondegenerate martingale limits $Z_\infty$ and $W^{\lambda,\mu}_\infty$ from Proposition~\ref{prop:mg-convergence} to construct traveling waves for the KPP equation in $\H$.
We rely on McKean's link between BBM and the KPP equation, as well as ``smoothing equations'' satisfied in law by the martingale limits.

\subsection{A minimal-speed wave}

Recall from Theorem \ref{thm:construction} the definition \eqref{eq:defTW} of $\Phi$:
\begin{equation*}
  \Phi(x,y) = 1 - \E_y \exp\left(-\e^{-\sqrt{2}x} Z_\infty\right).
\end{equation*}
We show that $\Phi$ is a traveling wave on $\H$ of speed $c_* = \sqrt{2}$ in the sense of Definition~\ref{def:TW}.
Our main tool is the McKean representation of solutions of \eqref{eq:KPP}.
This connection between BBM and the KPP equation was first observed by McKean \cite{McKean75} in one dimension.
Here, we state a straightforward analogue valid in $\H$.
\begin{proposition}[McKean representation]
  \label{prop:McK}
  If $\phi \in L^\infty(\H)$ satisfies $0 \leq \phi \leq 1$, then
  \begin{equation*}
    u(t,x,y) \coloneqq \E_y\left(1 - \prod_{v \in \mathcal{N}_t^+} \big[1 -\phi\big(x - X_t(v),Y_t(v)\big) \big] \right)
  \end{equation*}
  is the unique solution of \eqref{eq:KPP} with initial condition $u(0, \anon) = \phi$.
\end{proposition}
\begin{proof}
  First suppose that $\phi \colon \R \to [0, 1]$ is additionally $\mathcal{C}^2$.
  We define
  \begin{equation*}
    q_\phi (t,x,y) = \E_y\left( \prod_{u \in \mathcal{N}_t} \big[1 - \phi\big(x - X_t(u), Y_t(u)\big)\big] \right).
  \end{equation*}
  By dominated convergence,
  \begin{equation}
    \label{eq:McKean-boundary}
    \lim_{t \to 0} q_\phi(t,x,y) = 1 - \phi(x,y) \And \lim_{y \to 0} q_\phi(t,x,y) = 1.
  \end{equation}
  To begin, we compute $\partial_tq_\phi|_{t=0}$.
  Applying the branching property at the first branching time of the BBM, we have
  \begin{align*}
    q_\phi(t,x,y) = &\e^{-t}\E_y[1 - \phi(x - X_t,Y_t)]\\
                    &\hspace{10pt}+ \int_0^t \e^{-s} \E_y\left(\E_{Y_s}\big[q_\phi(t-s,x - X_s,Y_s) \mid X_s,Y_s\big]^2 \ind{\inf_{[0,s]} Y > 0}\right) \d s.
  \end{align*}
  Thus Itô's formula, dominated convergence, and \eqref{eq:McKean-boundary} yield
  \begin{equation*}
    \partial_t q_\phi(0,x,y) = -q_\phi(0,x,y) + \frac{1}{2} \Delta q_\phi(0,x,y) + q_\phi(0,x,y)^2.
  \end{equation*}
  Now fix $t, h > 0$.
  Given $u \in \m{N}_h^+$, let $\m{N}_{t+h}^+(u)$ denote the set of descendants of $u$ alive at time $t + h$.
  Applying the branching property at time $h$, we have
  \begin{align*}
    q_\phi(t+h,x,y) &= \E_y \Bigg[\E_y\Bigg(\prod_{u \in \mathcal{N}_h^+} \prod_{v \in \mathcal{N}_{t+h}^+(u)} \big[1 - \phi\big(x - X_{t+h}(v),Y_{t+h}(v)\big)\big] \; \Big| \; \m{F}_h\Bigg)\Bigg]\\
                    &= \E_y\Bigg[\prod_{u \in \mathcal{N}_h^+} q_\phi\big(t,x - X_h(u),Y_h(u)\big)\Bigg].
  \end{align*}
  Our earlier computation allows us to differentiate this expression at $h = 0$ to find
  \begin{equation*}
    \partial_t q_\phi = \frac{1}{2} \Delta q_\phi + q_\phi^2 - q_\phi.
  \end{equation*}
  So $1-q_\phi$ solves \eqref{eq:KPP}.
  
  To extend this result to $\phi \in L^\infty(\H)$, let $u$ denote the unique solution of \eqref{eq:KPP} with initial data $\phi$.
  For all $\eps > 0$, $u(\eps, \anon) \in \m{C}^2$.
  Hence we have just shown that
  \begin{equation*}
    u^\eps(t,x,y) \coloneqq 1 - \E_y\left(\prod_{u \in \mathcal{N}_t} \big[1 - u\big(\eps, x - X_t(u), Y_t(u)\big)\big]\right)
  \end{equation*}
  solves \eqref{eq:KPP} with initial data $u(\eps, \anon)$.
  By the semigroup property, $u^\eps(t, \anon) = u(t + \eps, \anon)$ for all $t \geq 0$.
  We conclude by taking $\eps \to 0$ and using the fact that $u(\eps, \anon) \to \phi$ weakly in $L^\infty$.
\end{proof}
We now prove that the limit of the derivative martingale satisfies a recursive equation in distribution.
More precisely, using the branching property of the BBM, we show that the law of $Z_\infty$ is a fixed point of a multitype version of the so-called smoothing transform.
In the following, for fixed $t,s \geq 0$, we write $u \preceq v$ when a particle $v \in \m{N}_{t + s}^+$ is a descendant of $u \in \m{N}_t^+$.
\begin{lemma}
  \label{lem:smoothing}
  For all $y > 0$ and $t > 0$, we have
  \begin{equation}
    \label{eqn:smoothing}
    Z_\infty = \sum_{u \in \mathcal{N}_t^+} \e^{\sqrt{2} X_t(u) - 2 t} Z_\infty(u),
  \end{equation}
  where
  \begin{equation*}
    Z_\infty(u) \coloneqq \lim_{s \to \infty} \sum_{u \preceq v \in \mathcal{N}^{+}_{t+s}} \big[\sqrt{2} s - X_{t+s}(v)\big] \e^{\sqrt{2} X_{t+s}(v) - 2s} \For u \in \m{N}_t.
  \end{equation*}
  Moreover, conditionally on $\mathcal{F}_t$, the random variables $\big(Z_\infty(u) \midsemi u \in \mathcal{N}_t^+\big)$ are independent and $Z_\infty(u)$ has the distribution of $Z_\infty$ under law $\P_{Y_t(u)}$.
\end{lemma}
\begin{proof}
  Given $s, t \geq 0$ and $u \in \m{N}_t^+$, we define $\mathcal{N}_{t+s}^{+}(u) = \{v \in \mathcal{N}_{t+s}^{+} : u \preceq v\}$ and
  \begin{align*}
    W_s(u) &\coloneqq \sum_{v \in \mathcal{N}^{+}_{t,s}(u)} Y_{t+s}(v) \e^{\sqrt{2} [X_{t+s}(v) - X_t(u)] - 2 s},\\
    Z_s(u) &\coloneqq \sum_{v \in \mathcal{N}^{+}_{t,s}(u)} \big(\sqrt{2}s - [X_{t+s}(v) - X_t(u)]\big) Y_{t+s}(v) \e^{\sqrt{2} [X_{t+s}(v) - X_t(u)] - 2 s}.
  \end{align*}
  With this notation, we can write
  \begin{align}
    Z_{t+s} &= \sum_{u \in \mathcal{N}_t^+} \sum_{ v \in \mathcal{N}_{t+s}^{+}(u)}\big[\sqrt{2}(t+s)-X_{t+s}(v)\big] Y_{t+s}(v)\e^{\sqrt{2} X_{t+s}(v) - 2 (t+s)}\nonumber\\
            &= \sum_{u \in \mathcal{N}_t^+} \big[\sqrt{2}t - X_t(u)\big] \e^{\sqrt{2}X_t(u) - 2t} W_s(u) + \sum_{u \in \mathcal{N}_t^+} \e^{\sqrt{2}X_t(u) - 2t} Z_s(u).\label{eq:smoothing-prelim}
  \end{align}  
  The branching property implies that conditionally on $\mathcal{F}_t$, $\big(W_s(u), Z_s(u)\big)_{u \in \m{N}_t^+}$ are independent random pairs and $\big(W_s(u), Z_s(u)\big)$ has the law of $(W_s,Z_s)$ under $\P_{Y_t(u)}$.
  In particular, taking $s \to \infty$ and using Proposition~\ref{prop:mg-convergence}, we have
  \begin{equation*}
    \lim_{s \to \infty} \big(W_s(u), Z_s(u)\big) = \big(0,Z_\infty(u)\big) \quad \text{$\P_y$-a.s.}
  \end{equation*}
  The limits $\big(Z_\infty(u) \midsemi u \in \m{N}_t^+\big)$ are independent conditionally on $\mathcal{F}_t$ and share the law of $Z_\infty$ under $\P_{Y_t(u)}$.
  Moreover, \eqref{eqn:smoothing} follows from \eqref{eq:smoothing-prelim}.
\end{proof}
Using the branching property, we can check that $Z_\infty$ is positive precisely on the survival set of the BBM.
\begin{corollary}
  \label{cor:survival}
  For all $y > 0$, we have
  \begin{equation*}
    \{Z_\infty > 0\} = \{\mathcal{N}_t^+ \neq \emptyset \text{ for all } t \geq 0\} \quad \text{$\P_y$-a.s.}
  \end{equation*}
\end{corollary}
\begin{proof}
  Let $S \coloneqq \{\mathcal{N}_t^+ \neq \emptyset \text{ for all } t \geq 0\}$ denote the survival event of the BBM on $\H$.
  The definition of $Z_t$ immediately implies that
  \begin{equation}
    \label{eq:subset}
    \{ Z_\infty > 0\} \subset S.
  \end{equation}
  Hence it suffices to show that $\P_y(S) = \P_y(Z_\infty > 0)$ for all $y > 0$.
  We note that Proposition~\ref{prop:derivativeMartingale} implies that $q$ is positive on $\R_+$, so \eqref{eq:subset} yields $p \geq q > 0$.
  
  Given $y > 0$, we define the functions $p$ and $q$ by
  \begin{equation*}
    p(y) \coloneqq \P_y(S^c) \And q(y) \coloneqq \P_y(Z_\infty = 0).
  \end{equation*}
  Using the branching property at time $t$, we see that
  \begin{equation*}
    \P_y(S^c \mid \mathcal{F}_t) = \prod_{u \in \mathcal{N}_t^+} \P_{Y_t(u)}(S^c).
  \end{equation*}
  It follows that $p$ satisfies the recursive identity
  \begin{equation}
    \label{eq:p-recursive}
    p(y)= \E_y \prod_{u \in \mathcal{N}_t^+} p\big(Y_t(u)\big)
  \end{equation}
  for all $t \geq 0$.
  Similarly, using Lemma~\ref{lem:smoothing}, we have
  \begin{equation}
    \label{eq:q-recursive}
    q(y) = \E_y\big[\P_y(Z_\infty(u) = 0 \text{ for all } u \in \mathcal{N}_t^+ \mid \mathcal{F}_t)\big] =\E_y \prod_{u \in \mathcal{N}_t^+} q\big(Y_t(u)\big)
  \end{equation}
  for all $t \geq 0$.
  Combining \eqref{eq:p-recursive}, \eqref{eq:q-recursive}, and Proposition~\ref{prop:McK}, we see that $1-p$ and $1-q$ are both stationary solutions of \eqref{eq:KPP} on the Dirichlet half-line.
  That is, both are positive bounded solutions of \eqref{eq:steady}.
  Lemma~6.1 of \cite{BeG} states that there is only one such solution; we have previously denoted it by $\varphi$.
  It follows that $p = 1 - \varphi = q$, which completes the proof.
\end{proof}
Armed with Lemma~\ref{lem:smoothing} and Proposition~\ref{prop:McK} we are now able to show that $\Phi$ is a minimal-speed traveling wave.
\begin{lemma}
  \label{lem:PhiSatisfiesTW}
  The function $\Phi$ defined in \eqref{eq:defTW} is a traveling wave of speed $\sqrt{2}$.
\end{lemma}
\begin{proof}
  Fix $(x, y) \in \H$ and $t > 0$.
  Using every part of Lemma~\ref{lem:smoothing}, the tower property for $\m{F}_t$ yields
  \begin{align*}
    \Phi(x,y) &= 1 - \E_y \exp\left( - \sum_{u \in \mathcal{N}_t^+} \e^{\sqrt{2} (X_t(u) - \sqrt{2}t - x)} Z_\infty(u) \right)\\
              &= 1 - \E_y\prod_{u \in \mathcal{N}_{t}^+} \big[1 - \Phi\big(\sqrt{2} t + x - X_t(u),Y_t(u)\big) \big].
  \end{align*}
  Now Proposition~\ref{prop:McK} implies that the function
  \begin{equation*}
    \Phi(x - \sqrt{2} t,y) = 1 - \E_y\prod_{u \in \mathcal{N}_{t}^+} \big[1 - \Phi\big(x - X_t(u),Y_t(u)\big) \big]
  \end{equation*}
  solves the KPP equation \eqref{eq:KPP} with initial data $\Phi(x,y)$.
  It follows that $\Phi$ solves \eqref{eq:main}.
  Clearly, $\Phi$ is bounded.
  Moreover, because $Z_\infty$ is not identically zero, $\Phi$ is nonconstant in $x$, and thus neither $0$ nor $\varphi$.
  By Definition~\ref{def:TW}, $\Phi$ is a traveling wave of speed $\sqrt{2}$.
\end{proof}

\subsection{Higher-speed waves}
\label{sec:supercritical}
We now construct traveling waves with speeds $c > c_*$.
As for $\Phi$, we use the Laplace transform of the martingale limits $W^{\lambda,\mu}_\infty$.
The proof is very similar to that presented above, so we omit some repeated details.

Fix $t > 0$.
As with $Z_\infty$, the branching property of the BBM implies that the random variable $W_\infty^{\lambda, \mu}$ satisfies the smoothing transform
\begin{equation}
  \label{eq:newSmoothing}
  W_\infty^{\lambda, \mu} = \sum_{u \in \mathcal{N}_t^+} \e^{\lambda X_t(u) - (\lambda^2 + \mu^2 + 2)t/2} W_\infty^{\lambda, \mu}(u),
\end{equation}
where, conditionally on $\mathcal{F}_t$, the random variables $\big(W_\infty^{\lambda, \mu}(u) \midsemi u \in \mathcal{N}_t^+\big)$ are independent and share the distribution of $W_\infty^{\lambda, \mu}$ under $\P_{Y_t(u)}$.
The proof is analogous to that of \eqref{eqn:smoothing}.
In fact, when $\lambda^2 + \mu^2 < 2$, \eqref{eq:newSmoothing} and Proposition~\ref{prop:mg-convergence} imply that
\begin{equation}
  \label{eq:super-survival}
  \{ W_\infty^{\lambda, \mu} > 0 \} = \{\mathcal{N}_t^+ \neq \emptyset \text{ for all }  t \geq 0\} \quad \P_y\text{-a.s.}
\end{equation}
Using the fixed point equation \eqref{eq:newSmoothing}, we show that the Laplace transform of $W_\infty^{\lambda, \mu}$ corresponds to a traveling wave of \eqref{eq:KPP}.
\begin{proposition}
  \label{prop:super}
  For all $\lambda,\mu > 0$ such that $\lambda^2 + \mu^2 < 2$, the function
  \begin{equation*}
    \Phi_{\lambda,\mu} (x,y) \coloneqq 1 - \E_y \exp\left(-\e^{-\lambda x} W_\infty^{\lambda, \mu}\right)
  \end{equation*}
  is a traveling wave of \eqref{eq:KPP} with speed $\frac{\lambda^2 + \mu^2 + 2}{2 \lambda} > \sqrt{2}$.
\end{proposition}
\begin{proof}
  For all $(x, y) \in \H$ and $t > 0$, \eqref{eq:newSmoothing} and the tower property for $\m{F}_t$ yield
  \begin{equation*}
    \Phi_{\lambda,\mu}(x,y) = 1 - \E_y\left(\prod_{u \in \mathcal{N}_t^+} \left[ 1 - \Phi_{\lambda,\mu}\big(c t + x - X_t(u),Y_t(u)\big)\right] \right),
  \end{equation*}
  where $c = \frac{\lambda^2 + \mu^2 + 2}{2\lambda}$.
  Note that this speed is supercritical:
  \begin{equation*}
    c > \inf_{\R_+} \frac{\lambda^2 + 2}{2 \lambda} = \sqrt{2}.
  \end{equation*}
  Using Proposition~\ref{prop:McK}, we deduce that $\Phi_{\lambda,\mu}(x - c t,y)$ solves \eqref{eq:KPP}, so $\Phi$ itself solves \eqref{eq:main}.
  Moreover, $W_\infty^{\lambda,\mu}$ is not identically zero, so $\Phi_{\lambda, \mu}$ is neither $0$ nor $\varphi$.
  Therefore $\Phi_{\lambda, \mu}$ is a traveling wave of speed $c$.
\end{proof}
Provided Theorem~\ref{thm:unique} holds, the proof of Theorem~\ref{thm:construction} is now complete.
\begin{proof}[Proof of Theorem~\textup{\ref{thm:construction}}]
  Using Lemma~\ref{lem:PhiSatisfiesTW}, we observe that $\Phi$ is a traveling wave on $\H$ with speed $c_*$.
  Using Theorem~\ref{thm:unique}, we deduce that it is the unique minimal-speed traveling wave, up to translation in $x$.
  The second part of Theorem~\ref{thm:construction} is Proposition~\ref{prop:super}.
\end{proof}
Recall the quarter-disk
\begin{equation*}
  \m{Q} \coloneqq \{(\lambda, \mu) \in \R_+^2 : \lambda^2 + \mu^2 < 2\}.
\end{equation*}
We have now associated a traveling wave $\Phi_{\lambda, \mu}$ with every point $(\lambda, \mu) \in \m{Q}$.
This correspondence extends in some manner to the boundary $\partial \m{Q}$.
This is not the main aim of the paper, so we describe the extension informally.

For every $(\lambda,\mu) \in \bar{\m{Q}}$, one can associate a random variable $H^{\lambda,\mu}$ such that
\begin{equation*}
  v_{\lambda,\mu} \colon (t,x,y) \mapsto 1 - \E_y \exp\left(- \e^{-\lambda x + \frac{\lambda^2 + \mu^2 + 2}{2}t} H^{\lambda,\mu}\right)
\end{equation*}
solves the KPP equation \eqref{eq:KPP}.
When $(\lambda,\mu) \in \m{Q}$, Proposition~\ref{prop:super} allows us to take $H^{\lambda,\mu} = W_\infty^{\lambda,\mu}$.

If $\mu \in (0, \sqrt{2})$ and $\lambda = 0$, we can extend the definition of $W^{\lambda, \mu}$ to $\lambda = 0$ and set $H^{0,\mu} = W_\infty^{0,\mu}$.
Then $v_{0,\mu}$ is an entire solution of \eqref{eq:KPP} depending on $t$ and $y$ alone.
At large negative times, this solution resembles a one-dimensional traveling wave in $y$ of speed $\tfrac{\mu^2 + 2}{2 \mu}$ moving down from a large height toward the $x$-axis.
The wave ``reaches'' the $x$-axis at unit time and converges to the steady state $\varphi$ uniformly in $y$ as $t \to \infty$.

If $\lambda \in [0, \sqrt{2})$ and $\mu = 0$, we can construct
\begin{equation*}
  H^{\lambda,\mu} \coloneqq \lim_{t \to \infty} \sum_{u \in \mathcal{N}_t^+} Y_t(u) \e^{\lambda X_t(u) - (\lambda^2 + 2)t/2} \quad \text{a.s. and in $L^1.$}
\end{equation*}
Then if $\lambda \in (0, \sqrt{2})$, $v_{\lambda,0}(0, x, y)$ is a traveling wave in $\H$.
One can can show that its level sets behave similarly to those of $\Phi$ described in Theorem~\ref{thm:asymptotics}.
That is, as $y \to \infty$, the level sets become asymptotically vertical with a logarithmic offset.

In the degenerate case $\lambda = 0$, $v_{0,0}$ shares some qualitative properties with the solutions $v_{0,\mu}$ described above.
It is an entire solution of \eqref{eq:KPP} depending on $t$ and $y$ alone.
At large negative times, $v_{0,0}$ resembles an exponentially-stretched profile moving down from a great height at an exponential rate.
The profile reaches the $x$-axis at unit time and converges uniformly to $\varphi$ as $t \to \infty$.
We can thus view $v_{0,0}$ as an ``infinite-speed'' limit of $v_{0, \mu}$ as $\mu \searrow 0$.

Finally, for $(\lambda,\mu)$ on the circle $\lambda^2 + \mu^2 = 2$, the additive martingale $W^{\lambda,\mu}$ must be replaced by a derivative-type martingale.
This corresponds to a traveling wave whose levels sets are inclined at angle $\arctan(\mu/\lambda)$ far from the boundary.
In this regime, the wave resembles a rotation of  the minimal-speed one-dimensional wave $w_{c_*}$.
The minimal-speed wave $\Phi$ constructed above corresponds to the special case $\lambda = \sqrt{2}$ and $\mu = 0$.

\section{Structure and tameness for minimal-speed waves}
\label{sec:PDE}

In this section, we use analytic methods to constrain an arbitrary minimal-speed traveling wave $\Psi$ on $\H^d$.
We first show that $\Psi$ is decreasing in $x$, increasing in $y$, and constant in $\tbf{x}'$.
It follows that $\Psi$ is essentially two-dimensional and we can restrict our attention to the half-\emph{plane} $\H^2$.
We then prove a sharp upper bound on the tail of $\Psi$ where $x \gg 1$.
This bound is termed ``tameness'' in the probabilistic literature; it plays a crucial role in our subsequent probabilistic arguments.

We collect the main results of this section in the following proposition.
Recall that we denote coordinates on $\H^d = \R \times \R^{d-2} \times \R_+$ by $(x, \tbf{x}', y)$.
\begin{proposition}
  \label{prop:a-priori}
  Let $\Psi$ be a traveling wave on $\H^d$ of speed $c_*$.
  Then $\Psi$ is independent of $\tbf{x}'$ and satisfies $0 < \Psi < \varphi$, $\partial_x \Psi < 0$, and $\partial_y \Psi > 0$.
  The limits $\Psi(-\infty, \anon) = \varphi$ and $\Psi(+\infty, \anon) = 0$ hold locally uniformly in $y$.
  Moreover, there exists $C > 0$ such that 
  \begin{equation}
    \label{eq:tame}
    \Psi(x, \tbf{x}', y) \leq C (1 + x_+)y \e^{-\sqrt{2} x} \ForAll (x,\tbf{x}', y) \in \H^d.
  \end{equation}
\end{proposition}
\noindent
In light of Theorem~\ref{thm:tail}, the tail bound \eqref{eq:tame} is sharp up to the constant factor wherever $x > \tfrac{1}{\sqrt{2}} \log_+ y$.

In the following, we use the notation $f \lesssim g$ when $f \leq C g$ for some universal constant $C \in \R_+$.
Likewise, $f \lesssim_\al g$ indicates that the constant $C$ can depend on the parameter $\al$.

\subsection{Strategy}
We begin by adapting maximum-principle arguments of Hamel and Nadirashvili~\cite{HN} to show that $\Psi$ has the expected monotonicity: $\partial_x\Psi < 0$, $\partial_y \Psi > 0$, and $\nab_{\tbf{x}'} \Psi = 0$.
As a result, $\Psi$ does not depend on $\tbf{x}'$ and the problem reduces to two dimensions.
It also follows that $\Psi$ approaches $\varphi$ on the left and $0$ on the right.
These regimes are separated by a smooth level set $\{\Psi = 1/2\}$ that coincides with the graph of a uniformly smooth function $x = \sigma(y)$, at least away from $\partial \H$.
Using a uniqueness result on the whole plane $\R^2$, we can show that $\sigma' \to 0$ as $y \to \infty$.
That is, the level set $\{\Psi = 1/2\}$ is asymptotically vertical far from $\partial \H$.

The remainder of the argument combines comparison methods with potential theory.
Counterintuitively, the comparison portion is based on a family of compactly supported \emph{subsolutions} that ensure that $\Psi$ roughly decays like $\e^{-\sqrt{2} x}$ where $x > \sigma(y)$.
We exploit this loose form of regularity in our potential theoretic arguments.
In the following discussion, we focus on $x,y>1$; the rest of the half-plane can be handled easily.

Using the aforementioned exponential decay, we show that $\Theta \coloneqq \e^{\sqrt{2} x}\Psi$ is nearly harmonic on the domain $\{x > 6 \sigma(y)\}$.
Because $\sigma$ is sublinear, this domain is similar to a quarter-plane and can be conformally mapped there with $r^{\smallO(1)}$ distortion, where $r \coloneqq \sqrt{x^2 + y^2}$ denotes the radial coordinate.
On the quarter-plane, explicit analysis based on the Herglotz representation theorem shows that positive harmonic functions with suitable boundary data grow at most quadratically in $r$.
Composing with our conformal map, we see that $\Theta$ grows at most like $r^{2 + \smallO(1)}$.
This polynomial bound implies that $\sigma$ grows no faster than logarithmically in $y$.

This additional quantitative information implies that $\{x > 6 \sigma(y)\}$ can be conformally mapped to the quarter-plane with \emph{bounded} distortion.
Using our quadratic bound on the quarter-plane, we find $\Theta \lesssim r^2$.
To conclude, we observe that $\Theta \lesssim (x+1)y$ on the rays $\{x=0\},$ $\{x = y\}$, and $\{y = 0\}$ while $\Theta$ grows no more than quadratically in the interior of the acute sectors $\{x > y\}$ and $\{x < y\}$.
The Phragm\'en--Lindel\"of principle thus allows us to extend the estimate $\Theta \lesssim (x+1)y$ from the boundaries of the sectors to the sectors themselves, and hence to the entire quarter-plane.

\subsection{Monotonicity and structure}
We begin with a general traveling wave $\Psi$ on $\H^d$.
We show that $\Psi$ lies between the two one-dimensional steady states.
\begin{lemma}
  \label{lem:order}
  If $\Psi$ is a traveling wave on $\H^d$, then $0 < \Psi < \varphi$.
\end{lemma}
\begin{proof}
  Let $(\s{P}_t)_{t \geq 0}$ denote the semigroup corresponding to the parabolic evolution
  \begin{equation}
    \label{eq:parabolic-half-line}
    \partial_t W = \frac{1}{2} \Delta W + W - W^2
  \end{equation}
  on $\R_+$ with Dirichlet boundary data.
  That is, if $W(t, y)$ solves \eqref{eq:parabolic-half-line} on $\R_+$ with $W|_{\partial \R_+} = 0$ and $W(0, \anon) = W_0$, then $(\s{P}_tW_0)(y) \coloneqq W(t, y)$.
  
  By Definition~\ref{def:TW}, there exists $M \in \R_+$ such that
  \begin{equation}
    \label{eq:bounded}
    0 \leq \Psi \leq M
  \end{equation}
  Recall that $\varphi$ from \eqref{eq:steady} is the unique bounded positive Dirichlet steady state on $\R_+$.
  Because $M$ is a supersolution of \eqref{eq:parabolic-half-line}, its evolution $\s{P}_t M$ is decreasing in $t$ and thus has a nonnegative bounded limit $\s{P}_\infty M$ solving \eqref{eq:steady}.
  Comparison and the hair trigger effect from~\cite[Theorem~1.3(A)]{BeG} imply that $\s{P}_\infty M \geq \s{P}_\infty 1 = \varphi.$
  Since $\varphi$ is the unique positive bounded solution of \eqref{eq:steady}, we in fact have $\s{P}_\infty M = \varphi$.
  (We do not apply Theorem~1.3(A) of~\cite{BeG} directly to $M$ because the theorem assumes $u_0 \leq 1$.
  As the above argument shows, this hypothesis can be relaxed to boundedness.)
  Using \eqref{eq:bounded}, the comparison principle thus implies that
  \begin{equation*}
    0 \leq \Psi \leq \s{P}_\infty M = \varphi.
  \end{equation*}
  Since $\Psi$ is neither $0$ nor $\varphi$, the lemma follows from the strong maximum principle.
\end{proof}
We now show that \emph{minimal-speed} waves have the expected monotonicity.
Our argument follows the proof of Lemma~5.1 in~\cite{HN}, which establishes the analogous result in the whole space.
\begin{proposition}
  \label{prop:monotone}
  Let $\Psi$ be a traveling wave on $\H^d$ of speed $c_* = \sqrt{2}$.
  Then $\partial_x \Psi < 0$, $\partial_y \Psi > 0$, and $\nab_{\tbf{x}'} \Psi = 0$.
  Moreover, $\partial_x \log \Psi \geq -\sqrt{2}$.
\end{proposition}
\begin{proof}
  Take $\theta = (\theta_x, \theta_{\tbf{x}'}, \theta_y) \in S^{d-1}$ such that $\theta_y \geq 0$.
  We write $\partial_\theta \coloneqq \theta \cdot \nab$ for the derivative in direction $\theta$.
  We suppose that $\inf \partial_\theta \Psi < 0$.
  We show that this implies that $\theta_x > 0$; the desired bounds follow.

  Define
  \begin{equation*}
    v \coloneqq \frac{\partial_\theta \Psi}{\Psi} = \partial_\theta \log \Psi.
  \end{equation*}
  When $y \geq 1$, Schauder estimates imply that $v$ is uniformly bounded.
  On $\partial\H^d$, we have $\Psi = 0$ and hence $\partial_x \Psi = \nab_{\tbf x'} \Psi = 0$ while the Hopf lemma yields $\partial_y \Psi > 0$.
  Because $\theta_y \geq 0$, it follows from elliptic estimates up to the boundary that $v_-$ is uniformly bounded where $0 < y < 1$.
  Thus $v_-$ is uniformly bounded.
  Because $\partial_\theta \Psi < 0$ somewhere, we have
  \begin{equation*}
    \inf v = -m
  \end{equation*}
  for some $m \in \R_+$.
  Thus there exists $(\tbf{x}_n)_{n \in \N} \subset \H^d$ such that $v(\tbf{x}_n) \to - m$ as $n \to \infty$.
  Define $\Psi_n \coloneqq \Psi(\anon + \tbf{x}_n)$ and $v_n \coloneqq v(\anon + \tbf{x}_n)$ for $n \in \N$.

  We consider several cases.
  First suppose $\Psi_n$ does not vanish locally uniformly as $n \to \infty$ and $\limsup_{n \to \infty} y_n > 0$.
  We restrict to a subsequence with $\inf_n y_n > 0$.
  Schauder estimates allow us to extract subsequential limits $y_\infty \in (0, \infty]$, $\Psi_\infty \not\equiv 0$, and $v_\infty$ of $(y_n, \Psi_n, v_n)_{n \in \N}$ such that $v_\infty(0) = -m = \min v_\infty$.
  One can easily check that these satisfy
  \begin{equation}
    \label{eq:ratio}
    \frac{1}{2} \Delta v_\infty + \frac{\nab \Psi_\infty}{\Psi_\infty} \cdot \nab v_\infty + c_* \partial_x v_\infty - \Psi_\infty v_\infty = 0
  \end{equation}
  in the domain $\H^d - y_\infty \tbf{e}_y$ under the convention that $\H^d - \infty \tbf{e}_y = \R^d$.
  Since $v_\infty$ achieves its minimum at the origin, $\nab v_\infty(0) = 0$ and $\Delta v_\infty(0) \geq 0$.
  Then $v_\infty(0) < 0$ and \eqref{eq:ratio} imply that $\Psi_\infty(0) = 0$, contradicting the strong maximum principle.

  Still assuming $\Psi_n$ does not vanish locally uniformly in the limit, suppose $y_n \to 0$.
  Boundary elliptic estimates imply that $v \to \infty$ locally uniformly as $y \to 0$ if $\theta_y > 0$.
  From the definition of $\tbf{x}_n$, we must have $\theta_y = 0$ in this case.
  
  Now $v_\infty$ satisfies \eqref{eq:ratio} on $\H^d$.
  This case is more delicate because $\Psi_\infty|_{\partial \H^d} = 0$, so \eqref{eq:ratio} seems singular at the boundary.
  To resolve this, let $\Lambda \coloneqq \partial_y \Psi_\infty|_{y = 0} > 0$, which is a function of $(x, \tbf{x}')$.
  Evaluating \eqref{eq:main} at $y = 0$ and taking a limit, we see that $\partial_y^2 \Psi_\infty|_{y = 0} = 0$.
  By Taylor's theorem, there exist a nonempty connected neighborhood $U$ of $0$ in $\closure$ and $\Gamma \in \m{C}^\infty(U)$ such that
  \begin{equation}
    \label{eq:Taylor}
    \Psi_\infty = \Lambda y + \Gamma y^3 \quad \text{in } U.
  \end{equation}
  We now consider the advection term in \eqref{eq:ratio}.
  The ratio $(\nab_{x,\tbf{x}'}\Psi_\infty)/\Psi_\infty$ is bounded on $U$, so the only component of concern is $\partial_y \Psi_\infty/\Psi_\infty$.
  Indeed, $\partial_y \Psi_\infty > 0$ while $\Psi_\infty = 0$ on $\partial\H^d$.
  However, we can use \eqref{eq:Taylor} to compute
  \begin{equation}
    \label{eq:deriv-ratio}
    \partial_y v_\infty = \partial_y\left(\frac{\partial_\theta\Psi_\infty}{\Psi_\infty}\right) = 2 \Lambda^{-2}(\Lambda \partial_\theta \Gamma - \Gamma \partial_\theta \Lambda) y + \m{O}(y^2) \quad \text{in }U.
  \end{equation}
  We then find
  \begin{equation*}
    \frac{\partial_y \Psi_\infty}{\Psi_\infty} \partial_y v_\infty = 2 \Lambda^{-2}(\Lambda \partial_\theta \Gamma - \Gamma \partial_\theta \Lambda) + \m{O}(y) \quad \text{in }U.
  \end{equation*}
  Likewise,
  \begin{equation*}
    \partial_y^2 v_\infty = 2 \Lambda^{-2}(\Lambda \partial_\theta \Gamma - \Gamma \partial_\theta \Lambda) + \m{O}(y) \quad \text{in }U.
  \end{equation*}
  We thus conclude that
  \begin{equation*}
    \frac{\partial_y \Psi_\infty}{\Psi_\infty} \partial_y v_\infty = \partial_y^2 v_\infty + \m{O}(y) \quad \text{in }U
  \end{equation*}
  Thus this singular first-order term acts like a regular second-order term.
  Perhaps after shrinking $U$, we are free to assume that $\abs{v_\infty} \geq m/2$ on $U$.
  Then we can write \eqref{eq:ratio} as
  \begin{equation}
    \label{eq:ratio-desingular}
    \left(\frac{1}{2} \Delta + \partial_y^2\right) v_\infty + \left(\frac{\partial_x \Psi_\infty}{\Psi_\infty} + c_*\right) \partial_x v_\infty + \frac{\nab_{\tbf{x}'} \Psi_\infty}{\Psi_\infty} \cdot \nab_{\tbf{x}'} v_\infty - \big[\Psi_\infty  + \m{O}(y)\big]v_\infty = 0.
  \end{equation}
  The coefficients in this operator are bounded and $\frac{1}{2}\Delta + \partial_y^2$ is (uniformly) elliptic.
  Consider \eqref{eq:ratio-desingular} at the origin.
  There $v_\infty$ achieves its negative minimum and \eqref{eq:deriv-ratio} implies that $\partial_y v_\infty(0) = 0$.
  By the Hopf lemma, $v \equiv -m$ in $U$.
  Using this in \eqref{eq:ratio}, we obtain $\Psi_\infty = 0$ in $U$, which contradicts the strong maximum principle, as $\Psi_\infty \not \equiv 0$ by hypothesis.

  The above contradictions imply that $\Psi$ vanishes locally uniformly along $(\tbf{x}_n)_{n \in \N}$.
  Again suppose $\limsup_{n \to \infty} y_n > 0$.
  We define
  \begin{equation*}
    w_n \coloneqq \frac{\Psi(\anon + \tbf{x}_n)}{\Psi(\tbf x_n)} \e^{\sqrt{2} x}
  \end{equation*}
  and extract subsequential limits $y_\infty \in (0, \infty]$ and $w_\infty$.
  The latter satisfies $w_\infty(0) = 1$ by construction.
  Because $\Psi(\tbf{x}_n) \to 0$, we are in a linear regime and $w_\infty$ is harmonic:
  \begin{equation*}
    \Delta w_\infty = 0 \quad \text{in } \H^d - y_\infty \tbf{e}_y.
  \end{equation*}
  Moreover, when $y_\infty < \infty$, $w_\infty = 0$ on the boundary $y = -y_\infty$.
  Positive harmonic functions satisfying the Dirichlet condition are unique up to scaling (this follows from the representation formula (13) in~\cite{Rudin}, for example).
  Since ${w_\infty(0) = 1}$, we can identify $w_\infty = \frac{y + y_\infty}{y_\infty}$ if $y_\infty < \infty$ and $w_\infty = 1$ if $y_\infty = \infty$.
  In each case, $\partial_\theta w_\infty(0) \geq 0$ because $\theta_y \geq 0$.
  On the other hand, $v(\tbf{x}_n) \to -m$ implies that
  \begin{equation*}
    0 \leq \partial_\theta w_\infty(0) = -m + \sqrt{2} \theta_x.
  \end{equation*}
  We conclude that $\theta_x > 0$ and $m \leq \sqrt{2} \theta_x$.

  Finally, suppose $y_n \to 0$.
  Then we define
  \begin{equation*}
    \ti w_n \coloneqq \frac{\Psi(\anon + \tbf{x}_n)}{\Psi(\tbf{x}_n + \tbf{e}_y)} \e^{\sqrt{2}x}
  \end{equation*}
  and extract a subsequential limit as above.
  We have $\ti w_\infty(\tbf{e}_y) = 1$, $\Delta \ti w_\infty = 0$, and $\ti w_\infty|_{\partial \H^d} = 0$.
  It follows that $\ti w_\infty = y$ and $\partial_\theta \ti w_\infty(0) \geq 0$.
  Reasoning as above, we again obtain $\theta_x > 0$ and $m \leq \sqrt{2} \theta_x$.

  We have now shown that $\partial_\theta \Psi \geq 0$ whenever $\theta_y \geq 0$ and $\theta_x \leq 0$.
  It follows that $\partial_x \Psi \leq 0$ and $\partial_y\Psi \geq 0$.
  If we take $\theta_x = \theta_y = 0$,
  \begin{equation*}
    0 \leq \partial_{-\theta} \Psi = -\partial_\theta \Psi \leq 0.
  \end{equation*}
  That is, $\partial_\theta \Psi = 0$, meaning $\nab_{\tbf{x}'} \Psi = 0$.
  Finally, $\Psi$ cannot be constant in $y$, so the strong maximum principle implies that $\partial_y \Psi > 0$.
  Similarly, if $\partial_x \Psi = 0$, then it is a bounded positive steady state of the KPP equation on the half-line.
  The unique such solution is $\varphi$, and we have assumed that $\Psi \not\equiv \varphi$.
  It follows that $\Psi$ is nonconstant in $x$ as well, so $\partial_x \Psi < 0$.

  Finally, we showed above that $m = m(\theta) \leq (\sqrt{2} \theta_x)_+$.
  Taking $\theta = \tbf{e}_x$ so $\theta_x = 1$, we see that $m \leq \sqrt{2}$.
  Recalling the definition of $m$, we have
  \begin{equation*}
    \inf \partial_x \log \Psi \geq -\sqrt{2}
  \end{equation*}
  as claimed.
\end{proof}
Since $\Psi$ is constant in $\tbf{x}'$, we can drop those variables.
In the remainder of the paper, we assume $d = 2$, so $\Psi$ is a traveling wave on $\H \coloneqq \R \times \R_+$ and thus a function of $(x,y)$.
\begin{corollary}
  \label{cor:limits}
  The following limits hold locally uniformly in $\m{C}^1$ in $y \in [0, \infty)$\textup{:}
  \begin{equation*}
    \lim_{x \to -\infty} \Psi(x, y) = \varphi(y) \And  \lim_{x \to +\infty} \Psi(x, y) = 0.
  \end{equation*}
\end{corollary}
\begin{proof}
  Because $\Psi$ is bounded and monotone in $x$, the limits $\lim_{x \to \pm \infty} \Psi$ exist and are bounded steady states of the KPP equation on $\R_+$.
  Moreover, $\Psi|_{x = -\infty} > \Psi|_{x = +\infty}$.
  The only two bounded steady states on $\R_+$ are $\varphi$ and $0$.
  The corollary follows from the uniform continuity of $\Psi$.
\end{proof}
We next consider the behavior far from the boundary.
In the following, we use the notation $\smallO_s(y)$ to indicate a function $f(s, y)$ such that $f(s, y)/y \to 0$ as $y \to \infty$ pointwise (but not necessarily uniformly) in $s$.
Recall that $w_{c_*}$ is the unique minimal-speed one-dimensional traveling wave satisfying \eqref{eq:wave-constant}.
\begin{lemma}
  \label{lem:level}
  For all $s \in (0, 1)$, the level set $\Psi^{-1}(s)$ can be expressed as the (rotated) graph $\{x = \sigma_s(y)\}$ of a locally smooth and increasing function $\sigma_s \colon (\varphi^{-1}(s), \infty) \to \R$.
  This function satisfies $\partial_y^k \sigma_s \to 0$ as $y \to \infty$ for all $k \geq 1$; in particular, $\sigma_s(y) = \smallO_s(y)$ as $y \to \infty$.
  Moreover, for all $k \geq 0$,
  \begin{equation*}
    \lim_{\ell \to \infty} \norm{\Psi(x, y) - w_{c_*}\big(x - \sigma_s(y) + w_{c_*}^{-1}(s)\big)}_{\m{C}^k(\R \times [\ell, \infty))} = 0.
  \end{equation*}
\end{lemma}
\begin{proof}
  Fix $s \in (0, 1)$.
  Due to $\partial_x\Psi < 0$ (from Proposition~\ref{prop:monotone}) and Corollary~\ref{cor:limits}, ${\Psi^{-1}(s) \subset \{y > \varphi^{-1}(s)\}}$ and for each $y \in (\varphi^{-1}(s), \infty)$, there is a unique $x \in \R$ such that $\Psi(x, y) = s$.
  Let $\sigma_s(y)$ denote this value of $x$.
  Because $\partial_x\Psi < 0,$ the implicit function theorem ensures that $\sigma_s$ is locally smooth.
  We note that the ``local'' qualifier is necessary because $\sigma_s \to -\infty$ as $y \searrow \varphi^{-1}(s)$.
  In any case, $\partial_y \Psi > 0$ implies that $\sigma_s' > 0$.

  We now consider the limiting behavior of $\sigma_s$ at infinity.
  Given a sequence $(y_n)_{n \in \N}$ tending to infinity, define $\Psi_n \coloneqq \Psi\big(\anon + (\sigma_s(y_n), y_n)\big)$.
  Then $\Psi_n(0) = s$ for all $n \in \N$.
  Taking $n \to \infty$, we can extract a locally uniform subsequential limit $\Psi_\infty$ that solves the traveling wave PDE \eqref{eq:main} on the whole space $\R^2$.
  The limit satisfies $\Psi_\infty(0) = s$, so $\Psi_\infty$ is neither identically $0$ nor $1$.
  By Theorem~1.7(i-c) in~\cite{HN}, $\Psi_\infty$ is a function of $x$ alone.
  It follows that
  \begin{equation*}
    \Psi_\infty(x, y) = w_{c_*}\big(x + w_{c_*}^{-1}(s)\big).
  \end{equation*}
  Because the limit is unique, we have $\Psi\big(x + \sigma_s(y'), y + y'\big) \to w_{c_*}\big(x + w_{c_*}^{-1}(s)\big)$ locally uniformly in $(x, y)$ as $y' \to \infty$.
  In fact, Schauder estimates imply that this convergence holds locally uniformly in $\m{C}^k$ for every $k \geq 1$.
  Hence
  \begin{equation*}
    \nab \Psi\big(\sigma_s(y), y\big) \to w_{c_*}'\big(w_{c_*}^{-1}(s)\big) \tbf{e}_x \quad \text{as } y \to \infty.
  \end{equation*}
  Since $\{x = \sigma_s(y)\}$ is a level set of $\Psi$, the above gradient is orthogonal to the tangent vector $\big(\sigma_s'(y), 1\big)$.
  Because $w_{c_*}'\big(w_{c_*}^{-1}(s)\big) \neq 0$, it follows that $\sigma_s' \to 0$ as $y \to \infty$.
  Uniform smoothness then implies that $\partial_y^k \sigma_s \to 0$ for every $k \geq 1$.
  
  We have now shown that
  \begin{equation*}
    \lim_{y \to \infty} \norm{\Psi(x, y) - w_{c_*}\big(x - \sigma_s(y) + w_{c_*}^{-1}(s)\big)}_{\m{C}_x([-L, L])} = 0
  \end{equation*}
  for all $L > 0$.
  Dini's second theorem allows us to upgrade the uniformity in $x$ from local to global~\cite[pp. 81, 270]{PS}.
  We apply it on the compactification $[-\infty, \infty]$, and must thus verify that
  \begin{equation}
    \label{eq:endpoint-limits}
    \Psi(\pm \infty, y) \to w_{c_*}(\pm \infty) \quad \text{as } y \to \infty.
  \end{equation}
  Corollary~\ref{cor:limits} states that $\Psi(-\infty, y) = \varphi(y)$ and $\Psi(+\infty, y) = 0$.
  Moreover, $\varphi(y) \to 1$ as $y \to \infty$, while $w_{c_*}(-\infty) = 1$ and $w_{c_*}(+\infty) = 0$.
  This confirms the endpoint convergence \eqref{eq:endpoint-limits}, so Dini yields
  \begin{equation*}
    \lim_{y \to \infty} \norm{\Psi(x, y) - w_{c_*}\big(x - \sigma_s(y) + w_{c_*}^{-1}(s)\big)}_{\m{C}_x(\R)} = 0.
  \end{equation*}
  Taking the limit superior in $y$, we are free to write this as
  \begin{equation*}
    \lim_{\ell \to \infty} \norm{\Psi(x, y) - w_{c_*}\big(x - \sigma_s(y) + w_{c_*}^{-1}(s)\big)}_{\m{C}(\R \times [\ell, \infty))} = 0.
  \end{equation*}
  The higher-regularity statements then follow from Schauder estimates.
\end{proof}
In the remainder of the section, we take $s = 1/2$ and let $\sigma \coloneqq \sigma_{1/2}$.
We let $\sigma_+$ denote the positive part of $\sigma$.

\subsection{A subsolution}
Our analysis of $\Psi$ hinges on the heuristic that $\Psi$ roughly decays like $\e^{-\sqrt{2} x}$ to the right of its $1/2$-level set $\{x = \sigma(y)\}$.
This decay was foreshadowed in Proposition~\ref{prop:monotone}, which states that $\partial_x \log \Psi \geq -\sqrt{2}.$
We would like to prove an almost-matching upper bound, which would state that $\Psi$ cannot decay at a rate much slower than $\sqrt{2}$.
This is impossible globally, as $\Psi$ is nearly constant in $x$ far on the left.
Thus our bounds will only hold where $\Psi$ is somewhat small.

Our key tool is a compactly-supported subsolution that varies in time.
It will move to the left while growing exponentially.
By deploying this subsolution beneath $\Psi$, we will find that $\Psi$ cannot decay too slowly, for otherwise its level set will be far to the right of the true location $\sigma$.
Thus, somewhat counterintuitively, we use a subsolution to prove an upper bound.

Our traveling wave $\Psi$ can be viewed as a solution of the KPP equation \eqref{eq:KPP} moving with velocity $c_* \tbf{e}_x$.
That is, $\Psi(\tbf{x} - c_* t \tbf{e}_x)$ solves \eqref{eq:KPP}.
Our left-moving subsolution is based on a compactly-supported subsolution of \eqref{eq:KPP} that moves at a slower speed but grows exponentially in time.

Let
\begin{equation*}
  \m{A} \coloneqq \partial_t - \frac{1}{2} \Delta - 1
\end{equation*}
denote the parabolic operator corresponding to the linearization of \eqref{eq:KPP} about $0$.
Given $c > 0$, let $(\m{S}_c z)(t, \tbf{x}) \coloneqq \e^{-c(x-ct)} z(t, \tbf{x} - c t \tbf{e}_x)$ denote an exponential tilt followed by a shift into the frame moving at velocity $c\tbf{e}_x$.
Then one can check that
\begin{equation}
  \label{eq:conjugate}
  \m{S}_c^{-1} \m{A} \m{S}_c = \partial_t - \frac{1}{2} \Delta + \frac{c^2 - c_*^2}{2}.
\end{equation}
Thus this tilt and shift merely change $\m{A}$ by a multiple of the identity.

Let $\psi > 0$ denote the principal Dirichlet eigenfunction of $-\frac{1}{2}\Delta$ on the unit ball $B_1 \subset \R^2$ normalized by $\psi(0) = \norm{\psi}_\infty = 1$.
We extend $\psi$ by $0$ to the entire plane $\R^2$.
Let $\mu$ denote the corresponding principal eigenvalue.
Given $R > 0$, the dilation $\psi(\anon/R)$ is the principal eigenfunction on the $R$-ball $B_R$ with principal eigenvalue $\mu/R^2$.

Recall that we are looking for a compactly-supported subsolution that moves at a speed $c < c_*$ and grows in time.
In this spirit, we will choose $c < c_*$, $\lambda > 0$, and $R > 0$ such that
\begin{equation*}
  \e^{\lambda t} \phi\left(\frac{\tbf{x}}{R}\right)
\end{equation*}
lies in the nullspace of $\m{S}_c^{-1} \m{A} \m{S}_c$.
In light of \eqref{eq:conjugate}, this is equivalent to
\begin{equation}
  \label{eq:dispersion}
  \lambda + \frac{\mu}{R^2} - \frac{c_*^2 - c^2}{2} = 0.
\end{equation}
Ultimately, we wish to show that solutions $\Psi$ of \eqref{eq:main} decay like $\e^{-\sqrt{2} x}$.
When we deploy our subsolution in \eqref{eq:main}, it will move to the left at the relative speed $\eps \coloneqq c_* - c > 0$.
In time $1$, it will move distance $\eps$ to the left and grow by a factor of $\e^{\lambda}$.
We can interpret this as \emph{spatial} decay at rate $\lambda/\eps$.
If we want to prove exponential decay of rate $\sqrt{2}$, we want $\lambda/\eps$ to be close to $\sqrt{2}$.
Rearranging the dispersion relation \eqref{eq:dispersion}, we want
\begin{equation}
  \label{eq:sub-condition}
  1 \gg \sqrt{2} - \frac{\lambda}{\eps} = \frac{\mu}{\eps R^2} + \frac{\eps}{2}.
\end{equation}
Thus to obtain the bounds we desire, we must use a large radius (and a very flat eigenfunction), weak exponential growth, and a speed slightly slower than $c_*$.
We choose $R$ and $\lambda$ so that the two terms on the right of \eqref{eq:sub-condition} are equal.
Expressing our parameters in terms of $\eps$, we choose
\begin{equation}
  \label{eq:params}
  c_\eps \coloneqq c_* - \eps, \quad \lambda_\eps \coloneqq \eps (c_*- \eps) = \eps c_\eps, \And R_\eps \coloneqq \frac{\sqrt{2\mu}}{\eps}.
\end{equation}
Applying $\m{S}_c$, we see that
\begin{equation*}
  v_\eps(t, \tbf{x}) \coloneqq \exp\left[\lambda_\eps t - c_\eps(x - c_\eps t + R_\eps)\right] \phi\left(\frac{\tbf{x} - c_\eps t \tbf{e}_x}{R_\eps}\right)
\end{equation*}
satisfies $\m{A}v_\eps = 0.$
That is, $v_\eps$ solves the linearization of \eqref{eq:KPP} about $0$.
We must now account for the nonlinear absorption in the full equation.
Since our solutions of \eqref{eq:KPP} lie between $0$ and $1$, $v_\eps$ is certainly a poor approximate solution when it exceeds $1$.
Thus in practice, we use a small multiple $\al v_\eps$ on a time interval that ensures that $\al v_\eps \leq 1$, namely $0 \leq t \leq \lambda_\eps^{-1} \log \al^{-1}$.
To handle the nonlinear absorption on this interval, we multiply $\al v_\eps$ by a time-dependent factor $b(t) \leq 1$.
A simple computation shows that $\al b v_\eps$ is a subsolution of the full equation \eqref{eq:KPP} provided
\begin{equation*}
  \dot b \leq - \al b^2 \e^{\lambda_\eps t}.
\end{equation*}
Taking $b(0) = 1$ and solving the corresponding ODE, we choose
\begin{equation*}
  b_\eps^\al(t) \coloneqq \big[1 + \al \lambda_\eps^{-1} \big(\e^{\lambda_\eps t} - 1\big)\big]^{-1}.
\end{equation*}
We observe that
\begin{equation*}
  b_\eps^\al(t) > \frac{1}{1 + \lambda_\eps^{-1}} \quad \text{when } \al \e^{\lambda_\eps t} = 1.
\end{equation*}
It only remains to shift into the $c_* \tbf{e}_x$-moving frame.
We define
\begin{equation*}
  w_\eps^\al(t, \tbf{x}) \coloneqq \al b_\eps^\al(t) \exp\left[\lambda_\eps t - c_\eps(x + \eps t + R_\eps)\right] \phi\left(\frac{\tbf{x} + \eps t \tbf{e}_x}{R_\eps}\right).
\end{equation*}
Then for all $\al \in [0, 1]$ and $t \in [0, \lambda_\eps^{-1} \log \al^{-1}]$, $w_\eps^\al$ is a subsolution of the parabolic traveling wave equation
\begin{equation*}
  \partial_t W = \frac{1}{2} \Delta W + c_* \partial_x W + W - W^2.
\end{equation*}
Assuming $\eps \leq 2^{-1/2}$ and $t \leq \lambda_\eps^{-1} \log \al^{-1}$, we have
\begin{equation}
  \label{eq:sub-bds}
  \frac{\eps \al}{3} \e^{-2c_\eps R_\eps} \e^{\lambda_\eps t} \phi\left(\frac{\tbf{x} + \eps t \tbf{e}_x}{R_\eps}\right) \leq w_\eps^\al(t, \tbf{x}) \leq \al \e^{\lambda_\eps t} \tbf{1}_{B_{R_\eps}(-\eps t, 0)}.
\end{equation}
We use this subsolution to prove that $\Psi$ has an exponential character.
\begin{lemma}
  \label{lem:exponential}
  For all $\eps \in (0, 2^{-1/2}],$ $x \in \R$, $y \geq 1$, and $\ell \geq 0$,
  \begin{equation}
    \label{eq:exponential}
    \min\left\{\e^{(\sqrt{2} - \eps) \ell}\Psi(x + \ell, y), \, 1\right\} \lesssim_\eps \Psi(x, y) \leq \e^{\sqrt{2} \ell} \Psi(x + \ell, y).
  \end{equation}
\end{lemma}
\begin{proof}
  Using Proposition~\ref{prop:monotone}, we have
  \begin{equation*}
    \log \frac{\Psi(x + \ell, y)}{\Psi(x, y)} = \int_0^\ell \partial_x \log \Psi(x + \ell', y) \d \ell' \geq -\sqrt{2}\ell.
  \end{equation*}
  This establishes the right bound in \eqref{eq:exponential}.
  For the left bound, we take $c_\eps,\lambda_\eps,R_\eps$ as in \eqref{eq:params}.
  By Harnack (using $y \geq 1$), there exists $k_\eps$ independent of $(x, y)$ such that
  \begin{equation}
    \label{eq:Harnack}
    \Psi \geq k_\eps\Psi(x + \ell, y) \tbf{1}_{B_{R_\eps}(x+\ell, y + R_\eps)} \And \Psi(x, y) \geq k_\eps \Psi(x, y + R_\eps).
  \end{equation}
  It follows from \eqref{eq:sub-bds} that
  \begin{equation*}
    \Psi \geq w_\eps^\al(0, \anon - (x + \ell, y + R_\eps))
  \end{equation*}
  for
  \begin{equation*}
    \al = k_\eps \Psi(x + \ell, y).
  \end{equation*}
  We allow time to evolve until
  \begin{equation*}
    t_* \coloneqq \min\left\{\frac{\ell}{\eps},\, \lambda_\eps^{-1} \log \al^{-1}\right\}.
  \end{equation*}
  First suppose $t_* = \ell/\eps$.
  Rearranging the definitions of $\al$ and $t_*$, we note that in this case
  \begin{equation}
    \label{eq:handoff}
    \Psi(x + \ell, y) \e^{(\sqrt{2} - \eps)\ell} \leq k_\eps^{-1}.
  \end{equation}
  Now, the comparison principle and \eqref{eq:sub-bds} yield
  \begin{equation*}
    \Psi \geq w_\eps^\al\big(\ell/\eps, \anon - (x + \ell, y + R_\eps)\big) \geq \frac{\eps \al}{3} \e^{-2c_\eps R_\eps} \e^{(\sqrt{2} - \eps) \ell} \phi\left(\frac{\anon - (x, y + R_\eps)}{R_\eps}\right).
  \end{equation*}
  In particular,
  \begin{equation*}
    \Psi(x, y + R_\eps) \geq \frac{\eps \al}{3} \e^{-2c_\eps R_\eps} \e^{(\sqrt{2} - \eps) \ell}
  \end{equation*}
  Thus by \eqref{eq:Harnack}, we have
  \begin{align*}
    \Psi(x, y) \geq \frac{\eps k_\eps \al}{3} &\e^{-2c_\eps R_\eps} \e^{(\sqrt{2} - \eps) \ell}\\
                                              &= \frac{\eps k_\eps^2}{3} \e^{-2c_\eps R_\eps} \e^{(\sqrt{2} - \eps) \ell} \Psi(x + \ell, y) \gtrsim_\eps \e^{(\sqrt{2} - \eps) \ell}  \Psi(x + \ell, y).
  \end{align*}
  On the other hand, if $t_* = \lambda_\eps^{-1} \log \al^{-1} < \ell/\eps$, let $x_* \coloneqq x + \ell - \eps t_* > x.$
  Then the comparison principle and \eqref{eq:sub-bds} imply that
  \begin{equation*}
    \Psi(x_*, y + R_\eps) \geq  w_\eps^\al\big(t_*, x_* - (x + \ell), 0\big) \geq \frac{\eps}{3} \e^{-2 c_\eps R_\eps}.
  \end{equation*}
  Since $\Psi$ is decreasing in $x$, we have
  \begin{equation*}
    \Psi(x, y + R_\eps) \geq \Psi(x_*, y + R_\eps) \geq \frac{\eps}{3} \e^{-2 c_\eps R_\eps}.
  \end{equation*}
  Finally, \eqref{eq:Harnack} yields
  \begin{equation*}
    \Psi(x, y) \geq k_\eps \Psi(x, y + R_\eps) \geq \frac{\eps k_\eps}{3} \e^{-2 c_\eps R_\eps} \gtrsim_\eps 1.
  \end{equation*}
  Together with \eqref{eq:handoff}, these two alternatives imply the left bound in \eqref{eq:exponential}.
\end{proof}

\subsection{Potential theory}
We now examine the behavior of $\Psi$ to the right of $\sigma$ in detail.
By Lemma~\ref{lem:exponential}, $\Psi$ roughly decays like $\e^{-\sqrt{2} x}$ there.
It is helpful to remove this decay from our analysis, so we define
\begin{equation*}
  \Theta(x, y) \coloneqq \e^{\sqrt{2} x} \Psi(x, y).
\end{equation*}
Then Lemma~\ref{lem:exponential} has a particularly simple form in this context.
\begin{lemma}
  \label{lem:Theta-exp}
  Fix $\eps \in (0, 2^{-1/2}]$ and $\ell \geq 0$.
  Then if $x \geq \sigma_+(y)$,
  \begin{equation}
    \label{eq:Theta-exp}
    \Theta(x, y) \leq \Theta(x + \ell, y) \lesssim_\eps \e^{\eps \ell} \Theta(x, y).
  \end{equation}
\end{lemma}
In particular, $\Theta$ grows subexponentially in $x$ where $x \geq \sigma_+(y)$.
\begin{proof}
  We first observe that Proposition~\ref{prop:monotone} implies that $\partial_x\Theta \geq 0$.
  Thus we need only show the right inequality.
  
  We claim that $\Psi(z + \sigma_+(y), y) \to 0$ as $z \to \infty$ uniformly in $y$.
  To see this, fix $\delta > 0$.
  By Lemma~\ref{lem:level}, there exists $y' > 0$ such that
  \begin{equation*}
    \sup_{\R \times [y', \infty)} \abs{\Psi\big(z + \sigma(y), y\big) - w_{c_*}\big(z + w_{c_*}^{-1}(1/2)\big)} \leq \frac{\delta}{2}.
  \end{equation*}
  Since $w_{c_*}(+\infty) = 0$, there exists $L > 0$ such that
  \begin{equation*}
    \Psi(z + \sigma(y), y) \leq \delta
  \end{equation*}
  for all $z \geq L$ and $y \geq y'$.
  On the other hand, Corollary~\ref{cor:limits} states that $\Psi \to 0$ as $x \to \infty$ uniformly in $y \in [0, y']$, which proves the claim.

  Now fix $\eps \in (0, 2^{-1/2}]$ and $\ell \geq 0$.
  Let $k_\eps$ denote the implicit constant in \eqref{eq:exponential}, so that
  \begin{equation*}
    \Psi(x, y) \geq k_\eps \min\left\{\e^{(\sqrt{2} - \eps) \ell}\Psi(x + \ell, y), \, 1\right\}
  \end{equation*}
  for $y \geq 1$.
  By the uniform decay shown above, there exists $L_\eps > 0$ such that $\Psi(x, y) < k_\eps$ when $x \geq L_\eps + \sigma_+(y)$.
  In this case we must have
  \begin{equation*}
    \Psi(x, y) \geq k_\eps \e^{(\sqrt{2} - \eps) \ell}\Psi(x + \ell, y).
  \end{equation*}
  Multiplying by $\e^{\sqrt{2}x}$ and rearranging, we obtain \eqref{eq:Theta-exp} for all $x \geq L_\eps + \sigma_+(y)$ and $y \geq 1$.
  Using interior Harnack, we can extend the bound to all $x \geq \sigma_+(y)$.

  It remains to treat $y \in (0, 1]$.
  Here, boundary Harnack estimates imply that $\Psi(x, y) \asymp \Psi(x, 1) y$, where $f \asymp g$ indicates that $C^{-1}f \leq g \leq Cf$ for some $C \in [1, \infty)$.
  Using our result at $y = 1$ and multiplying by $y$, we obtain \eqref{eq:Theta-exp} for $y \in (0, 1)$ as well.
\end{proof}
Now, the tilted wave $\Theta$ solves
\begin{equation}
  \label{eq:almost-harmonic}
  -\frac{1}{2} \Delta \Theta = - F \coloneqq - \e^{-\sqrt{2} x} \Theta^2.
\end{equation}
The essential point is that $F$ decays exponentially in $x$, so $\Theta$ is ``almost harmonic.''
We use this property repeatedly to constrain $\Theta$ in the quarter-plane $\Q \coloneqq \R_+^2$.

Define the region $\Sigma \coloneqq \{x \geq 6 \sigma_+(y)\}$.
This is somewhat larger than $\{x > \sigma_+(y)\}$ for the following reason.
When $x > \sigma_+$, Lemma~\ref{lem:Theta-exp} implies that
\begin{equation*}
  \Theta(x, y) \lesssim_\eps \e^{\eps (x - \sigma_+)} \Theta(\sigma_+, y) \lesssim_\eps \exp\left[\eps(x - \sigma_+) + \sqrt{2} \sigma_+\right].
\end{equation*}
Hence
\begin{align*}
  \e^{-\sqrt{2} x} \Theta^2 \lesssim_\eps \exp\Big[2\eps(x - \sigma_+) + &2\sqrt{2} \sigma_+ - \sqrt{2} x\Big]\\
                                                                         &= \exp\left[-(\sqrt{2} - 2\eps) x + 2(\sqrt{2} - \eps) \sigma_+\right].
\end{align*}
On $\Sigma$, we have $x \geq x/2 + 3 \sigma_+$, so for $\eps \ll 1$, we have
\begin{equation}
  \label{eq:exp-F}
  F = \e^{-\sqrt{2} x} \Theta^2 \lesssim \e^{-x/2}.
\end{equation}
In the following, let $G_{\tbf{z}}^\Omega(\tbf{x})$ denote the Dirichlet Green function of $-\frac{1}{2} \Delta$ on a domain $\Omega$, so that $-\frac{1}{2} \Delta G_{\tbf{z}}^\Omega = \delta_{\tbf{z}}$ and $G_{\tbf{z}}^\Omega|_{\partial \Omega \cup \{\infty\}} = 0$ for all $\tbf{z} = (u, v) \in \Omega$.
We claim that
\begin{equation*}
  \Theta_F(\tbf{x}) \coloneqq \int_\Sigma F(\tbf{z}) G_{\tbf{z}}^\Sigma(\tbf{x}) \d \tbf{z} < \infty.
\end{equation*}
To see this, note that $\Sigma \subset \R_+ \times \R$, so by comparison $G_{\tbf{z}}^\Sigma \leq G_{\tbf{z}}^{\R_+ \times \R}$.
Moreover, we can check that  $\int_\R G_{(u, v)}^{\R_+ \times \R}(x, y) \ds v = G_u^{\R_+}(x)$.
Hence \eqref{eq:exp-F} yields
\begin{equation*}
  \Theta_F(\tbf{x}) = \int_\Sigma F(\tbf{z}) G_{\tbf{z}}^\Sigma(\tbf{x}) \d \tbf{z} \lesssim \int_{\R_+ \times \R}\e^{-u/2} G_{\tbf{z}}^{\R_+ \times \R}(\tbf{x}) \d \tbf{z} = \int_{\R_+} \e^{-u/2} G_u^{\R_+}(x) \d u.
\end{equation*}
We can explicitly compute $G_u^{\R_+}(x) = 2(x \wedge u) \leq 2u,$ which is clearly integrable against a decaying exponential.
Thus $\Theta_F$ is finite and, in fact, uniformly bounded.

Now $-\frac{1}{2} \Delta \Theta_F = F$, so by \eqref{eq:almost-harmonic}, $\bar{\Theta} \coloneqq \Theta + \Theta_F$ satisfies
\begin{equation*}
  \begin{cases}
    \Delta \bar\Theta = 0 & \text{in } \Sigma,\\
    \bar\Theta = \Theta & \text{on } \partial \Sigma.
  \end{cases}
\end{equation*}
Note that $\Theta, \Theta_F > 0$ in $\Sigma$, so $\bar\Theta > 0$ is a positive harmonic function in $\Sigma$.
We now recall that $\Sigma = \{x > 6\sigma_+(y), y > 0\}$ and $\sigma'(\infty) = 0$.
Hence at a large scale, $\Sigma$ resembles the quarter-plane $\Q \coloneqq \R_+^2$.
We thus expect $\bar \Theta$ to share its large-scale behavior with positive harmonic functions on the quarter-space.
\begin{lemma}
  \label{lem:distortion}
  There exists a conformal bijection $f \colon \Q \to \Sigma$ such that for all $\eps > 0$,
  \begin{equation*}
    \label{eq:distortion}
    \norm{\tbf{x}}^{-\eps} \lesssim_\eps \frac{\norm{f(\tbf{x})}}{\norm{\tbf{x}}} \lesssim_\eps \norm{\tbf{x}}^\eps.
  \end{equation*}
\end{lemma}
\begin{proof}
  We view $\Sigma$ and $\Q$ as subsets of the complex plane $\C$, on which we use coordinates $z = u + \iu v$.
  Then $\log$ is a conformal bijection from $\Q$ to the straight strip $\{0 < v < \pi/2\}$.
  We likewise apply $\log$ to $\Sigma$.
  Given $u > 0$, define
  \begin{equation*}
    \varsigma(u) \coloneqq \arccot\left(\frac{6 \sigma_+(y_u)}{y_u}\right),
  \end{equation*}
  where $y_u > 0$ satisfies
  \begin{equation*}
    [6 \sigma_+(y_u)]^2 + y_u^2 = \e^{2u}.
  \end{equation*}
  The height $y_u$ exists uniquely for each $u \in \R$ because $\sigma' > 0$.
  Thus $\varsigma(u) \leq \pi/2$ is the argument of the unique point on the curve $\{x = 6 \sigma_+\}$ whose radial coordinate is $\e^u$.
  It follows that $\log$ is a conformal bijection from $\Sigma$ to the curvilinear strip $S \coloneqq \{0 < v < \varsigma(u)\}$.
  Because $\sigma_+ \equiv 0$ for $y \leq \varphi^{-1}(1/2)$, we have $\varsigma \equiv \pi/2$ when $u \leq \log \varphi^{-1}(1/2)$.
  In the other direction, $\sigma = \smallO(y)$ as $y \to \infty$, so $y_u \sim \e^u$ and
  \begin{equation*}
    \frac{\pi}{2} - \varsigma(u) \sim 6 \e^{-u}\sigma(\e^u) \to 0 \quad \text{as } u \to +\infty.
  \end{equation*}
  Thus $S$ resembles the straight strip $\{0 < v < \pi/2\}$ when $\abs{u} \gg 1$.
  It is therefore reasonable to expect that there is a conformal bijection ${g \colon \{0 < v < \pi/2\} \to S}$ with low distortion at infinity.
  This is a well-studied problem in potential theory.
  Warschawski, for instance, constructs $g$ such that $g(z) \sim u$ at infinity; see Theorem~X in~\cite{Warschawski}.
  In particular,
  \begin{equation}
    \label{eq:log-distortion}
    \abs{\Re g(z) - u} = \smallO(u).
  \end{equation}
  We now define the conformal bijection $f \coloneqq \exp \circ g \circ \log \colon \Q \to \Sigma$.
  Then \eqref{eq:log-distortion} becomes \eqref{eq:distortion}.
\end{proof}
Employing this conformal map, $\bar\Theta \circ f$ becomes a positive harmonic function on the quarter-plane.
This allows us to constrain the growth of $\bar \Theta$.
In the following, let $\braket{\tbf{x}} \coloneqq (\norm{\tbf{x}}^2 + 1)^{1/2}$.
\begin{lemma}
  \label{lem:quarter}
  Let $h \in \m{C}(\bar \H)$ be a positive harmonic function on $\H$ that is nondecreasing in $\abs{x}$ on $\partial\H$.
  Then $h(\tbf{x}) \lesssim \braket{\tbf{x}}$ on $\H$.
\end{lemma}
\begin{proof}
  We rely on a representation theorem of Herglotz (due independently to Herglotz~\cite{Herglotz} and F.~Riesz~\cite{Riesz}): there exists a constant $A \geq 0$ such that
  \begin{equation}
    \label{eq:Herglotz}
    h(x, y) = Ay + \frac{y}{\pi} \int_{\R} \frac{h(t, 0) \d t}{(x - t)^2 + y^2}.
  \end{equation}
  As a consequence,
  \begin{equation}
    \label{eq:integrable}
    \int_{\R} \frac{h(t, 0)}{t^2 + 1} \d t < \infty.
  \end{equation}
  Since $h(\anon, 0)$ is nondecreasing on $\R_+$, we have
  \begin{equation*}
    \int_x^{2x} \frac{h(t, 0)}{t^2 + 1} \d t \geq h(x, 0) \int_x^{2x} \frac{\dn t}{t^2 + 1} \gtrsim \frac{h(x, 0)}{x} \ForAll x \geq 1.
  \end{equation*}
  By \eqref{eq:integrable}, the left side tends to $0$ as $x \to \infty$.
  It follows that $h(x, 0) \ll x$ as $x \to \infty$, though we will only use the weaker bound $h(x, 0) \lesssim x$.
  By symmetry,
  \begin{equation}
    \label{eq:boundary}
    h(\tbf{x}) \lesssim \braket{\tbf{x}} \quad \text{on } \partial \H.
  \end{equation}
  
  Now suppose $\abs{x} \leq m y$ for fixed $m > 0$.
  Then
  \begin{equation*}
    \frac{t^2 + y^2}{(x-t)^2 + y^2} = \frac{(t/y)^2 + 1}{(x/y - t/y)^2 + 1} \lesssim_m 1.
  \end{equation*}
  It follows that
  \begin{equation*}
    \frac{1}{\pi} \int_{\R} \frac{h(t, 0) \d t}{(x - t)^2 + y^2} \lesssim_m \int_{\R} \frac{h(t, 0)}{t^2 + y^2} \d t.
  \end{equation*}
  By \eqref{eq:integrable} and dominated convergence, the integral on the right tends to $0$ as $y \to \infty$.
  Thus the integral term in \eqref{eq:Herglotz} is negligible as we approach infinity in $\H$ from a direction that is not tangent to $\partial \H$.
  That is:
  \begin{equation}
    \label{eq:non-tangent}
    h(x, y) \sim_m Ay \quad \text{as } y \to \infty \enspace \text{if } \abs{x} \leq m y.
  \end{equation}
  In particular,
  \begin{equation}
    \label{eq:non-tangent-bound}
    h(\tbf{x}) \lesssim_m \braket{\tbf{x}} \quad \text{in } \{\abs{x} \leq m y\}.
  \end{equation}

  We now consider $h$ on $\Q$.
  Let $\zeta \colon \Q \to \H$ denote the square map $\zeta(z) \coloneqq z^2$.
  We define $g \coloneqq h \circ \zeta^{-1}$, which is a positive harmonic function on $\H$.
  By \eqref{eq:boundary} and \eqref{eq:non-tangent-bound}, $h(\tbf{x}) \lesssim \braket{\tbf{x}}$ on $\partial \Q$ and on the ray $\{x = y\}$.
  It follows that
  \begin{equation}
    \label{eq:sqrt-boundary}
    g(\tbf x) \lesssim \braket{\tbf{x}}^{1/2} \quad \text{on } \partial\H
  \end{equation}
  \emph{and} on the ray $\{x  = 0\}$, which is the image of $\{x = y\}$ under $\zeta$.
  Now $g$ must also admit a Herglotz representation, but we know that $g \ll y$ on the $y$-axis.
  By \eqref{eq:non-tangent}, we see that
  \begin{equation*}
    g(x, y) = \frac{y}{\pi} \int_{\R} \frac{g(t, 0) \d t}{(x - t)^2 + y^2}.
  \end{equation*}
  Using \eqref{eq:sqrt-boundary}, we find
  \begin{align}
    g(x, y) &\lesssim y \int_{\R} \frac{1 + \sqrt{\abs{t}} \d t}{(x - t)^2 + y^2}\nonumber\\
            &\lesssim 1 + \sqrt{y} \int_{\R} \frac{\sqrt{\abs{s + x/y}} \d s}{s^2 + 1} \lesssim 1 + \sqrt{\abs{x}} + \sqrt{y} \lesssim \braket{\tbf{x}}^{1/2}.\label{eq:sqrt}
  \end{align}
  Transferring this bound to $h = g \circ \zeta$, we find $h(\tbf{x}) \lesssim \braket{\tbf{x}}$ on $\Q$.
  A symmetric argument on the left quadrant $\R_- \times \R_+$ shows that this holds on the entire half-plane $\H$.
\end{proof}
\begin{corollary}
  \label{cor:polynomial}
  For all $\eps > 0$,
  \begin{equation}
    \label{eq:polynomial}
    \bar \Theta(\tbf{x}) \lesssim_\eps \braket{\tbf{x}}^{2 + \eps} \quad \text{in } \Sigma.
  \end{equation}
  Moreover, as $y \to \infty$, $\sigma(y) \leq \big[\sqrt{2} + \smallO(1)\big] \log y.$
\end{corollary}
\begin{proof}
  Using the square map $\zeta(z) \coloneqq z^2$ from above, $h \coloneqq \bar\Theta \circ f \circ \zeta^{-1}$ is a positive harmonic function on $\H$ that is continuous on $\bar \H$.
  Recall that $\Theta$ is increasing in $x$ and $y$.
  Since $\sigma' > 0$, it follows that $y \mapsto \Theta\big(6 \sigma_+(y), y\big)$ is increasing.
  Now $\bar\Theta = \Theta$ where $x = 6 \sigma_+(y)$, so $\Theta$ is \emph{decreasing} on $\partial \Sigma$ when its boundary is traversed in the counterclockwise direction.
  The conformal maps $f$ and $\zeta$ preserve the orientation of the boundary, so $h(x, 0)$ is decreasing in $x$.
  Also, $h(x, 0) = 0$ for $x \geq 0,$ as this portion of the boundary corresponds to $\{y = 0\} \subset \partial \Sigma$, where $\bar\Theta = \Theta = 0$.
  Together, these facts imply that $h(x, 0)$ is nondecreasing in $\abs{x}$.
  Thus by Lemma~\ref{lem:quarter}, $h(\tbf{x}) \lesssim \braket{\tbf{x}}$ in $\H$.
  It follows that $\bar \Theta \circ f = h \circ \zeta \lesssim \braket{\tbf{x}}^2$.
  Finally, Lemma~\ref{lem:distortion} yields \eqref{eq:polynomial}.

  We now turn to $\sigma$.
  By the definition of $\sigma$, Lemma~\ref{lem:Theta-exp}, and \eqref{eq:polynomial}, we have
  \begin{equation*}
    \frac{1}{2} \e^{\sqrt{2} \sigma(y)} = \Theta\big(\sigma(y), y\big) \leq \Theta\big(6 \sigma(y), y\big) \leq \bar\Theta\big(6 \sigma_+(y), y\big) \lesssim_\eps y^{2 + \eps}
  \end{equation*}
  for $y$ sufficiently large.
  It follows that $\sigma(y) \leq \big(\sqrt{2} + \eps\big)\log y + \m{O}_\eps(1)$ as $y \to \infty$ for all $\eps > 0$.
\end{proof}
Soft arguments implied that $\sigma$ is sublinear.
Using potential theory, we have now improved this to a logarithmic upper bound.
In turn, this quantitative sublinearity allows us to refine Corollary~\ref{cor:polynomial}.
\begin{lemma}
  \label{lem:quadratic}
  We have $\Theta(\tbf x) \lesssim \braket{\tbf x}^2$ on $\Q$.
\end{lemma}
\begin{proof}
  Let $\Omega \coloneqq \{x > 10 \log_+y\}$.
  By Corollary~\ref{cor:polynomial}, $\bar \Sigma \setminus \Omega$ is a bounded, and thus compact, region.
  Thus \eqref{eq:exp-F} yields $F \lesssim \e^{-x/2}$ on $\Omega$.
  We define $\Omega_F^\Omega \coloneqq \int_\Omega F(\tbf{z}) G_{\tbf{z}}^\Omega$, which is positive and uniformly bounded on $\Omega$ by the reasoning following \eqref{eq:exp-F}.
  Let $\omharm \coloneqq \Theta + \Theta_F^\Omega$, which is a positive harmonic function on $\Omega$.
  
  Following the proof of Lemma~\ref{lem:distortion}, the logarithm maps $\Omega$ to a curvilinear strip ${S = \{0 < v < \theta(u)\}}$, where $\theta(u)$ is the argument of the unique point on the curve $\{x = 10 \log_+ y\}$ of radius $\e^u$.
  Now, the logarithmic boundary of $\Omega$ allows us to conclude that
  \begin{equation*}
    \frac{\pi}{2} - \theta(u) \sim 10 u \e^{-u} \And \theta'(u) \sim 10 u \e^{-u} \quad \text{as } u \to \infty
  \end{equation*}
  while $\theta(u) \equiv \pi/2$ and $\theta'(u) \equiv 0$ for sufficiently negative $u$.
  Then Theorem~IX of~\cite{Warschawski} provides a conformal bijection $g \colon \{0 < v < \pi/2\} \to S$ such that
  \begin{equation*}
    g(z) = z + \log \lambda + \smallO(1) \quad \text{as } u \to \infty
  \end{equation*}
  for some $\lambda \in \R_+$.
  It follows that $f^\Omega \colon \Q \to \Omega$ given by $f^\Omega \coloneqq \exp \circ g \circ \log$ satisfies
  \begin{equation}
    \label{eq:low-distortion}
    \norm{f(\tbf{x})} \sim \lambda \norm{\tbf{x}} \quad \text{as } \norm{\tbf{x}} \to \infty.
  \end{equation}

  Now let $h \coloneqq \omharm \circ f^\Omega \circ \zeta^{-1}$.
  Following the proof of Corollary~\ref{cor:polynomial}, we see that $h$ satisfies the hypotheses of Lemma~\ref{lem:quarter}.
  Hence $h(\tbf{x}) \lesssim \braket{\tbf x}$.
  It again follows that $\omharm \circ f^\Omega = h \circ \zeta \lesssim \braket{\tbf{x}}^2$.
  Then \eqref{eq:low-distortion} implies that $\Theta \leq \omharm \lesssim \braket{\tbf{x}}^2$ on $\Omega$.
  Using Lemma~\ref{lem:Theta-exp}, we further have
  \begin{equation*}
    \Theta(x, y) \leq \Theta(10 \log_+y, y) \lesssim \braket{y}^2
  \end{equation*}
  on $\Q \setminus \Omega = \{0 < x < 10 \log_+ y\}$.
  The lemma follows.
\end{proof}
We can finally establish the main result of the section.
\begin{proof}[Proof of Proposition~\textup{\ref{prop:a-priori}}]
  The first parts of the proposition follow from Lemma~\ref{lem:order}, Proposition~\ref{prop:monotone}, and Corollary~\ref{cor:limits}.
  Thus it remains only to verify \eqref{eq:tame}.
  
  Lemma~\ref{lem:quadratic} and the $\m{C}^1$ regularity of $\Psi$ near $\partial \H$ imply that for all $x,y > 0$,
  \begin{align}
    \Theta(x, x) &\lesssim (x+1)x,\label{eq:mid}\\
    \Theta(0, y) &\lesssim y \wedge 1,\label{eq:Theta-left}\\
    \Theta(x, 0) &= 0.\label{eq:bottom}
  \end{align}  
  We divide $\Q$ into two sectors $\Gamma_1 \coloneqq \{x > y > 0\}$ and $\Gamma_2 \coloneqq \{0 < x < y\}$ each of opening angle $\frac{\pi}{4}$.
  Using \eqref{eq:mid}--\eqref{eq:bottom}, there exists $C > 0$ such that
  \begin{equation*}
    \Theta(x, y) \leq C(x+1)y \quad \text{on } \partial \Gamma_1 \cup \partial \Gamma_2.
  \end{equation*}
  Let $\ti\Theta(x, y) \coloneqq \Theta(x, y) - C(x+1)y$, so $\ti\Theta \leq 0$ on $\partial \Gamma_1 \cup \partial \Gamma_2$.
  By \eqref{eq:almost-harmonic}, $\ti \Theta$ is subharmonic on $\Q$.
  Moreover, Lemma~\ref{lem:quadratic} implies that $\ti\Theta \lesssim \braket{x}^2$.
  Now, there exist positive harmonic functions $h_i$ on $\Gamma_i$ such that $h_i(\tbf x) \gg \braket{\tbf x}^2$ on $\bar \Gamma_i$.
  For instance, we can use a suitable rotation of $\Re z^\al$ for any $\al \in (2, 4)$.
  Thus by the Phragm\'en--Lindel\"of principle~\cite{PL} for subharmonic functions, $\ti \Theta \leq 0$ on $\Gamma_i$ for each $i \in \{1, 2\}$.
  That is,
  \begin{equation*}
    \Theta(x, y) \leq C(x+1)y \quad \text{on } \Q.
  \end{equation*}
  Recalling that $\Psi = \e^{-\sqrt{2} x} \Theta$, we obtain \eqref{eq:tame} on $\Q$.
  On the other hand, boundary Schauder estimates imply that $\Psi \lesssim y \wedge 1$ on $\R_- \times \R_+$
  Hence $\Psi(x, y) \lesssim y \e^{-\sqrt{2} x}$ on $\R_- \times \R_+$.
  This completes the proof of the proposition.
\end{proof}

\section{Uniqueness of minimal-speed traveling waves}
\label{sec:uniqueness}

We can now establish the uniqueness portion of Theorem~\ref{thm:unique}.
\begin{proposition}
  \label{prop:uniquenessRephrased}
  If $\Psi$ is a traveling wave in $\H$ of speed $c_* = \sqrt{2}$, then there exists $\kappa>0$ such that $\Psi(x,y) = 1 - \E_y\exp \left(-\kappa \e^{-\sqrt{2}x}Z_\infty\right)$.
  In particular,
  \begin{equation*}
    \Psi(x, y) = \Phi\big(x - \tfrac{1}{\sqrt{2}} \log \kappa, y\big) \quad \text{on } \H.
  \end{equation*}
\end{proposition}
\noindent
That is, up to $x$-translation, $\Phi$ is the unique traveling wave of speed $c_*$.

The proof of Proposition \ref{prop:uniquenessRephrased} follows the method described in \cite{AlM22} to identify the fixed points of the smoothing transform.
We use the arbitrary traveling wave $\Psi$ to construct a product martingale, sometimes called the \emph{disintegration martingale}.
The tameness bound \eqref{eq:tame} allows us to associate this multiplicative martingale with a harmonic function in the quadrant with Dirichlet conditions.
Such functions are unique up to a multiplicative constant; this allows us to identify $\Psi$ as a shift of $\Phi$.
\begin{proof}
  Let $\Psi$ be a traveling wave in $\H$ of speed $\sqrt{2}$.
  Using the McKean representation (Proposition~\ref{prop:McK}), we observe that for all $(x,y) \in \H$ and $t \geq 0$,
  \begin{equation}
    \label{eqn:mcKeanRepTW}
    1- \Psi(x,y) = \E_y\left(\prod_{u \in \mathcal{N}^+_t} \big[1 - \Psi\big(\sqrt{2} t + x - X_t(u), Y_t(u)\big)\big]\right).
  \end{equation}
  Hence the branching property of the BBM implies that
  \begin{equation*}
    T_t(x) \coloneqq \prod_{u \in \mathcal{N}^+_t} \big[1 - \Psi\big(\sqrt{2} t + x - X_t(u), Y_t(u)\big)\big]
  \end{equation*}
  is a bounded martingale under law $\P_y$.
  We denote by $T_\infty(x)$ its almost sure limit.
  This is sometimes called the \textit{disintegration} of the function $\Psi$.
  
  We now introduce a shaved version of this martingale.
  Given $\alpha > 0$ and $x \in \R$, let
  \begin{equation*}
    \mathcal{N}_t^{+,\alpha}(x) \coloneqq \big\{u \in \mathcal{N}_t^+ : \ X_s(u) \leq \sqrt{2} s + x + \alpha\text{ for all } s \leq t\big\}.
  \end{equation*}
  It is a straightforward consequence of \eqref{eqn:mcKeanRepTW} that
  \begin{equation*}
    T_t^\alpha(x) \coloneqq \prod_{u \in \mathcal{N}^{+,\alpha}_t(x)} \big[1 - \Psi\big(\sqrt{2}t + x - X_t(u),Y_t(u)\big)\big]
  \end{equation*}
  is a bounded submartingale.
  Indeed, we take the multiplicative martingale $T$ and delete terms when the corresponding particle reaches the line $\sqrt{2} s + x + \al$.
  At such times $T^\al$ jumps up, and is thus a submartingale.
  As a result, this process converges $\P_y$-a.s. and in $L^1$ to a nondegenerate limit that we denote by $T_\infty^\alpha(x)$.
  
  Using the branching property, one can check that $T_\infty^\alpha$ satisfies the following almost sure recursion:
  \begin{equation}
    \label{eqn:multRecursionEquation}
    T_\infty^\alpha (x) = \prod_{u \in \mathcal{N}^{+,\alpha}_t(x)} T_\infty^{\alpha}[u]\big(\sqrt{2} t + x - X_t(u), Y_t(u)\big).
  \end{equation}
  The random variables $\big(T_\infty^{\alpha}[u] \midsemi u \in \mathcal{N}^{+,\alpha}_t(x)\big)$ are conditionally independent given $\m{F}_t$ and $T_\infty^\al[u]$ shares the law of $T_\infty^\alpha$ under $\P_{Y_t(u)}$.

  Now, Proposition~\ref{prop:a-priori} states that $\Psi \lesssim (x_++1)y \e^{-\sqrt{2}x}$.
  We use this to argue that for all $x \geq -\al$ and $y \geq 0$,
  \begin{equation}
    \label{eqn:aimTame}
    -\log T_\infty^\alpha(x) \leq 2 C \e^{-\sqrt{2} x} Z^\alpha_\infty ~\text{ a.s.}
  \end{equation}
  In the following, let
  \begin{equation*}
    \mathcal{B}_t \coloneqq \big\{u \in \mathcal{N}_t^{+,\alpha}(x) : \Psi\big(\sqrt{2} t + x - X_t(u),Y_t(u)\big) \geq \tfrac{1}{2}\big\}.
  \end{equation*}
  Noting that $-\log (1-a) \leq 2 a$ for all $a \leq 1/2$, \eqref{eq:tame} yields
  \begin{align*}
    -\log T_t^\alpha(x) &= -\sum_{u \in \mathcal{N}_t^{+,\alpha}(x)} \log\big[1 - \Psi\big(\sqrt{2} t + x - X_t(u),Y_t(u)\big)\big]\\
                        &\leq 2 C \e^{-\sqrt{2}x} [Z^\alpha_t + (1+x_+)W_t]\\
                        &\qquad\qquad - \sum_{u \in \mathcal{B}_t} \log\big[1 - \Psi\big(\sqrt{2} t + x - X_t(u),Y_t(u)\big)\big].
  \end{align*}
  Then \eqref{eqn:aimTame} follows from Propositions~\ref{prop:additive} and~\ref{prop:derivativeMartingale}, provided $\mathcal{B}_t = \emptyset$ for sufficiently large $t$.
  
  To show that $\m{B}_t$ is eventually empty, we observe that $Z$ is Cauchy in time because it converges (almost surely).
  It follows that the contribution of every individual particle becomes negligible as $t \to \infty$.
  Otherwise, branching events would cause $Z$ to jump non-negligibly at arbitrarily large times.
  Therefore
  \begin{equation}
    \label{eq:Z-terms-vanish}
    \sup_{u \in \mathcal{N}^+_t} \big[\sqrt{2} t - X_t(u)\big]_+ Y_t(u) \e^{\sqrt{2} X_t(u) - 2t} \to 0 \quad \text{a.s. on survival.}
  \end{equation}
  In more detail, if \eqref{eq:Z-terms-vanish} did not hold, then there would exist $\eps > 0$ such that with positive probability, the stopping times defined by $\tau_0 = 0$ and
  \begin{equation*}
    \tau_{n+1}\coloneqq \inf\big\{ t > \tau_n + 1 : \big[\sqrt{2} t - X_t(u)\big]_+ Y_t(u) \e^{\sqrt{2} X_t(u) - 2t} > 2 \eps \text{ for some }u \in \mathcal{N}_t^+ \big\}
  \end{equation*}
  are all finite.
  By Borel--Cantelli, with positive probability there exists a (random) subsequence $(n_k)_{k \in \N}$ such that a particle alive at some time $t \in [\tau_{n_k}, \tau_{n_{k+1}}]$ and located at some position $(x,y)$ satisfying
  \begin{equation*}
    (\sqrt{2} t -x)_+ y \e^{\sqrt{2} x - 2t}>\eps
  \end{equation*}
  splits into two children.
  In particular, at all such branching times, $Z_{t} > Z_{t-} + \eps$.
  This contradicts the Cauchy property of $Z$.

  Now recall that $W_t \to 0$ by Proposition~\ref{prop:additive}, so ${\sup_u Y_t(u) \e^{\sqrt{2} X_t(u) - 2t} \to 0}$ a.s. as $t \to \infty$.
  Thus \eqref{eq:tame} and \eqref{eq:Z-terms-vanish} yield
  \begin{align*}
    \sup_{u \in \mathcal{N}^+_t} \Psi\big(\sqrt{2} t + x - X_t(u), &Y_t(u)\big)\\[-2ex]
                                                                   &\lesssim \sup_{u \in \m{N}_t^+} \left[1 + x + \big[\sqrt{2}t - X_t(u)\big]_+\right] Y_t(u) \e^{\sqrt{2} X_t(u) - 2t} \to 0
  \end{align*}
  almost surely on survival.
  It follows that $\mathcal{B}_t$ is empty a.s. for sufficiently large $t$.
  
  Now, \eqref{eqn:aimTame} implies that $-\log T_\infty^\al$ has a first moment.
  We let $F^\al \colon \H \to \R_+$ denote its expectation:
  \begin{equation*}
    F^\alpha (x,y) \coloneqq \E_y\big[-\log T_\infty^\alpha(x)\big].
  \end{equation*}
  Fix $t > 0$.
  Using the almost sure recursion \eqref{eqn:multRecursionEquation} and the many-to-one lemma, we observe that $F^\alpha$ satisfies
  \begin{align*}
    F^\alpha (x,y) &= \E_y\left[\sum_{u \in \mathcal{N}^{+,\alpha}_t(x)} F^\alpha\big(\sqrt{2} t + x - X_t(u),Y_t(u)\big)\right]\\
                   &= \e^t \E_y\left[ F^\alpha(\sqrt{2} t + x + X_t,Y_t)\ind{ -X_s \leq \sqrt{2} s + x + \alpha, \, Y_s \geq 0 \midsemi s \leq t} \right],
  \end{align*}
  where $(X,Y)$ is a Brownian motion in $\R^2$ started from $(0,y)$ under $\P_y$.
  Applying the Girsanov transform for the Brownian motion, we obtain
  \begin{equation*}
    \e^{\sqrt{2} x}F^\alpha (x,y) = \E_y\left[ \e^{\sqrt{2}(X_t+x)} F^\alpha (x + X_t,Y_t) \ind{x + X_s \geq -\alpha, \, Y_s \geq  0\midsemi s \leq t}\right].
  \end{equation*}
  Writing $G^\alpha (x,y) \coloneqq \e^{\sqrt{2} x} F^\alpha (x,y)$, we see that for all $t > 0$ and $(x, y) \in \H$,
  \begin{equation*}
    G^\alpha (x,y) = \E_{y}\left[G^\alpha(x+X_t,Y_t) \ind{Y_s \geq 0, \, x+X_s \geq -\alpha \midsemi s \leq t}\right].
  \end{equation*}
  Using Itô's formula, we conclude  that $G^\alpha$ is a nonnegative harmonic function in the quarter-plane $\{x \geq -\alpha, y \geq 0\}$ with Dirichlet boundary data.
  Define $\zeta_\al(z) \coloneqq (z + \al)^2$, which biholomorphically maps this quarter-plane to the half-place $\H$.
  Then $G^\al \circ \zeta_\al^{-1}$ is a harmonic function on $\H$ that is continuous on $\bar\H$ and vanishes on $\partial\H$.
  By Herglotz's representation formula \eqref{eq:Herglotz}, there exists $\kappa_\alpha \in \R_+$ such that $G^\al \circ \zeta_\al^{-1} = \kappa_\al y/2$.
  Composing with $\zeta_\al$, we find $G^\alpha(x,y) = \kappa_\alpha (x + \alpha)y$ and hence
  \begin{equation*}
    F^\alpha(x,y) = \kappa_\alpha (x+\alpha) y \e^{-\sqrt{2} x}.
  \end{equation*}

  To complete the proof, we observe that for all $\alpha > 0$, almost surely
  \begin{align*}
    -\log T_\infty^\alpha(x) &= \lim_{t \to \infty} \E\big[ - \log T_\infty^\alpha(x) \mid \mathcal{F}_t\big]\\
                      &= \lim_{t\to \infty} \sum_{u \in \mathcal{N}^{+,\alpha}_t(x)} F^\alpha\big(\sqrt{2} t + x -X_t(u),Y_t(u)\big) = \kappa_\alpha \e^{-\sqrt{2} x} Z_\infty^\alpha.
  \end{align*}
  Using once again \eqref{eqn:unifUb}, for all $(x,y) \in \H$ there exists a random $\alpha_0 \in \R_+$ such that for all $\alpha > \alpha_0$ we have $T_\infty(x) = T_\infty^\alpha(x)$ and $Z_\infty = Z^\alpha_\infty$ under $\P_y$.
  This shows that $\kappa_\alpha$ (which is deterministic) is constant for sufficiently large $\alpha$.
  Writing $\kappa \coloneqq \lim_{\alpha \to \infty} \kappa_\alpha$, we obtain
  \begin{equation*}
    T_\infty(x) = \lim_{\alpha \to \infty} T_\infty^\alpha(x) = \exp\left(- \kappa \e^{-\sqrt{2}x} Z_\infty\right) \quad \P_y\text{-a.s.}
  \end{equation*}
  Finally, \eqref{eqn:mcKeanRepTW} yields
  \begin{equation*}
    \Psi(x,y) = 1 - \E_y T_\infty(x) = 1 - \E_y\exp\left(- \kappa \e^{-\sqrt{2}x} Z_\infty\right).
    \qedhere
  \end{equation*}
\end{proof}

\section{Traveling wave asymptotics}
\label{sec:asymptotics}

In this section, we study the behavior at infinity of the traveling waves $\Phi$ and $\Phi_{\lambda,\mu}$ constructed in Section~\ref{sec:construction}.
In particular, we prove Theorems~\ref{thm:asymptotics} and~\ref{thm:tail}.

In Section~\ref{subsec:mintw}, we consider the asymptotic behavior of the law of $Z_\infty$ as the initial height $y$ tends to infinity.
This determines the large-$y$ asymptotics of $\Phi$.
We take up the same question for $\Phi_{\lambda,\mu}$ in Section~\ref{subsec:hghtw}; this allows us to complete the proofs of Theorems~\ref{thm:unique} and \ref{thm:asymptotics}.
Finally, in Section~\ref{subsec:extendedAsymptotics} we use potential theory and Theorem~\ref{thm:asymptotics} to study the behavior of $\Phi$ as $x \to \infty$ and thereby prove Theorem~\ref{thm:tail}.

\subsection{The minimal-speed wave far from the horizontal axis}
\label{subsec:mintw}

In this subsection, we relate the minimal-speed half-space wave $\Phi$ to the corresponding one-dimensional wave $w_{c_*}$ defined \eqref{eqn:defw}.
\begin{proposition}
  \label{prop:tw}
  We have
  \begin{equation*}
    \lim_{y \to \infty} \Phi\left(x + \tfrac{1}{\sqrt{2}}\log y, y\right) = w_{c_*}(x)
  \end{equation*}
  uniformly in $x \in \R$.
\end{proposition}
By the definition \eqref{eq:defTW} of $\Phi$, we can equivalently determine the asymptotic properties of $Z_\infty$ under $\P_y$ as $y \to \infty$.
For this purpose, it is convenient to define a consistent family $(Z(y) \midsemi y \geq 0)$ of martingales on a single probability space $(\Omega, \mathcal{F}, \P)$ such that for all $y > 0$, $Z(y)$ has the law of $Z$ under $\P_y$.
In the remainder of this subsection, let $\big(X_t(u), Y_t(u) \midsemi u \in \mathcal{N}_t\big)$ be a branching Brownian motion in $\R^2$ started from $(0,0)$.
Given $t \geq 0$ and $y > 0$, we set
\begin{equation*}
  \mathcal{N}^y_t \coloneqq \big\{u \in \mathcal{N}_t : Y_s(u) \geq - y \, \text{ for all } \, s \leq t\big\}.
\end{equation*}
We then define
\begin{align*}
  &Z_t(y) \coloneqq \sum_{u \in \mathcal{N}^y_t} \big[\sqrt{2} t - X_t(u)\big] \big(Y_t(u) + y\big) \e^{\sqrt{2} X_t(u) - 2t}\\
  \text{and} \quad &Z_\infty(y) \coloneqq \lim_{t \to \infty} Z_t(y).
\end{align*}
Since $\big((X_t(u),Y_t(u) + y) \midsemi u \in \mathcal{N}_t^y\big)$ has the law of a branching Brownian motion in $\H$ starting from $(0,y)$, we conclude that for all $y > 0$, $Z_\infty(y)$ has the same law as $Z_\infty$ under $\P_y$.

Recall from Section~\ref{sec:preliminaries} that $D_\infty$ is the a.s. limit as $t \to \infty$ of the derivative martingale of the one-dimensional BBM $(X_t(u) \midsemi u \in \mathcal{N}_t)$.
We prove the following asymptotic for $Z_\infty(y)$ as $y \to \infty$, which implies Proposition~\ref{prop:tw}.
\begin{proposition}
  \label{prop:cvZy}
  We have
  $\displaystyle
  \lim_{y \to \infty} \frac{Z_\infty(y)}{y} = D_\infty$ in probability.
\end{proposition}
Recall that $\mathcal{H} = \sigma\big(X_s(u), u \in \mathcal{N}_s \midsemi s \geq 0\big)$ is the sigma-field associated to the horizontal movement of the BBM.
We prove Proposition~\ref{prop:cvZy} by controlling the first two moments of $Z_\infty(y)$ conditionally on $\mathcal{H}$.
\begin{lemma}
  \label{lem:firstMoment}
  For all $y > 0$, we have $\displaystyle \E[Z_\infty(y) \mid \mathcal{H}] = y D_\infty$ a.s.
\end{lemma}
\begin{proof}                    
  For $y > 0$, we compute $\E[Z_\infty(y) \mid \mathcal{H}]$ using the approximation of $Z$ by shaved martingales introduced earlier.
  Given $\alpha > 0$, we define
  \begin{equation*}
    Z^\alpha_t(y) \coloneqq \sum_{u \in \mathcal{N}^y_t} \big[\sqrt{2} t + \alpha - X_t(u)\big] \big(Y_t(u) + y\big) \e^{\sqrt{2} X_t(u) - 2t} \ind{X_s(u) \leq \sqrt{2}s + \al \midsemi s \leq t}
  \end{equation*}
  and $Z_\infty^\alpha(y) \coloneqq \lim_{t \to \infty} Z^\alpha_t(y)$.
  Using the independence of the horizontal and the vertical displacement in the BBM, we see that for all $t, y > 0$,
  \begin{equation*}
    \E[Z^\alpha_t(y) \mid \mathcal{H}] = \sum_{u \in \mathcal{N}_t} \big[\sqrt{2} t + \alpha - X_t(u)\big] \e^{\sqrt{2} X_t(u) - 2t}\ind{X_s(u) \leq \sqrt{2} s + \alpha \midsemi s \leq t} \E_y(B_{t \wedge T_0})
  \end{equation*}
  almost surely, where $B$ is a Brownian motion with $B_0=y$ and $T_0$ is its hitting time with the origin.
  As $(B_{t\wedge T_0}\midsemi t\ge0)$ is a martingale, we obtain
  \begin{equation*}
    \E[Z^\alpha_t(y) \mid \mathcal{H}] = y \sum_{u \in \mathcal{N}_t} \big[\sqrt{2} t + \alpha - X_t(u)\big] \e^{\sqrt{2} X_t(u) - 2t}\ind{X_s(u) \leq \sqrt{2} s + \alpha \midsemi s \leq t} \quad \text{a.s.}
  \end{equation*}
  The sum is simply the one-dimensional shaved derivative martingale $D_t^\al$ defined in \eqref{eqn:shavedDerivativeMartingaleDimension1}, so
  \begin{equation}
    \label{eq:vertical-expectation}
    \E[Z^\alpha_t(y) \mid \mathcal{H}] = y D_t^\al \quad \text{a.s.}
  \end{equation}

  Now, Lemma~\ref{lem:cvShavedMartingale} implies that $Z^\alpha(y)$ is uniformly integrable with an almost sure $L^1$ limit $Z_\infty^\al(y)$.
  Since $D^\al$ converges in the same manner, \eqref{eq:vertical-expectation} yields
  \begin{equation*}
    \E[Z^\alpha_\infty(y) \mid \mathcal{H}] = \lim_{t \to \infty} \E[Z^\alpha_t(y) \mid \mathcal{H}] = \lim_{t \to \infty} y D_t^\al = y D_\infty^\al \quad \text{$\P_y$-a.s.}
  \end{equation*}
  Finally, $\alpha \mapsto Z^\alpha_\infty(y)$ is a.s. increasing and converges to $Z_\infty(y)$.
  Hence by monotone convergence,
  \begin{equation*}
    \E[Z_\infty(y) \mid \mathcal{H}] = \lim_{\alpha \to \infty} y D^\alpha_\infty = y D_\infty \quad \text{a.s.} \qedhere
  \end{equation*}
\end{proof}
We now turn to the second moment of $Z_\infty(y)$ conditioned on the horizontal motion.
\begin{lemma}
  \label{lem:secondMoment}
  There exists an a.s. finite $\m{H}$-measurable random variable $\Upsilon_{\infty}$ such that for all $y > 0$, 
  \begin{equation*}
    \E\left[Z_\infty(y)^2 \mid \mathcal{H}\right] \leq y^2D_\infty^2 + y \Upsilon_{\infty} \quad \text{a.s.}
  \end{equation*}
\end{lemma}
\begin{proof}
  Fix $y > 0$.
  In contrast to the expectation, the second moment of $Z$ reflects the correlation, and thus shared history, of particles in $\m{N}_t$.
  Given $t > 0$ and $u ,v \in \m{N}_t$, let $\tau^t_{u,v}$ be the age of the most recent common ancestor of $u$ and $v$, with the convention that $\tau_{u,u}^t = t$.
  Also, for $0 \leq r \leq t$, let
  \begin{equation*}
    G_y(r) \coloneqq \E_y\left( B_{r \wedge T_0}^2 \right),
  \end{equation*}
  where $B$ is a Brownian motion beginning at $y$ and $T_0$ is its hitting time at the origin.
  As in the previous proof, the independence between horizontal and vertical motion in the BBM and the martingale property of $B$ yield
  \begin{multline*}
    \E\left[Z_t(y)^2 \mid \mathcal{H}\right]\\
    = \sum_{u, v \in \mathcal{N}_t} \big[\sqrt{2} t- X_t(u)\big]\big[\sqrt{2}t - X_t(v)\big] \e^{\sqrt{2}[X_t(u)+X_t(v)] - 4 t} G_y(\tau^t_{u,v}) \quad \text{a.s.}
  \end{multline*} 
  Using the martingale property of $(B_t^2 - t, t \geq 0)$ and the Brownian scaling, we have
  \begin{equation*}
    G_y(r) = y^2 + \E_y(r \wedge T_0) = y^2 \left[1 + \E_1\left(\tfrac{r}{y^2} \wedge T_0\right) \right].
  \end{equation*}
  Using the explicit density for the hitting time $T_0$, we obtain the following bound:
  \begin{align*}
    \E_1\left( s \wedge T_0 \right) &= \int_0^\infty \frac{\dd u}{\sqrt{2\pi u^3}} \e^{-1/2u} s \wedge u\\
                                    &\leq \int_0^s \frac{\dd u}{\sqrt{2 \pi u}} + s\int_s^\infty \frac{\dd u}{\sqrt{2\pi u^3}} \leq \frac{4}{\sqrt{2\pi}} \sqrt{s}.
  \end{align*}
  For sufficiently large random $t$, \eqref{eqn:formulaForMaxDep} implies that $\max_{u \in \mathcal{N}_t} X_t(u) \leq \sqrt{2} t$.
  Combining the previous displays, for all such $t$ we have
  \begin{multline*}
    \E\left[Z_t(y)^2 \mid \mathcal{H} \right] - y^2 D_t^2 \\
    \leq \frac{4 y}{\sqrt{2\pi}} \sum_{u, v \in \mathcal{N}_t} \sqrt{\tau_{u,v}} \big[\sqrt{2} t- X_t(u)\big]\big[\sqrt{2}t - X_t(v)\big] \e^{\sqrt{2}[X_t(u)+X_t(v)] - 4 t} \,.
  \end{multline*}
  Define the $\m H$-measurable random variable
  \begin{equation*}
    \Upsilon_\infty \coloneqq \liminf_{t \to \infty} \frac{4}{\sqrt{2\pi}} \sum_{u, v \in \mathcal{N}_t} \sqrt{\tau_{u,v}} \big[\sqrt{2} t- X_t(u)\big]\big[\sqrt{2}t - X_t(v)\big] \e^{\sqrt{2}[X_t(u)+X_t(v)] - 4 t}.
  \end{equation*}
  Then Fatou's lemma yields
  \begin{equation*}
    \E\left[Z_\infty(y)^2 \mid \mathcal{H} \right] \leq y^2 D_\infty^2 + y \Upsilon_\infty.
  \end{equation*}
  
  To complete the proof, we must show that $\Upsilon_\infty < \infty$ almost surely.
  Given $\alpha > 0$, we define the event
  \begin{equation*}
    G_\alpha \coloneqq \Big\{\ \max_{u \in \mathcal{N}_t} X_t(u) \leq \sqrt{2} t - \tfrac{1}{4}\log_+ t + \alpha\,\text{ for all }\, t \geq 0 \Big\}.
  \end{equation*}
  It follows from the results of \cite{Hu16} that for all $\lambda < \frac{1}{2\sqrt{2}}$,
  \begin{equation*}
    \lim_{t \to \infty} \max_{u \in \mathcal{N}_t}\big[X_t(u) - \sqrt{2} t + \lambda \log t\big]  = -\infty \quad \text{a.s.}
  \end{equation*}
  Hence $\P(G_\alpha) > 0$ for all $\alpha > 0$ and $\lim_{\alpha \to \infty} \P(G_\alpha) = 1$.
  It therefore suffices to show that $\Upsilon_{\infty} < \infty$ a.s. on $G_\alpha$.
  
  By Fatou, we have
  \begin{align*}
    \frac{\sqrt{2\pi}}{4}&\E(\Upsilon_{\infty} \indset{G_\alpha})\\
    \leq&\liminf_{t \to \infty} \E\left( \sum_{u \in \mathcal{N}_t} \big[\sqrt{2} t + \alpha- X_t(u)\big] \e^{\sqrt{2} X_t(u) - 2t} \ind{X_s(u) \leq \sqrt{2} s - \tfrac{1}{4}\log_+ s +\alpha \midsemi s\le t} \Gamma_u \right),
  \end{align*}
  where $\Gamma_u \coloneqq \sum_{v \in \mathcal{N}_t} \sqrt{\tau_{u,v}} \big[\sqrt{2}t - X_t(v) + \alpha\big] \e^{\sqrt{2}X_t(v) - 2 t} \ind{X_s(v) \leq \sqrt{2}s + \al \midsemi s \leq t}$.

  We now employ a spine decomposition corresponding to the shaved derivative martingale $D^\al$.
  Let $\h\Q^\al$ denote the law of a one-dimensional BBM with spine in which $R_t \coloneqq \sqrt{2} t + \alpha - X_t(\xi_t)$ is a Bessel process of dimension $3$ started from $\alpha$, the spine branches at rate $2$, and non-spine particles perform standard BBMs.
  For all $t \geq 0$, the spine decomposition theorem allows us to write
  \begin{multline*}
    \E\left(\sum_{u \in \mathcal{N}_t} \big[\sqrt{2} t + \alpha- X_t(u)\big] \e^{\sqrt{2} X_t(u) - 2t} \ind{X_s(u) \leq \sqrt{2} s - \tfrac{1}{4}\log_+ s +\alpha \midsemi s\le t} \Gamma_u \right)\\
    = \al\hat{\E}^\alpha\left( \Gamma_{\xi_t} \ind{X_s(\xi_s) \leq \sqrt{2} s - \tfrac{1}{4} \log_+ s + \alpha \midsemi s \le t} \right).
  \end{multline*}
  We decompose $\Gamma_{\xi_t}$ as $\sqrt{t} \big[\sqrt{2}t - X_t(\xi_t) + \alpha\big] \e^{\sqrt{2} X_t(\xi_t)-2t} +\tilde{\Gamma}$.
  Let $\mathcal{Y}$ denote the filtration associated to the spine trajectory and branching times $\{\tau_k\}_{k \in \N}$.
  Using the branching property and the martingale property of $D^\alpha$, we have
  \begin{equation*}
    \hat{\E}^\alpha(\tilde{\Gamma} \mid \mathcal{Y}) \leq \sum_{k \in \N} \sqrt{\tau_k} \ind{\tau_k < t} \big[\sqrt{2}\tau_k + \alpha - X_{\tau_k}(\xi_{\tau_k})\big] \e^{\sqrt{2} X_{\tau_k}(\xi_{\tau_k}) - 2\tau_k} \eqqcolon \hat{\Gamma} \quad \text{a.s.}
  \end{equation*}
  Combining the above displays, we obtain
  \begin{align*}
    &\E\left(\sum_{u \in \mathcal{N}_t} \big[\sqrt{2} t + \alpha- X_t(u)\big] \e^{\sqrt{2} X_t(u) - 2t} \ind{X_s(u) \leq \sqrt{2} s - \tfrac{1}{4}\log_+ s +\alpha \midsemi s\le t} \Gamma_u \right)\\
    &\hspace{0.5pt}\leq\alpha \sqrt{t} \hat{\E}^\alpha\left( \big[\sqrt{2} t +\alpha - X_t(\xi_t)\big]\e^{\sqrt{2} X_t(\xi_t) - 2t} \right) + \alpha\hat{\E}^\alpha\left( \ind{X_s(\xi_s) \leq \sqrt{2} s - \tfrac{1}{4} \log_+ s + \alpha \midsemi s \le t} \hat{\Gamma} \right).
  \end{align*}  
  As $(\tau_k)_{k \in \N}$ are the atoms of a Poisson process of intensity $2$ independent of $R$, it follows that
  \begin{equation}
    \label{eq:Bessel-terms}
    \begin{aligned}
      \frac{\sqrt{2\pi}}{4\al}\E(\Upsilon_{\infty} \indset{G_\alpha})
      \leq &\liminf_{t \to \infty} \sqrt{t} \e^{\sqrt{2}\alpha}\hat{\E}^\alpha\left( R_t \e^{-\sqrt{2} R_t} \right)\\
      &+ \liminf_{t \to \infty} 2 \int_0^t \sqrt{s} \e^{\sqrt{2}\alpha} \hat{\E}^\alpha\left( R_s \e^{-\sqrt{2} R_s}\ind{R_s \geq \frac{1}{4} \log_+ s} \right) \d s.
    \end{aligned}
  \end{equation}
  Using the density of the Bessel process and dominated convergence, we can check that
  \begin{equation*}
    \lim_{t \to \infty} \sqrt{t} \h\E^\al\left( R_t \e^{-\sqrt{2} R_t} \right) = \lim_{t \to \infty} \frac{1}{2\al\sqrt{2\pi}} \int_0^\infty \e^{-\sqrt{2}y}\left[\e^{-\frac{(y-\alpha)^2}{2t}} - \e^{-\frac{(y+\alpha)^2}{2t}} \right] \dd y = 0.
  \end{equation*}
  For the second term in \eqref{eq:Bessel-terms}, we note that $\h\Q^\alpha(R_s \in [n, n+1]) \lesssim_\al (n+1)^3s^{-3/2}$ for all $n \in \N$ and $s \geq 1$.
  It follows that
  \begin{equation*}
    \hat{\E}^\alpha\left( R_s \e^{-\sqrt{2} R_s}\ind{R_s \geq \frac{1}{4} \log_+ s} \right) \lesssim_\al s^{-\frac{3}{2}} \hspace{-6pt}\sum_{n > \frac{1}{4}\log_+ s} (n+1)^4 \e^{-\sqrt{2}n} \lesssim s^{-\big(\frac{3}{2}+\frac{\sqrt{2}}{4}\big)} \log_+^4 s.
  \end{equation*}
  Because this is integrable against the weight $\sqrt{s}$, the second term in \eqref{eq:Bessel-terms} is finite.
  Together, these bounds show that $\E(\Upsilon_{\infty} \indset{G_\alpha}) < \infty$.
  Hence $\Upsilon_\infty < \infty$ a.s. on $G_\al$, as desired.
\end{proof}
We can now complete the proofs of Propositions~\ref{prop:cvZy} and \ref{prop:tw}.
\begin{proof}[Proof of Proposition~\textup{\ref{prop:cvZy}}]
  Using Lemma~\ref{lem:firstMoment} and \ref{lem:secondMoment}, we observe that for all $y > 0$, we have
  \begin{equation*}
    \E\left(\left[\frac{Z_\infty(y)}{y} - D_\infty\right]^2 \; \Big|\; \mathcal{H}\right) \leq \frac{\Upsilon_{\infty}}{y} \quad \text{a.s.}
  \end{equation*}
  Thus conditioned on $\m{H},$ $Z_\infty(y)/y$ converges in $L^2$ to $D_\infty$ as $y \to \infty$.
  This implies convergence in probability, completing the proof.
\end{proof}
\begin{proof}[Proof of Proposition~\textup{\ref{prop:tw}}]
  By Proposition~\ref{prop:cvZy}, $Z_\infty(y)/y \to D_\infty$ in probability and hence in distribution as $y \to \infty$.
  In turn, this implies convergence of the Laplace transforms.
  Recalling \eqref{eq:defTW}, we find
  \begin{align*}
    1 - \Phi\left(x+\tfrac{1}{\sqrt{2}}\log y,y\right)
    &= \E \exp\left[-\e^{-\sqrt{2}x}\frac{Z_\infty(y)}{y}\right]
      \to \E \exp\left(-\e^{-\sqrt{2}x}D_\infty\right)
  \end{align*}
  as $y \to \infty$.
  Given the definition of $w_{c_*}$ in \eqref{eqn:defw}, we conclude that $\Phi\big(x + \tfrac{1}{\sqrt{2}}\log y,y\big)$ converges pointwise to $w_{c_*}(x)$ as $y \to \infty$.
  Applying Dini's second theorem on the compactification $[-\infty, \infty]$ as in the proof of Lemma~\ref{lem:level}, we see that the convergence is in fact uniform.
\end{proof}

\subsection{The higher-speed waves far from the boundary}
\label{subsec:hghtw}

In Section~\ref{sec:supercritical}, we constructed supercritical traveling waves via Laplace transforms of limits of the additive martingales
\begin{equation*}
  W_t^{\lambda,\mu} \coloneqq \sum_{u \in \mathcal{N}_t^{+}} \e^{\lambda X_t(u) } \sinh[\mu Y_t(u)] \e^{-(\lambda^2/2 + \mu^2/2 + 1)t}.
\end{equation*}
More precisely, for all $\lambda,\mu > 0$ with $\lambda^2 +\mu^2 < 2$, this martingale converges almost surely to a nondegenerate limit $W_\infty^{\lambda,\mu}$, and the function
\begin{equation*}
  \Phi_{\lambda,\mu} (x,y) \coloneqq 1-\E_y \exp\big(-\e^{-\lambda x} W^{\lambda,\mu}_\infty\big)
\end{equation*}
is a traveling wave with speed $c = \frac{\lambda^2 + \mu^2 + 2}{2\lambda}$. 

In this subsection, we study the asymptotic behavior of this traveling wave as $y \to \infty$.
As above, we focus on the martingale $W^{\lambda,\mu}$.
To begin, we construct a consistent family $\big(W_\infty^{\lambda,\mu}(y) \midsemi y \geq 0\big)$ of random variables on a single probability space.
Let
\begin{align*}
  &W_t^{\lambda,\mu}(y) \coloneqq \sum_{u \in \mathcal{N}_t^y} \e^{\lambda X_t(u)} \sinh\big(\mu [Y_t(u)+y]\big)\e^{-(\lambda^2/2 + \mu^2/2 + 1)t}\\
  \text{and} \quad &W_\infty^{\lambda,\mu}(y) \coloneqq \lim_{t \to \infty} W_t^{\lambda,\mu}(y).
\end{align*}
We relate $W_\infty^{\lambda,\mu}(y)$ to the following additive martingale associated to the BBM in~$\R^2$:
\begin{align*}
  &A^{\lambda,\mu}_t \coloneqq \sum_{u \in \mathcal{N}_t} \e^{\lambda X_t(u) + \mu Y(u) - (\lambda^2/2 + \mu^2/2+1)t}\\
  \text{and} \quad & A^{\lambda,\mu}_\infty \coloneqq \lim_{t \to \infty} A^{\lambda,\mu}_t.
\end{align*}

We intend to show that $\Phi_{\lambda,\mu}$ asymptotically resembles a one-dimensional wave rotated by angle $\theta(\lambda, \mu) \coloneqq \arctan(\mu/\lambda)$.
As in the introduction, let $R_{\lambda, \mu}$ denote \emph{clockwise} rotation by angle $\theta(\lambda,\mu)$.
In a certain sense, $A^{\lambda,\mu}_t$ is related to a one-dimensional additive martingale by the rotation $R_{\lambda, \mu}$.
Given $\rho \in (0, \sqrt{2})$, let
\begin{equation*}
  A_t^\rho \coloneqq \sum_{u \in \mathcal{N}_t} \e^{\rho X_t(u) - (\rho^2/2 +1)t}
\end{equation*}
denote said martingale, which has a nondegenerate limit $A_\infty^\rho$.
For each $c > c_* = \sqrt{2}$, there is a unique $\rho \in (0, \sqrt{2})$ such that $c = \frac{\rho^2 + 2}{2\rho}$.
Taking this value of $\rho$, we define
\begin{equation}
  \label{eq:supercritical-1D-wave}
  w_c(x) \coloneqq 1 - \E \exp \left(-\e^{-\rho x} A_\infty^\rho\right).
\end{equation}
This is a one-dimensional traveling wave of speed $c$.
By the rotational invariance in law of BBM in $\R^2$, $A^{\lambda,\mu}_\infty \overset{\text{(d)}}{=} A_\infty^{\rho(\lambda, \mu)}$ for $\rho(\lambda, \mu) \coloneqq \sqrt{\lambda^2 + \mu^2}$.
It follows that
\begin{equation}
  \label{eq:super-wave-add}
  w_{c(\lambda, \mu)}(x) = 1 - \E \exp \left(-\e^{-\rho(\lambda, \mu) x} A_\infty^{\lambda, \mu}\right)
\end{equation}
for $c(\lambda, \mu)$ given by \eqref{eq:speed}.

In this subsection, we prove the following analogue of Propositions~\ref{prop:tw} and \ref{prop:cvZy}.
Recall $\m{Q} \coloneqq \{(\lambda, \mu) \in \R_+^2 : \lambda^2 + \mu^2 < 2\}$.
\begin{proposition}
  \label{prop:wave}
  For all $\lambda,\mu \in \m{Q}$, we have $\e^{-\mu y}W_\infty^{\lambda,\mu}(y) \to A_\infty^{\lambda,\mu}$ in probability as $y \to \infty$.
  Moreover,
  \begin{equation}
    \label{eq:incline}
    \Phi_{\lambda,\mu} \circ R_{\lambda, \mu}(x, y) \to w_{c(\lambda,\mu)}(x)
  \end{equation}
  uniformly in $x \in \R$ as $y \to \infty$.
\end{proposition}
\begin{proof}
  For all $t, y > 0$, we have
  \begin{align*}
    \e^{-\mu y} W_t^{\lambda,\mu}(y) = \sum_{u \in \mathcal{N}_t(y)} \e^{\lambda X_t(u) + \mu Y(u) - (\lambda^2/2 + \mu^2/2+1)t}\left(1  -\e^{-2\mu [Y_t(u)+y]}\right).
  \end{align*}
  Therefore, $y \mapsto \e^{-\mu y} W_t^{\lambda,\mu}(y)$ is nondecreasing in $y$ and bounded by $A^{\lambda,\mu}_t$ almost surely.
  As a consequence, $y \mapsto \e^{-\mu y} W_\infty^{\lambda,\mu}(y)$ is a.s. nondecreasing in $y$ and
  \begin{equation*}
    \lim_{y \to \infty} \e^{-\mu y} W_\infty^{\lambda,\mu}(y) \leq A^{\lambda,\mu}_\infty \quad \text{a.s.}
  \end{equation*}
  By uniform integrability and monotone convergence, we have
  \begin{equation*}
    \E\big[\lim_{y \to \infty} \e^{-\mu y} W_\infty^{\lambda,\mu}(y)\big] = \lim_{y \to \infty} \e^{-\mu y} \sinh(\mu y) = 1.
  \end{equation*}
  Therefore $\E\big[A^{\lambda,\mu}_\infty - \lim_{y \to \infty} \e^{-\mu y} W_\infty^{\lambda,\mu}(y)\big] \leq 0$, and we conclude that
  \begin{equation}
    \label{eq:super-mart-asymp}
    \lim_{y \to \infty} \e^{-\mu y} W_\infty^{\lambda,\mu}(y) = A^{\lambda,\mu}_\infty \quad \text{a.s.}
  \end{equation}

  Turning to the asymptotic behavior of $\Phi_{\lambda,\mu}$, we compute
  \begin{equation*}
    R_{\lambda, \mu}(x, y) = \big((\lambda x + \mu y)/\rho, \, (- \mu x + \lambda y)/\rho\big) \eqqcolon (\ti x, \ti y).
  \end{equation*}
  Note in particular that $\lambda \ti x = \mu \ti y + \rho x$, where $\rho = \rho(\lambda, \mu) = \sqrt{\lambda^2 + \mu^2}$.
  Now take $(x, y) \in R_{\lambda,\mu}^{-1}\H$.
  It follows that
  \begin{equation*}
    \Phi_{\lambda, \mu} \circ R_{\lambda, \mu}(x, y) = 1 - \E_{\ti y}\exp\left(-\e^{-\lambda \ti x} W_\infty^{\lambda, \mu}\right) = 1 - \E_{\ti y} \exp\left(-\e^{-\rho x} \e^{-\mu \ti y} W_\infty^{\lambda, \mu}\right).
  \end{equation*}
  If we fix $x \in \R$ and take $y \to \infty$, we also have $\ti y \to \infty$.
  Thus \eqref{eq:super-mart-asymp} and \eqref{eq:super-wave-add} yield
  \begin{equation}
    \label{eq:rotated-convergence}
    \Phi_{\lambda, \mu} \circ R_{\lambda, \mu}(x, y) \to 1 - \E \exp\left(-\e^{-\rho x} A_\infty^{\lambda, \mu}\right) = w_{c(\lambda, \mu)}(x) \quad \text{as } y \to \infty.
  \end{equation}
  We extend $\Phi_{\lambda, \mu}$ by $0$ to the entire plane $\R^2$.
  Because $\Phi_{\lambda, \mu}$ is decreasing in $x$ and increasing in $y$, one can check that $\Phi_{\lambda, \mu} \circ R_{\lambda, \mu}(\anon, y)$ is a nonincreasing function for each $y > 0$ fixed.
  Moreover, for all $y > 0$,
  \begin{equation*}
    \Phi_{\lambda, \mu} \circ R_{\lambda, \mu}(-\infty, y) = 1 = w_{c(\lambda, \mu)}(-\infty) \And \Phi_{\lambda, \mu} \circ R_{\lambda, \mu}(+\infty, y) = 0 = w_{c(\lambda, \mu)}(+\infty).
  \end{equation*}
  Applying Dini's second theorem on the compactification $[-\infty, \infty]$, we see that the limit \eqref{eq:rotated-convergence} in fact holds uniformly in $x$.
\end{proof}
We can now describe the limits of our waves in every direction.
\begin{proof}[Proof of Theorem~\textup{\ref{thm:asymptotics}}]
  Take $\Phi^* \in \{\Phi, \Phi_{\lambda,\mu}\}_{(\lambda, \mu) \in \m{Q}}$.
  Our traveling-wave constructions automatically imply that $0 < \Phi^* < 1$, $\partial_x \Phi^* \leq 0$, and $\partial_y \Phi^* \geq 0$.
  Combining the bounded convergence theorem with (the proof of) Corollary~\ref{cor:survival} and \eqref{eq:super-survival}, we find
  \begin{equation}
    \label{eq:proof-limits}
    \Phi^*(-\infty, y) = \P_y(\m{N}_t^+ \neq \emptyset \text{ for all } t \geq 0) = \varphi(y) \And \Phi^*(+\infty, y) = 0
  \end{equation}
  for all $y > 0$.
  Uniform continuity implies that this convergence is locally uniform in $y$.
  Moreover, joint monotonicity implies that $\Phi^*(x, +\infty) \to 1 = \varphi(+\infty)$ as $x \to -\infty$.
  (Alternatively, Propositions~\ref{prop:tw} and~\ref{prop:wave} imply that $\Phi^*(x, +\infty) = 1$ for all $x \in \R$.)
  Applying Dini's theorem on the compactification $[0, \infty]$, we see that the first limit in \eqref{eq:proof-limits} holds uniformly in $y$.
  Because the left and right limits are distinct, the strong maximum principle implies that $\partial_x \Phi^* < 0$.
  Similarly, because $\Phi^* > 0$ in $\H$ but $\Phi^* = 0$ on $\partial \H$, we have $\partial_y \Phi^* > 0$.
  Finally, \eqref{eq:y-limits} combines Propositions~\ref{prop:tw} and~\ref{prop:wave}.
\end{proof}
According to \eqref{eq:incline}, the level sets of $\Phi_{\lambda, \mu}$ are asymptotically inclined at angle $\arctan(\mu/\lambda)$ relative to vertical.
For a given speed $c > c_*$, this angle varies strictly monotonically along $\m{P}_c$.
It follows that the waves in $\m{P}_c$ are distinct modulo translation.
Using Proposition~\ref{prop:uniquenessRephrased} and the above observation, we can now complete the proof of Theorem~\ref{thm:unique} and thus bridge the gap in the proof of Theorem~\ref{thm:construction}.
\begin{proof}[Proof of Theorem~\textup{\ref{thm:unique}}]
  By Proposition~\ref{prop:monotone}, any KPP traveling wave on $\H^d$ with minimal speed is a function of $(x, y)$ alone.
  Thus it suffices to prove uniqueness in two dimensions, i.e., on $\H$.
  Given a minimal-speed wave $\Phi$ on $\H$, Proposition~\ref{prop:uniquenessRephrased} provides a constant $\eta \in \R$ such that $\Psi(x, y) = \Phi(x - \eta, y)$, where $\Phi$ is the wave defined by \eqref{eq:defTW}.
  Hence there is precisely one traveling wave on $\H^d$, modulo translation.

  Now fix $c > c_*$ and recall the set $\m{P}_c$ from \eqref{eq:speed-set}.
  For all $(\lambda,\mu) \in \mathcal{P}_c$, $\Phi_{\lambda,\mu}$ defined in \eqref{eq:def-super} is a traveling wave of speed $c$ (Proposition~\ref{prop:super}).
  Moreover, Proposition~\ref{prop:wave} implies that distinct values of $(\lambda,\mu)$ produce distinct waves.
  Thus there are infinitely many traveling waves of speed $c$ that are distinct modulo translation.
\end{proof}

\subsection{Minimal-speed tail asymptotics}
\label{subsec:extendedAsymptotics}

We now examine the asymptotic behavior of $\Phi$ as $x \to \infty$.
Let $\omega(y) \coloneqq \tfrac{1}{\sqrt{2}} \log y$.
From Proposition~\ref{prop:tw}, we know that the level sets of $\Psi$ follow the curve $x = \omega(y)$ as $y \to \infty$.
Thus $\Phi$ decays to the right of this curve.
The following result controls this decay.

Recall the constant $K_* > 0$ from \eqref{eq:wave-constant}, which governs the tail of the one-dimensional wave.
We can state Theorem~\ref{thm:tail} as
\begin{proposition}
  \label{prop:tail-asymp}
  There exists $E \in L^\infty(\H)$ such that if $x > \omega_+(y)$,
  \begin{equation}
    \label{eq:tail-asymp}
    \Phi(x, y) = K_*\big[x - \tfrac{1}{\sqrt{2}}\log_+\norm{\tbf{x}} + E(x, y)\big] y \e^{-\sqrt{2} x}.
  \end{equation}
\end{proposition}
To prove this, we return to conformal mappings and potential theory.
We focus on the function $\Theta \coloneqq \e^{\sqrt{2} x} \Phi$, which is nearly harmonic in $\{x > \omega_+(y)\}$.
We begin by constructing an explicit holomorphism mapping a domain similar to $\{x > \omega_+\}$ to the quarter-plane $\Q$.
This allows us to use various explicit formul\ae~on the quarter-plane.
The distortion induces by this holomorphism leads to the $\log \norm{\tbf{x}}$ term in \eqref{eq:tail-asymp}.

As a first application, we use the Phragm\'en--Lindel\"of principle to improve our tail bound from $\Theta \lesssim (x + 1)y$ to $\Theta \lesssim (x - \omega_+ + 1) y$ on $\{x > \omega_+\}$.
This allows us to argue that the ``anharmonic'' part of $\Theta$ is negligible---it can be absorbed in the error $E$ in \eqref{eq:tail-asymp}.
We are thus left with the analysis of a positive harmonic function on the quarter-plane.
From this point, the Herglotz representation formula is strong enough to complete the proof of Proposition~\ref{prop:tail-asymp}.

\subsubsection*{Conformal map to quarter-plane}
To begin, we construct a conformal map $\eta$ from $\Q$ to a domain similar to $\{x > \omega_+\}$.
We define $\eta$ through its inverse
\begin{equation*}
  \eta^{-1}(z) \coloneqq z - \tfrac{1}{\sqrt{2}} \log (z + 1).
\end{equation*}
Throughout this section, we frequently identify $z = x + \iu y \in \C$ with $(x, y) \in \R^2$.
Solving the equation $\Re \eta^{-1}(z) = 0$, we can readily check that $\eta$ maps $\Q$ to the region
\begin{equation*}
  \Lambda \coloneqq \left\{(x, y) \in \H : 0 < y < \sqrt{\e^{2 \sqrt{2} x} - 1 - x^2}\,\right\} \subset \{x > \omega_+(y)\}.
\end{equation*}
Note that
\begin{equation*}
  (\eta^{-1})'(z) = 1 - \frac{1}{\sqrt{2}(z + 1)}
\end{equation*}
satisfies $1 - \tfrac{1}{\sqrt{2}} \leq |(\eta^{-1})'| \leq 1 + \tfrac{1}{\sqrt{2}}$ on $\Q$.
It follows that $\eta \colon \Q \to \Lambda$ is a biholomorphism.

Next, we can write $\{x > \omega_+\} = \big\{x > 0,\, 0 < y < \e^{\sqrt{2} x}\big\}$.
Because
\begin{equation*}
  \sqrt{\e^{2 \sqrt{2} x} - 1 - x^2} = \e^{\sqrt{2} x} + \m{O}(1) \quad \text{on } \R_+,
\end{equation*}
the difference $\{x > \omega_+(y)\} \setminus \Lambda$ lies a bounded distance from $\Lambda$.
Thus by the Harnack inequality, it suffices to prove \eqref{eq:tail-asymp} on $\Lambda$.

Although $\eta$ itself has no simple explicit expression, we can easily construct an approximation
\begin{equation*}
  \varpi(z) \coloneqq z + \tfrac{1}{\sqrt{2}} \log(z + 1)
\end{equation*}
Indeed,
\begin{equation*}
  \eta^{-1}\circ \varpi(z) = z - \frac{1}{\sqrt{2}}\left[1 + \frac{\log(z+1)}{\sqrt{2}(z + 1)}\right] = z + \m{O}(1).
\end{equation*}
Since $\eta$ is uniformly Lipschitz, this yields
\begin{equation*}
  \eta(z) = \varpi(z) + \m{O}(1).
\end{equation*}
Using this approximation, we establish two bounds that will be useful in subsequent calculations.
\begin{equation}
  \label{eq:boundary-eta}
  \eta(0, y) = (\omega_+(y), y) + \m{O}(1)
\end{equation}
and
\begin{equation}
  \label{eq:exp-eta}
  \e^{-\sqrt{2}\Re \eta(z)} \asymp \abs{z+1}^{-1}\e^{-\sqrt{2}x}.
\end{equation}

Now recall $\Theta(x, y) \coloneqq \e^{\sqrt{2} x} \Phi(x, y)$, which satisfies
\begin{equation*}
  -\frac{1}{2}\Delta \Theta = - \Theta^2 \e^{-\sqrt{2}x} \eqqcolon -F.
\end{equation*}
We define
\begin{equation*}
  \Theta^\eta \coloneqq \Theta \circ \eta \colon \Q \to \R_+.
\end{equation*}
This function satisfies
\begin{equation*}
  -\frac{1}{2}\Delta \Theta^\eta = -\abs{\eta}^2 F \circ \eta \eqqcolon -F^\eta.
\end{equation*}

We have shown that $\Phi(\omega(y), y) \asymp 1$ when $y \geq 1$.
Hence $\Theta(\omega(y), y) \asymp y$ there.
Using \eqref{eq:boundary-eta}, Harnack estimates up to the boundary imply that
\begin{equation*}
  \Theta^\eta(0, y) \asymp y \ForAll y \geq 0.
\end{equation*}
Moreover, the tameness bound in Proposition~\ref{prop:a-priori} and the boundedness of $\eta'$ imply that
\begin{equation*}
  \Theta^\eta \lesssim 1 + x^2 + y^2 \quad \text{on }\Q
\end{equation*}
as well as
\begin{equation*}
  \Theta^\eta(x, x) \lesssim 1 + x^2 \And \Theta^\eta(0, x) = 0.
\end{equation*}
Note that $\Theta^\eta$ is subharmonic.
We can thus apply the Phragm\'en--Lindel\"of principle on the sectors $\{0 < \theta < \pi/4\}$ and $\{\pi/4 < \theta < \pi/2\}$ as in the proof of Proposition~\ref{prop:a-priori} to conclude that
\begin{equation}
  \label{eq:improved}
  \Theta^\eta \lesssim (x + 1) y \quad \text{on } \Q.
\end{equation}
We further use this to bound $F^\eta$.
Noting that $\abs{\eta'}^2 \asymp 1$, \eqref{eq:exp-eta} and \eqref{eq:improved} yield
\begin{equation}
  \label{eq:F-eta}
  F^\eta \lesssim (x + 1)^2y^2 \e^{-\sqrt{2} \Re \eta} \lesssim \frac{(x + 1)^2 y^2}{\abs{z + 1}} \e^{-\sqrt{2} z} \lesssim (x + 1)^2 y \eta^{-\sqrt{2}x}.
\end{equation}

\subsubsection*{Anharmonic estimates}
We now control the ``anharmonic'' component of $\Theta^\eta$ generated by $F^\eta$.
Let $G_{\tbf{z}}$ denote the Dirichlet Green function of $-\frac{1}{2} \Delta$ on $\Q$ centered at $\tbf{z} \in \Q$.
We formally define
\begin{equation}
  \label{eq:anharmonic-int}
  \Theta_F^\eta(\tbf{x}) \coloneqq \int_{\Q} F^\eta(\tbf{z}) G_{\tbf{z}}(\tbf{x}) \d \tbf{z},
\end{equation}
which satisfies $-\frac{1}{2} \Delta \Theta_F^\eta = F^\eta$.
We refer to $\Theta_F^\eta$ as the ``anharmonic component'' of $\Theta^\eta$ because $\Theta^\eta + \Theta_F^\eta$ is harmonic.
To make this decomposition rigorous, we must show that the integral in \eqref{eq:anharmonic-int} is finite.

Let $\tau_x$ and $\tau_y$ denote reflection in $\{x = 0\}$ and $\{y = 0\}$, respectively, in $[0, \infty)$.
Then the method of images yields an explicit formula for $G$:
\begin{equation*}
  G_{\tbf{z}}(\tbf{x}) = \frac{1}{\pi} \log\left(\frac{\norm{\tbf{x} - \tau_x\tbf{z}}\|\tbf{x} - \tau_y\tbf{z}\|}{\norm{\tbf{x} - \tbf{z}}\norm{\tbf{x} + \tbf{z}}}\right).
\end{equation*}
We make use of the following asymptotics:
\begin{lemma}
  \label{lem:Green}
  Fix $\tbf{z} \coloneqq (u, v) \in \Q$.
  Then for all $\tbf{x} \in \Q$,
  \begin{equation*}
    G_{\tbf{z}}(\tbf{x}) \asymp
    \begin{cases}
      \log \frac{u \wedge v}{\norm{\tbf{x} - \tbf{z}}} & \text{in } B_{(u \wedge v)/10}(\tbf{z}),\\[1ex]
      \frac{uv}{\norm{\tbf z}^2} \frac{xy}{\norm{\tbf{x} - \tbf{z}}^2} & \text{in } B_{2 \norm{\tbf{z}}} \setminus B_{(u \wedge v)/10}(\tbf{z}),\\[1.6ex]
      \frac{uv xy}{\norm{\tbf x}^4} & \text{in } B_{2 \norm{\tbf{z}}}^c.
    \end{cases}
  \end{equation*}
\end{lemma}
\begin{proof}
  Define
  \begin{equation*}
    \al \coloneqq \frac{\norm{\tbf{x} - \tau_x \tbf{z}}}{2u}, \quad \beta \coloneqq \frac{\|\tbf{x} - \tau_y \tbf{z}\|}{2v}, \quad \gamma \coloneqq \frac{\norm{\tbf{x} + \tbf{z}}}{2 \norm{\tbf{z}}}, \And \delta \coloneqq \frac{\norm{\tbf{z}}}{2(u \vee v)}.
  \end{equation*}
  Then the identity $\frac{uv}{u \wedge v} = u \vee v$ yields
  \begin{equation}
    \label{eq:in}
    G_{\tbf{z}}(\tbf{x}) = \frac{1}{\pi} \log \frac{u \wedge v}{\norm{\tbf{x} - \tbf{z}}} + \frac{1}{\pi} \log \frac{\al\beta}{\gamma \delta}.
  \end{equation}
  Let $B \coloneqq B_{(u \wedge v)/10}(\tbf{z})$ and suppose $\tbf{x} \in B$.
  Then $\al,\beta,\gamma \in [19/20, 21/20]$ and $\delta \in [1/2, \sqrt{2}/2]$.
  Hence
  \begin{equation*}
    \frac{u \wedge v}{\norm{\tbf{x} - \tbf{z}}} \geq 10 \quad \text{while} \quad \frac{\al\beta}{\gamma \delta} \in [1.2, 2.4],
  \end{equation*}
  so the first term in \eqref{eq:in} dominates the second.
  
  Now suppose $\norm{\tbf x - \tbf z} \geq (u \wedge v)/10$.
  A brief algebraic computation yields
  \begin{equation}
    \label{eq:algebra}
    G_{\tbf{z}}(\tbf{x}) = \frac{1}{2\pi} \log \left(\frac{\norm{\tbf{x} - \tau_x\tbf{z}}^2\|\tbf{x} - \tau_y\tbf{z}\|^2}{\norm{\tbf{x} - \tbf{z}}^2\norm{\tbf{x} + \tbf{z}}^2}\right) = \frac{1}{2 \pi} \log \frac{q + s}{q + r}
  \end{equation}
  for
  \begin{align*}
    q &\coloneqq (x^2 - u^2)^2 + (y^2 - v^2)^2,\\
    s &\coloneqq (x + u)^2(y + v)^2 + (x - u)^2(y - v)^2,\\
    r &\coloneqq (x - u)^2(y + v)^2 + (x + u)^2(y - v)^2.
  \end{align*}
  Note that
  \begin{equation}
    \label{eq:quartic-diff}
    s - r = \big[(x+u)^2 - (x - u)^2\big]\big[(y + v)^2 - (y - v)^2\big] = 16xyuv > 0.
  \end{equation}
  We claim that the ratio $(q + s)/(q + r)$ is uniformly bounded on $B^c$.
  By \eqref{eq:quartic-diff}, it suffices to show that
  \begin{equation}
    \label{eq:diff-cond}
    xyu v = \frac{s - r}{16} \lesssim q + r = \norm{\tbf{x} - \tbf{z}}^2\norm{\tbf{x} + \tbf{z}}^2.
  \end{equation}
  This always holds when $\norm{\tbf{x}} \geq 2 \norm{\tbf{z}}$, for then the right side is of order $\norm{\tbf x}^4$ while the left side is at most of order $\norm{\tbf{x}}^2\norm{\tbf{z}}^2$.
  So we can assume that $\norm{\tbf{x}} \leq 2 \norm{\tbf{z}}$ and without loss of generality that $u \leq v$.
  Then $\norm{\tbf{x} + \tbf{z}} \asymp v$ and $y \lesssim v$, so it suffices to show that
  \begin{equation}
    \label{eq:1D-cond}
    xu \lesssim \norm{\tbf{x} - \tbf{z}}^2 \quad \text{on } B_{u/10}^c(\tbf{z}).
  \end{equation}
  We break this into two cases.
  If $x \leq 2u$, then indeed
  \begin{equation*}
    xu \leq 2 u^2 \leq 200(u^2/100) \leq 200 \norm{\tbf{x} - \tbf{z}}^2.
  \end{equation*}
  Otherwise if $x > 2u$, we have $(x - u)^2 \geq x^2/4$, so
  \begin{equation*}
    xu \leq x^2/2 \leq 2(x^2/4) \leq 2\norm{\tbf{x} - \tbf{z}}^2.
  \end{equation*}
  Having confirmed \eqref{eq:1D-cond} in each case, we have verified \eqref{eq:diff-cond}.
  Therefore
  \begin{equation*}
    1 < \frac{q + s}{q + r} \lesssim 1 \quad \text{on } B^c.
  \end{equation*}
  Then \eqref{eq:algebra} and \eqref{eq:quartic-diff} yield
  \begin{equation*}
    G_{\tbf{z}}(\tbf{x}) \asymp \frac{q + s}{q + r} - 1 = \frac{s - r}{q + r} \asymp \frac{xyuv}{\norm{\tbf{x} - \tbf{z}}^2\norm{\tbf{x} + \tbf{z}}^2}.
  \end{equation*}
  Now $\norm{\tbf{x} + \tbf{z}} \asymp \norm{\tbf{z}}$ on $B_{2\norm{\tbf{z}}}$ while $\norm{\tbf{x} - \tbf{z}} \asymp \norm{\tbf{x} + \tbf{z}} \asymp \norm{\tbf{x}}$ on $B_{2\norm{\tbf{z}}}^c$.
  The lemma follows.
\end{proof}
We combine this with \eqref{eq:F-eta} to bound $\Theta_F^\eta$.
\begin{lemma}
  \label{lem:anharmonic}
  The integral in \eqref{eq:anharmonic-int} is well-defined and $0 < \Theta_F^\eta(x, y) \lesssim y$ on $\Q$.
\end{lemma}
\begin{proof}
  Because $F > 0$ in $\H$, we automatically have $\Theta_F^\eta > 0$.
  Recall that the Green function is symmetric: $G_{\tbf{z}}(\tbf{x}) = G_{\tbf{x}}(\tbf{z})$.
  Here, it is convenient to use the symmetric formulation
  \begin{equation*}
    \Theta_F^\eta(\tbf x) = \int_{\Q} G_{\tbf{x}}(\tbf{z}) F(\tbf{z})  \d \tbf{z}.
  \end{equation*}
  Writing $\tbf{z} = (u, v)$, \eqref{eq:F-eta} yields
  \begin{equation}
    \label{eq:int-upper}
    \Theta_F^\eta(\tbf x) \lesssim \int_{\Q} G_{\tbf{x}}(\tbf{z}) (u + 1)^2 v \e^{-\sqrt{2} u} \d \tbf{z}.
  \end{equation}
  In the following, define $m \coloneqq x \wedge y$, $B \coloneqq B_{m/10}(\tbf{x})$, and $D \coloneqq B_{2\norm{\tbf{x}}}$.
  We divide the right side of \eqref{eq:int-upper} into three integrals $I_1,I_2$, $I_3$ over the regions $B$, $D \setminus B$, $D^c$, respectively.
  We bound these contributions separately.

  Using Lemma~\ref{lem:Green} on $B$, we have
  \begin{align}
    I_1 \lesssim \int_B \log\left(\frac{m}{\norm{\tbf{z} - \tbf{x}}}\right) (u + 1)^2 v \e^{-\sqrt{2} u} \d \tbf{z} &\lesssim y \e^{-\sqrt{2}x/2} \int_B \log\left(\frac{m}{\norm{\tbf{z} - \tbf{x}}}\right) \ds \tbf{z}\nonumber\\
                                                                                                                    &\lesssim m^2 y \e^{-\sqrt{2} x/2} \lesssim y.\label{eq:I1}
  \end{align}
  Next, on $D \setminus B$, we use Lemma~\ref{lem:Green} and $\frac{xy}{\norm{\tbf{x}}^2} \lesssim \frac{m}{y}$ to write
  \begin{align}
    I_2 &\lesssim \frac{xy}{\norm{\tbf{x}}^2} \int_{D \setminus B} u(u + 1)^2 \e^{-\sqrt{2} u} v^2 \norm{\tbf{z} - \tbf{x}}^{-2} \ds \tbf{z}\nonumber\\
        &\lesssim m y \int_D \e^{-\sqrt{2} u/2} \big(\norm{\tbf{z} - \tbf{x}} \vee m\big)^{-2} \ds \tbf{z}.\label{eq:middle}
  \end{align}
  Making a ``rectangular'' approximation,
  \begin{equation*}
    \int_D \e^{-\sqrt{2} u/2} \big(\norm{\tbf{z} - \tbf{x}} \vee m\big)^{-2} \ds \tbf{z} \lesssim \int_{\R_+} \e^{-\sqrt{2}u/2} \d u \cdot \int_{\R} \frac{\dn v}{(v - y)^2 + m^2}\lesssim \frac{1}{m}.
  \end{equation*}
  Thus \eqref{eq:middle} yields
  \begin{equation}
    \label{eq:I2}
    I_2 \lesssim y.
  \end{equation}
  Finally, on $D^c$, Lemma~\ref{lem:Green} yields
  \begin{equation*}
    I_3 \lesssim xy \int_{D^c} \frac{v^2}{(u^2 + v^2)^2} \e^{-\sqrt{2}u /2} \d \tbf{z}.
  \end{equation*}
  Integrating first in $v$, we find
  \begin{equation}
    \label{eq:I3}
    I_3 \lesssim xy \int_{\R_+} \frac{\e^{-\sqrt{2} u/2}}{u \vee \norm{\tbf{x}}} \d u \lesssim y.
  \end{equation}
  Using \eqref{eq:I1}, \eqref{eq:I2}, and \eqref{eq:I3} in \eqref{eq:int-upper}, we obtain
  \begin{equation*}
    \Theta_F^\eta(\tbf x) \lesssim \int_{\Q} G_{\tbf{x}}(\tbf{z}) (u + 1)^2 v \e^{-\sqrt{2} u} \d \tbf{z} = I_1 + I_2 + I_3 \lesssim y.
    \qedhere
  \end{equation*}
\end{proof}

\subsubsection*{Boundary estimates}
We now analyze the ``harmonic component'' of $\Theta^\eta$, namely $\bar\Theta^\eta \coloneqq \Theta^\eta + \Theta_F^\eta$.
This is a positive harmonic function on the quarter-plane.
Because $\Theta_F^\eta = 0$ on $\partial \Q$, \eqref{eq:improved} yields the following estimate on the boundary:
\begin{equation}
  \label{eq:boundary-harmonic}
  \bar\Theta^\eta(0, y) \lesssim y \And \bar \Theta^\eta(x, 0) = 0.
\end{equation}
\begin{lemma}
  \label{lem:boundary}
  There exists $A \geq 0$ such that
  \begin{equation*}
    \bar\Theta^\eta(x, y) = \big[Ax + \m{O}(1)\big]y \quad \text{on } \Q.
  \end{equation*}
\end{lemma}
\begin{proof}
  Recall the square map $\zeta \coloneq \Q \to \H$ and define $g \coloneqq \bar\Theta^\eta \circ \zeta^{-1}$, which is a positive harmonic function on the half-plane.
  By the Herglotz representation theorem, there exists $A \geq 0$ such that
  \begin{equation*}
    g(x, y) = A y + g_\partial(x, y) \quad \text{for} \quad g_{\partial}(x, y) \coloneqq \frac{y}{\pi}   \int_{\R}  \frac{g(t, 0) \d t}{(x - t)^2 + y^2}\,.
  \end{equation*}
  Composing \eqref{eq:boundary-harmonic} with the square-root $\zeta^{-1}$, we have
  \begin{equation*}
    g(x, 0) \lesssim \sqrt{x_-}.
  \end{equation*}
  In \eqref{eq:sqrt}, we showed that this implies that $g_\partial \lesssim \braket{\tbf{x}}^{1/2}$.
  Hence $g_\partial \circ \zeta \lesssim \braket{\tbf{x}}$ while
  \begin{equation*}
    g_\partial \circ \zeta = \bar\Theta^\eta \lesssim y \quad \text{on } \partial \Q.
  \end{equation*}
  By the Phragm\'en--Lindel\"of principle, we obtain $g_\partial \circ \zeta \lesssim y$ on $\Q$.
  On the other hand, $y \circ \zeta = xy$, so
  \begin{equation*}
    \bar\Theta^\eta = g \circ \zeta = Axy + g_\partial \circ \zeta = A\big[x + \m{O}(1)\big] y \quad \text{on } \Q.
    \qedhere
  \end{equation*}
\end{proof}
Combining Lemmas~\ref{lem:anharmonic} and \ref{lem:boundary}, we have shown that
\begin{equation}
  \label{eq:total}
  \Theta^\eta(x, y) = A\big[x + \m{O}(1)\big]y \quad \text{on } \Q
\end{equation}
for some $A \geq 0$.
We now translate this bound back to the region $\Lambda = \eta(\Q)$.
\begin{lemma}
  \label{lem:composition}
  If $x > \omega_+(y)$, we have $\Theta(x, y) = A\big[x - \tfrac{1}{\sqrt{2}} \log_+ \norm{\tbf{x}} + \m{O}(1)\big] y$.
\end{lemma}
\begin{proof}
  As noted earlier, $\{x > \omega_+(y)\} \setminus \Lambda$ lies a bounded distance from $\Lambda$, so Harnack allows us to reduce the problem to $\Lambda$.
  
  Recall that $\eta^{-1}(z) = z - \tfrac{1}{\sqrt{2}}\log(z + 1)$.
  Mixing real and complex notation, we write this as
  \begin{equation*}
    \eta^{-1}(x, y) = \big(x - \tfrac{1}{\sqrt{2}} \log \abs{z + 1}, y - \tfrac{1}{\sqrt{2}}\arg(z + 1)\big)
  \end{equation*}
  for $z = x + \iu y$.
  Thus \eqref{eq:total} becomes
  \begin{equation*}
    \Theta(x, y) = \big(\Theta^\eta \circ \eta^{-1}\big)(x, y) = A \big[x -  \tfrac{1}{\sqrt{2}} \log \abs{z + 1} + \m{O}(1)\big]\big[y - \tfrac{1}{\sqrt{2}}\arg(z + 1)\big].
  \end{equation*}
  Now $ \log \abs{z + 1} = \log_+ \abs{z} + \m{O}(1)$ and
  \begin{equation*}
    x\arg(z + 1)  = x \arctan\frac{y}{x + 1} < y.
  \end{equation*}
  Therefore
  \begin{equation*}
    \Theta(x, y) = A\big[x - \tfrac{1}{\sqrt{2}} \log_+ \norm{\tbf{x}} + \m{O}(1)\big] y.
    \qedhere
  \end{equation*}
\end{proof}

\subsubsection*{Matching}
It remains only to identify the nonnegative constant $A$.
\begin{proof}[Proof of Proposition~\textup{\ref{prop:tail-asymp}} and Theorem~\textup{\ref{thm:tail}}]
  By Lemma~\ref{lem:composition}, there exist $A \geq 0,$ $M > 0$, and $E \colon \H \to \R$ such that $\abs{E} \leq M$ and
  \begin{equation}
    \label{eq:quarter}
    \Phi(x, y) = A \big[x - \tfrac{1}{\sqrt{2}} \log \norm{\tbf{x}} + E(x, y)\big] y \e^{-\sqrt{2} x} \quad \text{on } \big\{x > \tfrac{1}{\sqrt{2}} \log_+ y\big\}.
  \end{equation}
  On the other hand, we recall from Theorem~\ref{thm:asymptotics} and \eqref{eq:wave-constant} that
  \begin{equation}
    \label{eq:local}
    \Phi\big(x + \tfrac{1}{\sqrt{2}}\log y, y\big) \to w_{c_*}(x) \quad \text{as } y \to \infty
  \end{equation}
  locally uniformly in $x$ and
  \begin{equation}
    \label{eq:1D-asymp}
    w_{c_*}(x) \sim K_* x \e^{-\sqrt{2} x} \quad \text{as } x \to \infty
  \end{equation}
  for some $K_* > 0$.
  Fix $\eps > 0$.
  Then by \eqref{eq:1D-asymp}, there exists $x_0(\eps) \geq \frac{M}{\eps}$ such that
  \begin{equation}
    \label{eq:1D-ratio}
    K_* (1 - \eps) x \e^{-\sqrt{2} x} \leq w_{c_*}(x) \leq K_* (1 + \eps) x \e^{-\sqrt{2} x} \ForAll x \geq x_0.
  \end{equation}
  We evaluate $\Phi$ at $x = x_0 + \tfrac{1}{\sqrt{2}} \log y$ and take $y \to \infty$.
  Before doing so, we note that $\log\norm{\tbf{x}} - \log y \to 0$ along this sequence.
  Combining \eqref{eq:quarter} and \eqref{eq:local}, we therefore find
  \begin{equation*}
    A(x_0 - M) \e^{-\sqrt{2} x_0} \leq w_{c_*}(x_0) = \lim_{y \to \infty} \Phi\big(x_0 + \tfrac{1}{\sqrt{2}}\log y, y) \leq A(x_0 + M) \e^{-\sqrt{2} x_0}.
  \end{equation*}
  Dividing by $x_0 \e^{-\sqrt{2} x_0}$ and using $x_0 \geq M/\eps$ and \eqref{eq:1D-ratio}, we obtain
  \begin{equation*}
    A(1 - \eps) \leq K_*(1 + \eps) \And K_*(1 - \eps) \leq A(1 + \eps).
  \end{equation*}
  That is,
  \begin{equation*}
    \frac{1 - \eps}{1 + \eps} \leq \frac{A}{K_*} \leq \frac{1 + \eps}{1 - \eps}.
  \end{equation*}
  Since $\eps > 0$ is arbitrary, we have $A = K_*$ as desired.
\end{proof}

\printbibliography
\end{document}